\documentclass[10pt]{amsart}            

\usepackage{amssymb,amsmath,amsthm}
\usepackage{hyperref}
\usepackage{color}
\usepackage{mathrsfs}
\usepackage{amsaddr}

\makeatletter
\def\blfootnote{\gdef\@thefnmark{}\@footnotetext}
\makeatother

\renewcommand\mathcal{\mathscr}

\usepackage{ulem}
\renewcommand{\emph}{\normalem}

\theoremstyle{plain}
\newtheorem{theorem}{Theorem}[section]
\newtheorem*{theorem*}{Theorem}
\newtheorem{lemma}[theorem]{Lemma}
\newtheorem*{lemma*}{Lemma}

\newtheorem{proposition}[theorem]{Proposition}

\theoremstyle{remark}
\newtheorem{remark}[theorem]{Remark}
\newtheorem*{remark*}{Remark}

\theoremstyle{definition}
\newtheorem{definition}[theorem]{Definition}
\newtheorem*{definition*}{Definition}

\newtheorem{Basic assumptions}[theorem]{Basic assumptions}

\numberwithin{equation}{section}

\newcount\quantno
\everydisplay{\quantno=0}\everycr{\quantno=0}
\newcommand\quant{\advance\quantno by1
                      \ifnum\quantno=1\qquad\else\quad\fi\forall }
\newcommand\itemno[1]{(\romannumeral #1)}

\renewcommand\Re{\operatorname{\mathrm{Re}}}

\newcommand\rest[1]{\kern-.1em
          \lower.5ex\hbox{$\scriptstyle #1$}\kern.05em}

\renewcommand\mod[1]{\vert{#1}\vert}
\newcommand\bigmod[1]{\bigl\vert{#1}\bigr|}
\newcommand\Bigmod[1]{\Bigl\vert{#1}\Bigr|}

\newcommand\upnorm[2]{{\Vert{#1}\Vert_{(#2)}}}
\newcommand\upbignorm[2]{{\big\Vert{#1}\big\Vert_{(#2)}}}
\newcommand\norm[2]{{\Vert{#1}\Vert_{#2}}}

\newcommand\bignorm[2]{\left.{\bigl\Vert{#1}\bigr\Vert_{#2}}\right.}
\newcommand\bignormto[3]{\left.{\bigl\Vert{#1}\bigr\Vert_{#2}^{#3}}\right.}
\newcommand\Bignorm[2]{\left.{\Bigl\Vert{#1}\Bigr\Vert_{#2}}\right.}

\newcommand\bigopnorm[2]{\big|\!\big|\!\big| {#1} \big|\!\big|\!\big|_{#2}}
\newcommand\bigopnormto[3]{\big|\!\big|\!\big| {#1} \big|\!\big|\!\big|_{#2}^{#3}}

\newcommand\prodo[2]{\left\langle#1,#2\right\rangle}

\newcommand\smallfrac[2]{\mbox{\small$\displaystyle\frac{#1}{#2}$}}
\newcommand\wrt{\,\text{\rm d}}

\newcommand\bS{\mathbf{S}}

\newcommand\BC{\mathbb{C}}

\newcommand\BR{\mathbb{R}}

\newcommand\cA{\mathcal{A}}   
  
\newcommand\cB{\mathcal{B}}

\newcommand\cD{\mathcal{D}}  
 
\newcommand\cE{\mathcal{E}}

\newcommand\cG{\mathcal{G}}   
 
\newcommand\cH{\mathcal{H}}   
\newcommand\frh{\mathfrak{h}} 
\newcommand\cI{\mathcal{I}}   
 
\newcommand\cJ{\mathcal{J}}
\newcommand\cK{\mathcal{K}}    
\newcommand\cL{\mathcal{L}}    
     
\newcommand\cN{\mathcal{N}}    
\newcommand\cO{\mathcal{O}}    
\newcommand\cP{\mathcal{P}}  \newcommand\fP{\mathfrak{P}}   
\newcommand\cQ{\mathcal{Q}}
\newcommand\cR{\mathcal{R}}

\newcommand\cV{\mathcal{V}}

\newcommand\cY{\mathcal{Y}}

\newcommand\al{\alpha}
\newcommand\be{\beta}
\newcommand\ga{\gamma}    \newcommand\Ga{\Gamma}
\newcommand\de{\delta}
  \newcommand\vep{\varepsilon}
\newcommand\ka{{\kappa}}
\newcommand\la{\lambda}   
\newcommand\om{\omega}    \newcommand\Om{\Omega}  
\newcommand\si{\sigma}    \newcommand\Si{\Sigma}
\newcommand\te{\theta}
\newcommand\vp{\varphi}

\newcommand\vr{\varrho}

\newcommand\OV{\overline}
\newcommand\funnyk{k\hbox to 0pt{\hss\phantom{g}}}

\newcommand\lu[1]{L^1(#1)}

\newcommand\lp[1]{L^p(#1)}
\newcommand\laq[1]{L^q(#1)}
\newcommand\ld[1]{L^2(#1)}

\newcommand\lr[1]{L^r(#1)}
\newcommand\ly[1]{L^\infty(#1)}

\newcommand\hu[1]{H^1(#1)}

\newcommand\huRtau[1]{H_{\cR_\tau}^1(#1)}

\newcommand\ghu[1]{{\frh}^1(#1)}

\newcommand\ghuH[1]{{\frh}_{\cH}^1(#1)}
\newcommand\ghuP[1]{{\frh}_{\cP}^1(#1)}
\newcommand\ghuat[1]{{\frh}_{\mathrm{at}}^1(#1)}
\newcommand\ghumax[1]{{\frh}_{\mathrm{max}}^1(#1)}
\newcommand\ghuRtau[1]{{\frh}_{\cR_\tau}^1(#1)}

\newcommand\ghuI[1]{{\frh}_{\mathrm{I}}^1(#1)}

\newcommand\wh{\widehat}
\newcommand\wt{\widetilde}
\newcommand\whH{\widehat{\phantom{G}}\hbox to 0pt{\hss $H$}}

\newcommand\emspace{\hbox to 6pt{\hss}}
\newcommand\ds{\displaystyle}

\newcommand\rmi{\hbox{\rm (i)}}
\newcommand\rmii{\hbox{\rm (ii)}}
\newcommand\rmiii{\hbox{\rm (iii)}}
\newcommand\rmiv{\hbox{\rm (iv)}}
\newcommand\rmv{\hbox{\rm (v)}}

\newcommand\ir{\int_{-\infty}^{\infty}}
\newcommand\ioty{\int_0^{\infty}}
\newcommand\dtt[1]{\,\frac{\mathrm {d} #1}{ #1}}

\newcommand\One{{\mathbf{1}}}

\newcommand\e{\mathrm{e}}

\newcommand\Ric{\mathop{\rm Ric}}

\newcommand\cRtau{\cR_{\tau}}

\newcommand\cRtauO{\cR_{\tau}^0}

\newcommand\cRtauOO{\cR_{\tau}^{0,0}}

\newcommand\cRtauinfty{\cR_{\tau}^{\infty}}

\newcommand\cRtauOinfty{\cR_{\tau}^{0,\infty}}

\newcommand\cJOOtau{\cJ_\tau^{0,0}}

\newcommand\cKO{\cJ_\tau^0}
\newcommand\cKOO{\cJ_\tau^{0,0}}
\newcommand\cKOinfty{\cJ_\tau^{0,\infty}}
\newcommand\cKinfty{\cJ_\tau^\infty}

\newcommand\Sisi{\Si}
\newcommand\Sisieta{\Si_{\eta}}
\newcommand\cPS{\cP^{{0}}}

\newcommand\cPSeta{\cP^{\eta}}
\newcommand\cPSvepk{\cP^{\vep_k}}
\newcommand\cGS{\cG_{\Si}}

\newcommand\Msi{M_{\si} }
\newcommand\Msieta{M_{\si}^{\eta} }
\newcommand\Mt{M_{t} }
\newcommand\Mteta{M_{t}^{\eta} }

\newcommand\dist{\mathrm{dist}}

\newcommand\Nabla{\nabla\!\!\!\!\nabla}
\newcommand\DDelta{\Delta\!\!\!\!\Delta}

\DeclareSymbolFont{EUEX}{U}{euex}{m}{n}

\DeclareSymbolFont{euexlargesymbols}{U}{euex}{m}{n}
\DeclareMathSymbol{\intop}{\mathop}{euexlargesymbols}{"52}
     \def\int{\intop\nolimits}

\DeclareSymbolFont{euexsymbols}     {U}{euex}{m}{n}
\DeclareMathSymbol{\smallint}{\mathop}{euexsymbols}{"52}

\begin{document}

\title[Riesz--Hardy spaces]{Local Riesz transform and local Hardy spaces \\
on Riemannian manifolds with bounded geometry}

\subjclass[2010]{42B30, 42B35, 58C99} 

\keywords{Local Hardy space, local Riesz transform, bounded geometry, locally doubling manifolds,
potential analysis on strips.}


\author[S. Meda and G. Veronelli]
{Stefano Meda and Giona Veronelli}

\address{ 
Dipartimento di Matematica e Applicazioni - Universit\`a di Milano-Bicocca\\
via R.~Cozzi 55, I-20125 Milano, Italy
\break 
Stefano Meda: stefano.meda@unimib.it\\
Giona Veronelli:
giona.veronelli@unimib.it}

\begin{abstract}
We prove that if $\tau$ is a large positive number, then the atomic Goldberg-type space $\mathfrak{h}^1(N)$ and the space $\mathfrak{h}_{\mathcal R_\tau}^1(N)$ of all integrable functions on $N$ whose \textit{local Riesz transform} $\mathcal R_\tau$ is integrable are the same space on any complete noncompact Riemannian manifold $N$ with Ricci curvature bounded from below and positive injectivity radius. We also relate $\mathfrak{h}^1(N)$ to a space of harmonic functions on the slice $N\times (0,\delta)$ for $\delta>0$ small enough.  
\end{abstract}

\maketitle


\setcounter{section}{0}
\section{Introduction} \label{s:Introduction}

The classical Hardy space $\hu{\BR^n}$ plays an important role in Euclidean Harmonic Analysis and has been
the object of a huge number of investigations.  Its theory, which is available also in book form (see, for instance,
\cite{St2,Gr}), is well understood, and has its roots in the 
seminal papers \cite{FS,SW}.  In the first, C.~Fefferman and E.M.~Stein proved, amongst other important results,
that $\hu{\BR^n}$ can be equivalently defined in terms of the Riesz transforms, of
various kinds of maximal operators and square functions.  In the second, Stein and G.~Weiss considered a space
of generalised conjugate harmonic functions that may be identified with $\hu{\BR^n}$.  
Their results were complemented by R.R.~Coifman \cite{Coi} and 
R.~Latter \cite{La}, who proved that $\hu{\BR^n}$ admits an atomic decomposition.  All these characterisations
corroborate the idea that the space $\hu{\BR^n}$ is central in Harmonic Analysis
and illustrate its flexibility, a feature of great importance in the applications.  This beautiful theory, or part of it,
has been extended in various directions: see, for instance, \cite{CW,FoS,CG,MPR,FK,DZ1,DZ2,DJ,GLY,Di,AMR,To,HLMMY,AM,ADH}
and the references therein.  We observe in passing that on certain 
examples of nondoubling measure spaces, such as the hyperbolic disc, a perhaps surprising phenomenon occurs:
the Hardy-type spaces defined in terms 
of the Riesz transform, the Poisson maximal operator and the heat maximal operator are different spaces
\cite{MaMVV}.  For more on the attempts to define an effective Hardy-type space on noncompact symmetric
spaces and generalisations thereof, see \cite{A1,I,Lo,CMM1,MMV2,MMV3,MaMV1} and the references therein.  

The major drawback of $\hu{\BR^n}$ is that it is not ``stable'' under localisation, i.e.,
multiplication by smooth functions of compact support does not preserve $\hu{\BR^n}$. 
This fact induced D.~Goldberg \cite{G} to introduce a variant of $\hu{\BR^n}$, denoted by $\ghu{\BR^n}$
and quite often termed ``local Hardy space''.  It is fair to recall that R.S.~Strichartz \cite{Str} 
had defined a suggestive predecessor of $\ghu{\BR^n}$ on any compact Riemannian manifold.  
Goldberg proved several characterisations of $\ghu{\BR^n}$, which are the natural ``local'' counterparts of many of
those known for $\hu{\BR^n}$.  They include characterisations via several different maximal operators, 
local Riesz transforms, and an atomic decomposition.  It includes also a characterisation of $\ghu{\BR^n}$
in terms of a generalised system of conjugate harmonic functions on the slice $\BR^n\times (0,1)$.    

A careful reading of \cite{G} reveals that most of the properties of $\ghu{\BR^n}$ depend only on the local
structure of the Euclidean space and not on its geometry at infinity.  Thus, it is natural to 
speculate whether one can define an analogue of $\ghu{\BR^n}$ on ``locally Euclidean spaces''.  

Interesting examples of such spaces are the so-called RD-spaces, i.e., homogeneous spaces $X$ in
the sense of Coifman and Weiss with the additional property that a reverse doubling condition holds in $X$. 
Following up previous works of various authors \cite{DZ1,DZ2,GLY}, Dachun Yang and Yuan Zhou \cite{YZ1,YZ2} constructed 
on such spaces an interesting and quite complete theory of ``local Hardy spaces'' associated to given admissible functions.
See also \cite{HMY} for results concerning Triebel--Lizorkin spaces on RD-spaces and their relationships with local Hardy
spaces. In particular, note that if $N$ is an RD-space, then the local Hardy space $\ghu{N}$ defined below 
reduces to the space $H_\ell^{1,2} (N)$ of \cite{YZ1}.

Further important examples of ``locally Euclidean spaces'' are Riemannian manifolds.  
A subclass thereof on which a satisfactory theory of local Hardy spaces can be developed
is that of manifolds $N$ with \textit{bounded geometry}.  By this we mean that $N$ is a
complete connected noncompact Riemannian manifolds with Ricci curvature bounded from below and positive injectivity radius.
Notice that the Riemannian measure on $N$ may very well be nondoubling.  
In analogy with the classical case \cite{G}, one can define a number of spaces on $N$, 
including $\ghumax{N}$, $\ghuH{N}$, $\ghuP{N}$, $\ghuI{N}$, $\ghuat{N}$: specifically, 
$\ghumax{N}$, $\ghuH{N}$, $\ghuP{N}$
are defined in terms of maximal functions (associated to a suitable grand maximal operator, to the local heat maximal operator
and to the local Poisson maximal operator, respectively), $\ghuI{N}$ and $\ghuat{N}$ are ionic and atomic spaces,
respectively.  It may be worth observing that $\ghuI{N}$ can be equivalently defined using various kinds of ions, 
and similarly for $\ghuat{N}$, but with atoms playing the role of ions.  
For the sake of simplicity, we do not insist on this point in the introduction.  

In the inspiring paper \cite{T}, M.~Taylor proved, under the additional assumption that 
all the derivatives of the metric tensor are bounded, that $\ghumax{N} = \ghuI{N}$.  
In a much wider context that includes Riemannian manifolds with bounded geometry in the sense specified above, 
Meda and S.~Volpi \cite{MVo} introduced the space $\ghuat{N}$, and proved that $\ghuat{N} = \ghuI{N}$.  
It is fair to say that both \cite{T} and \cite{MVo} contain many additional material, including duality and
interpolation results and boundedness criteria for relevant (pseudo-) differential operators on $N$.  
Quite recently, A.~Martini, Meda and M.~Vallarino \cite{MaMV2}, following up a profound result of A.~Uchiyama \cite{U}, 
showed that if $N$ has bounded geometry, then $\ghumax{N} = \ghuH{N} = \ghuP{N} = \ghuI{N}$ (see also 
\cite{YZ1,YZ2}  for related results in the setting of RD-spaces).  
Consequently, the five spaces listed above coincide (and their norms are equivalent); for the sake of
brevity, we denote simply by $\ghu{N}$ the resulting space, equipped with any of the corresponding norms.  

A further natural local Hardy space on $N$ may be defined as follows.  
Denote by $\nabla$ the covariant derivative on $N$,
and by $\cL$ (minus) the Laplace--Beltrami operator, which we think of as an unbounded nonnegative operator on $\ld{N}$.  
For each positive number~$\tau$, denote by $\cL_\tau$ the translated Laplacian $\tau\cI+\cL$.  
We consider the \textit{translated Riesz transform} $\cRtau := \nabla \cL_\tau^{-1/2}$, $\tau>0$, 
and the \textit{Riesz--Goldberg} space
\begin{equation} \label{f: ghuBRn}
\ghuRtau{N}
:= \big\{f \in \lu{N}: \mod{\cRtau f} \in \lu{N}\big\}.  
\end{equation} 
We equip $\ghuRtau{N}$ with the norm $\ds \bignorm{f}{\ghuRtau{N}} := \bignorm{f}{1} + \bignorm{\mod{\cR_\tau f}}{1}$.  
E.~Russ \cite[proof of Theorem 14]{Ru} (see also \cite[Theorem 8]{MVo}) proved that if $\tau$ is large enough, then 
$\ghuRtau{N} \supseteq \ghu{N}$ on a class of Riemannian manifolds that include those of bounded geometry. 
It is then natural to speculate whether $\ghuRtau{N}$ agrees with $\ghu{N}$ in this generality, thereby extending 
the result for $\BR^n$ proved by Goldberg via Fourier transform techniques.  We remark that the 
Riesz transform $\nabla\cL^{-1/2}$ (which corresponds to the limiting case where $\tau=0$)
is unbounded from $\ghu{\BR^n}$ to $\lu{\BR^n}$.  

In this paper, we answer to this deceptively simple question in the affirmative.  
Our main result, Theorem~\ref{t: char local Riesz}, states that if 
$N$ is a complete connected noncompact Riemannian manifold with bounded geometry, 
then $\ghuRtau{N}=\ghu{N}$ as long as $\tau$ is large enough.  
Our strategy of proof has its roots in an old and beautiful idea of 
Stein and Weiss (see, in particular, \cite[Theorem~A]{SW}), who realised that certain powers (slightly below $1$)
of the gradient of harmonic functions are subharmonic.  This idea is central in the classical proof that 
if $u$ is a harmonic function on $\BR_+^{n+1} := \big\{(x,t)\in \BR^n\times (0,\infty): t>0\big\}$, then 
\begin{equation} \label{f: equivalence harmonic}
\bignorm{\partial_t u_{\vert_{{\BR^n}\times\{0\}}}}{\hu{\BR^n}}
\asymp \, \bignorm{\mod{\Nabla u}^*}{\lu{\BR^n}}
\asymp \, \sup_{t>0}\,  \int_{\BR^n} \, \bigmod{\Nabla u(x,t)} \wrt x,  
\end{equation}
where $\Nabla$ denotes the gradient on $\BR^{n+1}$ and the superscript $*$ stands
for nontangential maximal function (see, for instance, \cite[Ch.~VII]{St1}).  This result has 
a natural counterpart for $\ghu{\BR^n}$ \cite{G}, where the slice $\BR^n\times (0,1)$
plays the role of $\BR_+^{n+1}$ in the classical case.    

There is a major problem in extending the latter result to Riemannian manifolds: if the curvature of $N$
is not nonnegative, then powers ($\leq 1$) of the gradient of harmonic functions on $N\times \BR$ may not be subharmonic.   
M.~Dindo\v s \cite[Chapter~6]{Di} was able to overcome this problem and 
to work out an effective strategy (modifying significantly that of Stein and Weiss) to prove an analogue
of \eqref{f: equivalence harmonic} on bounded domains of (compact) manifolds, endowed with a possibly nonsmooth metric.  
His strategy hinges on the observation, derived from the Bochner--Weitzenbock formula,
that if~$u$ is a harmonic function on an open set of an $(n+1)$-dimensional Riemannian manifold $M$ 
with Ricci curvature bounded from below by~$-\kappa^2$
and $(n-1)/n < q \leq 1$, then $\bigmod{\nabla u}^q$ is $q\kappa^2$-subharmonic, i.e., 
it satisfies an inequality of the form $\cL \mod{\nabla u}^q \leq q\kappa^2\, \mod{\nabla u}^q$.  

We adapt Goldberg's approach and extend Dindo\v s' strategy 
to the case of noncompact Riemannian manifolds $N$ of bounded geometry.
We consider the slice $\Si :=N\times (0,2\si)$, and  
prove that if $\si$ is small enough (see \eqref{f: sigma}), then a harmonic function in $\Si$
satisfies the maximal inequality $\ds \int_N \, \bigmod{\Nabla u(x)}^* \wrt\nu(x) <\infty$
if and only if $\ds \sup_{t\in (0,2\si)} \int_N \, \bigmod{\Nabla u(x,t)} \wrt\nu(x) <\infty$ and $\mod{\Nabla u}$
tends to $0$ at infinity, uniformly in each closed subslice of $\Si$ (see Theorem~\ref{t: analogue SW D}).  
Here $\nu$, $\Nabla$ and~${}^*$ denote the Riemannian density, the gradient of $N\times\BR$ and 
an appropriate nontangential maximal function (defined at the beginning of Section~\ref{s: Maximal}), 
respectively. 
Our strategy requires estimating the Poisson operator and powers of the Green operator associated to $\Si$.  
In particular, we show that if $\si$ is small enough, then the 
integral kernels of such operators are ``integrable at infinity in $\Si$'' (see Sections~\ref{s: Poisson} 
and \ref{s: Green}).  Their rate of decay at infinity is controlled 
by $\la_1 :=\pi/(2\si)$.  Notice that $-\la_1^2$ is the first eigenvalue of the Dirichlet Laplacian on the interval $[0,2\si]$.  
Clearly~$\la_1$ increases as~$\si$ decreases: this is the reason for which we choose $\si$ small.  

The last ingredient we need in the proof our main result is
a careful analysis of the kernel of the translated Riesz transform $\cR_\tau$.  This technical part
is confined in Section~\ref{s: The local Riesz transform}.   

The paper is organized as follows.  Section~\ref{s: Background material} contains some preliminary estimates
extensively used in the sequel.  
In Sections~\ref{s: Poisson} and \ref{s: Green} we establish some potential estimates on $\Si$.
Section~\ref{s: Maximal} contains some maximal estimates for certain potentials on $\Si$.
Section~\ref{s: G} is devoted to the analogue on slices of $\Si$ of certain results of Dindo\v s \cite{Di}.
The analysis of the local Riesz transform for Riemannian manifolds with 
bounded geometry, together with some basic information concerning the Goldberg-type space $\ghu{N}$, 
is contained in Section~\ref{s: The local Riesz transform}, where our main result concerning $\huRtau{N}$, 
Theorem~\ref{t: char local Riesz}, is proved.

We shall use the ``variable constant convention'', and denote by $C,$
possibly with sub- or superscripts, a constant that may vary from place to 
place and may depend on any factor quantified (implicitly or explicitly) 
before its occurrence, but not on factors quantified afterwards. 

Throughout the paper, given $p$ in $[1,\infty]$, we denote by $p'$ the conjugate exponent of $p$.

\section{Background material and preliminary estimates}
\label{s: Background material}

\subsection{Standing assumptions}
In this paper $N$ will always denote an $n$-dimensional complete connected noncompact 
Riemannian manifold with \textbf{bounded geometry}.  By this we mean that the Ricci curvature of $N$
satisfies ${\Ric_{N}} \geq - \kappa^2$ for some nonnegative number $\kappa$,  
and the injectivity radius is strictly positive.  
The Riemannian measure of $N$ will be denoted by $\nu$.
The operator norm of a bounded linear operator $T$ from $\lp{N}$ to $\laq{N}$ will be 
denoted by $\bigopnorm{T}{p;q}$. 

Denote by $\nabla$ and $\Delta$ the gradient and the (negative) Laplace--Beltrami operator on~$N$, respectively.
Set $\cL = -\Delta$.  The operator $\cL$, initially defined on smooth functions with compact support,
admits a unique self adjoint extension, still denoted by $\cL$, in $\ld{N}$.  For any nonnegative number $\tau$
denote by $\cL_\tau$ the operator $\tau\cI+\cL$, where $\cI$ denotes the identity operator.  In particular, 
$\cL_0 = \cL$. 
Denote by $\cH_t^N$ and $h_t^N$ the heat semigroup $\e^{-t\cL}$ and the corresponding heat kernel, respectively.  
The following are well known consequences of our assumptions:
\begin{enumerate}
	\item[\itemno1]
		$N$ is \textit{locally Ahlfors regular}. Indeed, by Bishop-Gromov's volume comparison theorems 
		and by a well known estimate due to C.B.~Croke \cite[Prop. 14]{Cr}, for each $R>0$
		there exist positive constants $C_1$ and $C_2$ such that
		\begin{equation}\label{f: Ahlfors}
		C_1\, r^n \leq \nu\big(B_r(x)\big) \leq C_2 \, r^n
		\quant x \in N \quant r \in (0,R].   
		\end{equation} 
		Thus, in particular, $\nu$ is \textit{locally doubling}.  
		Furthermore, there exist nonnegative constants $\al$ and $\be$ and $C$ such that 
		\begin{equation} \label{f: volume growth}
			\nu\big(B_r(x)\big) \leq C \, r^\al\, \e^{2\be r}
			\quant x \in N \quant r \in [1,\infty);  
		\end{equation}
	\item[\itemno2]
		there exist positive constants $c$ and $C$ such that  
		\begin{equation} \label{f: assumptions on ht}
		h_t^N(x,y) \leq C\, \ga(t) \,  \e^{-c d(x,y)^2/t}
			\quant x,y\in N \quant t>0,
		\end{equation} 
		where $\ga(t) := \max(t^{-n/2},t^{-1/2})$ (see, for instance, \cite[Theorem~3]{CF}).  
		Note that \eqref{f: assumptions on ht} directly implies the following \textit{ultracontractivity
		estimate} for $\cH_t^N$:
		\begin{equation} \label{f: ultracontractivity Ht}
		\bigopnorm{\cH_t^N}{1;\infty} \leq C \, \ga(t)  \quant t >0.
		\end{equation}
		Suppose now that  $1\leq q\leq r\leq \infty$.  Then there exists a constant $C$ such that 
		\begin{equation} \label{f: Ht LqLr}
		\bigopnorm{\cH_t^N}{q;r} \leq C \, \ga(t)^{1/q-1/r} \quant t >0.
		\end{equation}
		The estimate \eqref{f: Ht LqLr} follows from \eqref{f: ultracontractivity Ht},  
		the contractivity of $\cH_t^N$ on $\lp{N}$ for all $p$ 
		in $[1, \infty]$, duality and interpolation; 
	\item[\itemno3] 
		\textit{Bakry's condition} \cite{Ba}  
		\begin{equation} \label{f: Bakry}
		\bigmod{\nabla\cH_t^N f}
			\leq \e^{\kappa^2 t} \, \cH_{t}^N\big(\bigmod{\nabla f}\big)
		\quant t>0
		\end{equation} 
		holds.
\end{enumerate}

\subsection{Ultracontractivity estimates for generalised Bessel potentials}
Propositions~\ref{p: consequences GBE}, \ref{p: consequences GBE II} and \ref{p: consequences GBE III} 
contain some basic estimates for certain (families of) operators that will arise frequently in the sequel.
It is convenient to set $\cD := \sqrt{\cL}$. 

\begin{proposition} \label{p: consequences GBE}
For any pair of numbers $\tau\geq \kappa^2$ and $\rho>0$ 
$$
\bigmod{\nabla (\tau\cI+t^2\cD^2)^{-\rho} f} 
\leq \big((\tau-\kappa^2)\cI+t^2\cD^2\big)^{-\rho}\big(\bigmod{\nabla f}\big)
\quant t \in (0,1].
$$
\end{proposition}

\begin{proof}
The subordination formula 
	\begin{equation} \label{f: subordination}
(\tau\cI+t^2\cD^2)^{-\rho} f
= \frac{1}{\Ga(\rho)} \, \ioty s^\rho \, \e^{-\tau s} \, \cH_{st^2}f {\dtt s}   
	\end{equation}
and Bakry's condition \eqref{f: Bakry} imply that 
$$
\begin{aligned}
	\bigmod{\nabla (\tau\cI+t^2\cD^2)^{-\rho} f}
	& \leq \frac{1}{\Ga(\rho)} \, \ioty s^\rho \, \e^{-\tau s} \, \bigmod{\nabla \cH_{st^2}f} {\dtt s} \\
	& \leq \frac{1}{\Ga(\rho)} \, \ioty s^\rho \, \e^{-(\tau-\kappa^2 t^2) s} \, \cH_{st^2}\bigmod{\nabla f} {\dtt s} \\
	& \leq \big((\tau-\kappa^2)\cI+t^2\cD^2\big)^{-\rho} \bigmod{\nabla f},
\end{aligned}
$$
as required.   
\end{proof}

\noindent  
Part of the proof of the next proposition is an adaptation of the proof of \cite[Proposition~2.2~\rmi]{MMV0}. 
Given a nonnegative number $\rho$ and a function $G:[0,\infty)\to \BC$, set
$$
\upbignorm{G}{\rho}
:= \sup_{\la\geq 0} \, (1+\la^2)^\rho\, \bigmod{G(\la)}
\quad\hbox{and}\quad
\Xi_\rho(G)
:= \sqrt{\upnorm{G}{\rho}\,\upnorm{G}{\rho+1}}.  
$$
In the next proposition $F$ and $\{F_t:t>0\}$ will denote functions on $[0,\infty)$.  It is straightforward
to check that if $F(\cD)$ is bounded from $\lu{N}$ to $\ld{N}$, then 
$F(\cD)$ is also bounded from $\ld{N}$ to $\ly{N}$, and $\bigopnorm{F(\cD)}{1;2} = \bigopnorm{F(\cD)}{2;\infty}$.
We shall use this observation without any further comment.  

\begin{proposition} \label{p: consequences GBE II}
There exists a positive constant $C$ such that the following hold for every $t>0$:
\begin{enumerate}
\item[\itemno1]
	if $1\leq q \leq r \leq \infty$, $\rho > n(1/q-1/r)/2$ and $\tau>0$, then
	$$
	\bigopnorm{(\tau\cI+t^2\cD^2)^{-\rho}}{q;r} \leq C \, \ga(t)^{2(1/q-1/r)};
	$$  
\item[\itemno2]
	if $\rho > n/4$, then 
	$\ds
	\bigopnorm{F(t\cD)}{1;2} = \bigopnorm{F(t\cD)}{2;\infty}   
	\leq C \, \ga(t)\, \upnorm{F}{\rho}  
	$
	and
	$\ds
	\bigopnorm{F(t\cD)}{1;\infty} 
	\leq C \, \ga(t)^2\, \upnorm{F}{2\rho};
	$
\item[\itemno3]
	if $\rho > n/4$, then     
	$\ds
	\bigopnorm{F_t(\cD)}{1;2} = \bigopnorm{F_t(\cD)}{2;\infty} 
	\leq C \, \upnorm{F_t}{\rho} 
	$
	and
	$\ds
	\bigopnorm{F_t(\cD)}{1;\infty} 
	\leq C \, \upnorm{F_t}{2\rho};
	$
\item[\itemno4]
	if $\rho>n/4$, then 
	$\ds
	\bignorm{\mod{\nabla F(t\cD) g}}{2} 
	\leq C\, t^{-1} \, \ga(t) \, \Xi_\rho\big(F\big) \bignorm{g}{1}
	$
	and
	$\ds
	\bignorm{\mod{\nabla F(t\cD) g}}{\infty} 
	\leq C \, \Xi_\rho\big(F(t\cdot)\big) \bignorm{g}{2};
	$
\item[\itemno5]
	if $\rho>n/4$, then $\bignorm{\mod{\nabla F_t(\cD) g}}{2} \leq C\,  \Xi_\rho\big(F_t\big)  \bignorm{g}{1}$ and
	$\bignorm{\mod{\nabla F_t(\cD) g}}{\infty} \leq C\,  \Xi_\rho\big(F_t\big)  \bignorm{g}{2}$.
\end{enumerate}
\end{proposition}

\begin{proof}
First we prove \rmi.  By \eqref{f: subordination} and the ultracontractivity estimate \eqref{f: Ht LqLr},
$$
\begin{aligned} 
\bigopnorm{(\tau\cI+t^2\cD^2)^{-\rho}}{q;r}
& \leq    \smallfrac{1}{\bigmod{\Ga(\rho)}} \,\,  \ioty s^{\rho}
         \, \e^{-\tau s} \, \bigopnorm{\e^{-st^2\cD^2}}{q;r} {\dtt s} \\
& \leq C \ioty  s^{\rho} \, \e^{-\tau s} \, \ga(st^2)^{1/q-1/r} {\dtt s}.
\end{aligned}
$$
The last integral is convergent because of the assumption $\rho>n(1/q-1/r)/2$.  
Now we write the last integral as the sum of the integrals over $(0,1/t^2)$ and $(1/t^2,\infty)$ and
observe that $\ga(st^2) = (st^2)^{-n/2}$ on $(0,1/t^2)$ and 
that $\ga(st^2) = (st^2)^{-1/2}$ on $(1/t^2,\infty)$.  It is straightforward to check that 
$$ 
\int_0^{1/t^2}  s^{\rho} \, \e^{-\tau s} \, (st^2)^{-n(1/q-1/r)/2} {\dtt s}
\leq C \min(t^{-n(1/q-1/r)}, t^{-2\rho})  
$$  
and that 
$\ds
\int_{1/t^2}^\infty  s^{\rho} \, \e^{-\tau s} \, (st^2)^{-(1/q-1/r)/2} {\dtt s}
\leq C \, \min(\e^{-(\tau-\vep)/t^2}, t^{-(1/q-1/r)})
$  
for $\vep$ small.  By combining the estimates above we get the required result. 

Next we prove \rmii.  By the spectral theorem 
$$
\sup_{t>0} \,
\bigopnorm{(\cI+t^2\cD^2)^{\rho} \, F(t\cD)}{2}
= \sup_{\la\geq 0} \, (1+\la^2)^{\rho} \, \mod{F(\la)}
= \upnorm{F}{\rho} < \infty.
$$
Thus, \rmi\ (with $q=1$ and $r=2$) yields  
\begin{equation} \label{1;2}  
\begin{aligned} 
\bigopnorm{F(t\cD)}{1;2}
& =    \bigopnorm{(\cI+t^2\cD^2)^{-\rho} \, (\cI+t^2\cD^2)^{\rho} \, F(t\cD)}{1;2} \\
& \leq \bigopnorm{(\cI+t^2\cD^2)^{\rho} \, F(t\cD)}{2} \, \bigopnorm{(\cI+t^2\cD^2)^{-\rho}}{1;2} \\
& \leq C \,\ga(t)
\quant t >0.  
\end{aligned}
\end{equation}
Furthermore, 
\begin{align} \label{1;infty}   
\bigopnorm{F(t\cD)}{1;\infty}
& =    \bigopnorm{(\cI+t^2\cD^2)^{-\rho} \, (\cI+t^2\cD^2)^{2\rho} \, F(t\cD)(\cI+t^2\cD^2)^{-\rho}}{1;\infty} \nonumber\\
& \leq \bigopnorm{(\cI+t^2\cD^2)^{-\rho}}{2;\infty}\, \bigopnorm{(\cI+t^2\cD^2)^{2\rho} \, F(t\cD)}{2} 
	\, \bigopnorm{(\cI+t^2\cD^2)^{-\rho}}{1;2} \nonumber\\
& \leq C \, \ga(t)^2\, \upnorm{F}{2\rho}
\quant t >0.  
\end{align}

Next we prove \rmiii.  We argue much as in the proof of \rmii, but with a slight difference.  Instead
of composing $F(t\cD)$ with $\big(\cI+t^2\cD^2)^\rho$, as in \rmii, we write 
$$
F_t(\cD)
= (\cI+\cD^2)^{-\rho} \, (\cI+\cD^2)^{\rho} \, F_t(\cD),
$$
and then proceed as above, using the estimate $\bigopnorm{(\cI+\cD^2)^{\rho} \, F_t(\cD)}{2}=\upnorm{F_t}{\rho}$,
which follows from the spectral theorem.  We omit the details.  

To prove \rmiv, observe that, by the Green formula (see, for instance, \cite[Lemma~4.4]{Gri}, together with 
\cite[Theorem~3.1]{He}),
$$
\bignormto{\mod{\nabla F(t\cD)g}}{2}{2}
= \big(\cL F(t\cD)g,  F(t\cD)g\big)  
= \frac{1}{t^2} \, \big(F_1(t\cD)g,  F(t\cD)g\big),
$$
where $F_1(z) := z^2\, F(z)$.  Schwarz's inequality then implies that 
\begin{equation} \label{f: gradient intermediate}
\begin{aligned}
\bignormto{\mod{\nabla F(t\cD)g}}{2}{2}
& \leq  \frac{1}{t^2}  \bignorm{F_1(t\cD)g}{2}\,  \bignorm{F(t\cD)g}{2},
\end{aligned}
\end{equation}
By \rmii, applied to $F$, and a similar estimate applied to $F_1$, we see that 
$\bignorm{F(t\cD)g}{2}\leq C \,\ga(t) \, \upnorm{F}{\rho} \bignorm{g}{1}$ and that 
$\bignorm{F_1(t\cD)g}{2}\leq C \,\ga(t) \, \upnorm{F_1}{\rho} \bignorm{g}{1}$.  
By combining the estimates above and the trivial observation that $\upnorm{F_1}{\rho} \leq \upnorm{F}{\rho+1}$ , we obtain that 
$$
\bignorm{\mod{\nabla F(t\cD)g}}{2}
\leq C\, t^{-1} \, \ga(t) \, \Xi_\rho(F) \bignorm{g}{1}
\quant t>0,
$$
as required.   

It remains to prove the second gradient estimate.  For $\tau >\kappa^2$ we write
$$
\begin{aligned}
\bigmod{\nabla F(t\cD)g}
& =    \bigmod{\nabla(\tau\cI+\cD^2)^{-\rho} \, (\tau\cI+\cD^2)^{\rho} \, F(t\cD)g} \\
& \leq C\, \big((\tau-\kappa^2)\cI+\cD^2\big)^{-\rho} \bigmod{\nabla(\tau\cI+\cD^2)^{\rho} \, F(t\cD)g};
\end{aligned}
$$
we have used Proposition~\ref{p: consequences GBE} in the inequality above.  
By \rmi, $\big((\tau-\kappa^2)\cI+\cD^2\big)^{-\rho}$ is bounded from $\ld{N}$ to $\ly{N}$, so that 
$$
\begin{aligned}
\bignorm{\mod{\nabla F(t\cD)g}}{\infty} 
& \leq C \bignorm{\mod{\nabla(\tau\cI+\cD^2)^{\rho} \, F(t\cD)g}}{2}.
\end{aligned}
$$
By arguing much as in \eqref{f: gradient intermediate}, we see that 
\begin{equation} \label{f: much as}
\begin{aligned}
	& \bignorm{\mod{\nabla(\tau\cI+\cD^2)^{\rho} \, F(t\cD)g}}{2} \\
	&\leq C \, \bigopnormto{\cL( \tau\cI+\cD^2)^{\rho} \, F(t\cD)}{2}{1/2} 
\bigopnormto{(\tau\cI+\cD^2)^{\rho} F(t\cD)}{2}{1/2}\! \bignorm{g}{2} \\
	&\leq C\, \Xi_\rho\big(F(t\cdot)\big) \bignorm{g}{2};
\end{aligned}
\end{equation} 
the last inequality follows from the spectral theorem.  

The proof of \rmv\ is similar to that of \rmiv.   By arguing much as in \eqref{f: gradient intermediate}, we see that 
$$
\begin{aligned}
\bignormto{\mod{\nabla F_t(\cD)g}}{2}{2}
\leq  C  \bignorm{\cL F_t(\cD)g}{2} \bignorm{F_t(\cD)g}{2}.
\end{aligned}
$$
By \rmiii\ and its proof, $\bignorm{F_t(\cD)g}{2}\leq C \, \upnorm{F_t}{\rho} \bignorm{g}{1}$ and 
$\bignorm{\cL F_t(\cD)g}{2}\leq C \, \upnorm{F_t}{\rho+1} \bignorm{g}{1}$.  
By combining the estimates above, we obtain that 
$$
\bignorm{\mod{\nabla F_t(\cD)g}}{2}
\leq C\,\Xi_\rho\big(F_t\big)  \bignorm{g}{1}
\quant t>0,
$$
and the first gradient estimate in \rmv\ is proved. 
In order to prove the second gradient estimate we proceed as in \rmiv.  If $\tau >\kappa^2$, then 
$$
\begin{aligned}
	\bignorm{\mod{\nabla F_t(\cD)g}}{\infty} 
	& \leq C \bignorm{\mod{\nabla(\tau\cI+\cD^2)^{\rho} \, F_t(\cD)g}}{2}.
\end{aligned}
$$
By \eqref{f: much as} (with $F_t$ instead of $F(t\cdot)$), 
$
	\bignorm{\mod{\nabla(\tau\cI+\cD^2)^{\rho} \, F_t(\cD)g}}{2} 
	\leq C\, \Xi_\rho\big(F_t\big) \bignorm{g}{2}.
$
This concludes the proof of \rmv\ and of the proposition.  
\end{proof}

\subsection{Estimates for the Poisson semigroup}
We denote by $\cP_t^N$ the Poisson semigroup $\e^{-t\cD}$.  Recall the subordination formula  
\begin{equation} \label{f: Poisson}
\cP_t^N
= t \, \ioty h_s^\BR (t) \, \cH_s^N {\dtt s}, 
\end{equation}
where $h_s^\BR$ denotes the standard Gauss--Weierstrass kernel on the real line.  
Notice the estimate
\begin{equation} \label{f: Pt LqLr}
\bigopnorm{\cP_t^N}{q;r} \leq C \, \ga(t)^{2(1/q-1/r)} \quant t >0,
\end{equation}
which is a simple consequence of the subordination formula above and the corresponding estimate 
\eqref{f: Ht LqLr} for $\cH_t^N$; see for instance \cite[Corollary 1.5]{CoM}.   
It is sometimes convenient to write $\cP_t^N = \cQ_t^0 + \cQ_t^\infty$, where
\begin{equation}\label{f: decomp Poisson}
\cQ_t^0
:= t \, \int_0^1 h_s^\BR (t) \, \cH_s^N {\dtt s}
\qquad\hbox{and} \qquad 
\cQ_t^\infty
:= t \, \int_1^\infty h_s^\BR (t) \, \cH_s^N {\dtt s}.
\end{equation}

\begin{proposition} \label{p: consequences GBE III}
There exists a positive constant $C$, independent of $f$, such that 
	\begin{enumerate}
		\item[\itemno1]
			$\bigmod{\nabla\cQ_t^0 f} \leq \e^{\kappa^2} \, \cQ_{t}^0\big(\bigmod{\nabla f}\big)$,
		\item[\itemno2]
			$\mod{\nabla\cQ_t^\infty f} \leq C \,\min(t,t^{-3/2}) \bignorm{f}{1}$ 
		\item[\itemno3]
			$\bignorm{\cQ_t^\infty f}{\infty} \leq C \,\min(t,t^{-1}) \bignorm{f}{1}$ 
	\end{enumerate}
for every $t>0$.  
\end{proposition}

\begin{proof}
First we prove \rmi. 
By Bakry's condition \eqref{f: Bakry},
$$
\begin{aligned}
	\mod{\nabla\cQ_t^0 f} 
	 \leq  t \, \int_0^1 h_s^\BR (t) \, \mod{\nabla \cH_s^Nf} {\dtt s} 
	 \leq  t \, \int_0^1 h_s^\BR (t) \, \e^{\kappa^2 s} \cH_s^N\mod{\nabla f} {\dtt s},
\end{aligned}
$$
which is clearly dominated by $\ds\e^{\kappa^2} \, t \, \int_0^1 h_{s}^\BR (t) \, \cH_s^N\mod{\nabla f} {\dtt s}
= \e^{\kappa^2} \, \cQ_{t}^0\big(\bigmod{\nabla f}\big)$,
as required.  

Next we prove \rmii.  Observe that, by Propositions~\ref{p: consequences GBE} and 
\ref{p: consequences GBE II}~\rmi\ (with $q=2$ and $r=\infty$), 
$$
\bignorm{\mod{\nabla \cH_s f}}{\infty}
\leq C \bignorm{\mod{\nabla (\tau\cI+ \cD^2)^\rho \cH_s f}}{2}.  
$$
By arguing much as in \eqref{f: gradient intermediate} (with $(\tau\cI+ \cD^2)^\rho \cH_s$ in place of $F(t\cD)$), we see that 
$$
\bignorm{\mod{\nabla (\tau\cI+ \cD^2)^\rho \cH_s f}}{2}
\leq \frac{1}{\sqrt s} \bignormto{(\tau\cI+ \cD^2)^\rho s \cD^2 \cH_s f}{2}{1/2} 
\bignormto{(\tau\cI+ \cD^2)^\rho \cH_s f}{2}{1/2}. 
$$
Now, set $\ds \om(s):= \sup_{\la\geq 0} \, (\tau + \la^2)^\rho\, s\la^2 \, \e^{-s\la^2/2}$.   By the spectral theorem, 
$$
\begin{aligned}
\bignorm{(\tau\cI+ \cD^2)^\rho s\cD^2\cH_s f}{2}
& \leq \bigopnorm{(\tau\cI+ \cD^2)^\rho s \cD^2 \cH_{s/2}}{2} \bignorm{\cH_{s/2}f}{2} \\ 
& \leq C \, \om(s) \, \bigopnorm{\cH_{s/2}}{1;2} \bignorm{f}{1} \\
& \leq C \, \max(s^{-\rho},1) \, \ga(s)^{1/2}  \bignorm{f}{1};
\end{aligned}
$$
the third inequality above follows from \eqref{f: Ht LqLr} and the fact that 
$$
\om(s)
= \sup_{v\geq 0}\, v^2 \, (\tau + v^2s^{-1})^\rho \, \e^{-v^2/2}
\leq \max(s^{-\rho},1).    
$$ 
Similarly, 
$
\ds\bignorm{(\tau\cI+ \cD^2)^\rho \cH_s f}{2}
\leq C \, \max(s^{-\rho},1) \, \ga(s)^{1/2}  \bignorm{f}{1}.
$
By combining the estimates above, we see that 
$
\ds\bignorm{\mod{\nabla \cH_s f}}{\infty}
\leq \frac{C}{\sqrt s} \, \max(s^{-\rho},1) \, \ga(s)^{1/2}  \bignorm{f}{1}.
$
Therefore 
$$
\begin{aligned}
\mod{\nabla\cQ_t^\infty f} 
\leq t \, \int_1^\infty \, h_s^\BR (t) \bignorm{\mod{\nabla \cH_s f}}{\infty} {\dtt s}
\leq  C \bignorm{f}{1} \, t \, \int_1^\infty \, h_s^\BR (t)\,s^{-3/4} \,{\dtt s}. 
\end{aligned}
$$
The last integral above is bounded above by $C\, \min(1,t^{-5/2})$.  Indeed, if $t\leq 1$, then 
$$
\int_1^\infty \, h_s^\BR (t)\,s^{-3/4} \,{\dtt s}
\leq C \int_1^\infty \,\e^{-t^2/(4s)} \,s^{-5/4} \,{\dtt s}
\leq C \int_1^\infty \,s^{-9/4} \wrt s,
$$
which is clearly finite, and, if $t\geq 1$, then 
$$
\int_1^\infty \, h_s^\BR (t)\,s^{-3/4} \,{\dtt s}
\leq C \int_1^\infty \,\e^{-t^2/(4s)} \,s^{-5/4} \,{\dtt s}
=    C \, t^{-5/2} \, \int_0^{t^2/4} \,u^{1/4}\e^{-u} \wrt u,
$$
which is bounded by $C\, t^{-5/2}$.  Therefore
$
\mod{\nabla\cQ_t^\infty f} 
\leq C \,\min(t,t^{-3/2}) \bignorm{f}{1},  
$
as required.  

Finally, we prove \rmiii.  We use the ultracontractivity of the heat semigroup, and estimate
$$
\begin{aligned}
\bignorm{\cQ_t^\infty f}{\infty} 
\leq C  \bignorm{f}{1} \, t\, \int_1^\infty h_s^{\BR} (t) \, s^{-1/2} {\dtt s} 
\leq C \bignorm{f}{1} \, t\, \int_1^\infty \e^{-t^2/(4s)} \, s^{-1} {\dtt s}. `
\end{aligned}
$$
Now, the change of variables $t^2/s=u$ transforms the last integral to 
$\ds t^{-2} \, \int_0^{t^2} \e^{-u/4} \, u {\dtt u}$.  This is bounded for $t$ small, and is bounded by $C \, t^{-2}$
for $t$ large.  By combining the estimates above, we get that 
$\bignorm{\cQ_t^\infty f}{\infty} 
\leq C \,\min(t,t^{-1}) \bignorm{f}{1},
$
as required.  
\end{proof}

\subsection{Estimates related to the wave propagator} 
Define the Fourier transform of an integrable function $\eta$ on the real line 
by $\ds\wh\eta(s) = \ir \eta(\la) \, \e^{-is\la}\wrt \la$.  
We analyse various operators by subordinating them to the wave propagator, an idea that
originates in \cite{CGT,T1}.  
At least formally, we may write $\ds\eta(\cD) = \frac{1}{2\pi} \ir \wh\eta(s) \, \cos(s\cD) \wrt s$,
whenever $\eta$ is even.  
Occasionally we need to integrate by parts in the integral above.  We do it
with the aid of \cite[Lemma~5.1]{MMV0}, which we restate for the reader's convenience.  
Hereafter~$\cO^\ell$ denotes the differential operator $s^\ell\partial_s^\ell$, acting on 
functions on the real line.  We set $\ds\cJ_\nu(v) = \frac{J_\nu(v)}{v^\nu}$, where
$J_\nu$ denotes the Bessel function of the first kind and of order $\nu$ (see, for instance, 
\cite[Section~5.3]{Le}).  

\begin{lemma}\label{l: P}
For every positive integer $J$ there exists a polynomial 
$P_{J}$ of degree $J$ without constant term, such that 
$$
\ir \wh\eta (t) \, \cos (vt) \wrt t
=   \ir  P_{J}(\cO)\wh\eta(t)\,  \cJ_{J-1/2} (t v)  \wrt t,
$$
for all functions $\eta$  such that $\cO^\ell \wh\eta\in \lu{\BR}\cap C_0(\BR)$ for all $\ell$ in $\{0,1,\ldots, J\}$.
\end{lemma} 

\noindent
Given a ``nice'' function $f$ on $N$, the formula above and the spectral theorem suggest to 
establish appropriate norm estimates of $\cJ_{J-1/2}\big(s\cD\big) f$.  This is done in the next lemma.  

\begin{lemma} \label{l: properties cJkmum}
Suppose that $\de >0$ and that $J$ is a positive integer.   There exists a constant~$C$ such that the following hold:
	\begin{enumerate}
		\item[\itemno1]
		if $J>n/2$, then
		$\bignorm{\cJ_{J-1/2}\big(s\cD\big)f}{1} \leq C\,s^{(\al-1)/2}\, \e^{\be s} \bignorm{f}{1}$ for every $s\geq \de$;
		\item[\itemno2]
		if $J>n/2+2$, then $\bignorm{\mod{\nabla \cJ_{J-1/2}\big(s\cD\big)f}}{1} 
		\leq C\, s^{(\al-3)/2}\, \e^{\be s} \bignorm{f}{1}$
		for every $s\geq \de$;
		\item[\itemno3]
		if $J>2+n$, then $\bignorm{\cJ_{J-1/2}\big(s\cD\big)f}{W^{1,\infty}(N)} 
		\leq C \bignorm{f}{1}$ for every $s\geq \de$.
	\end{enumerate}
\end{lemma}

\begin{proof}
Observe preliminarily that we can reduce the problem to the case where the support of $f$ 
is contained in $B_{\de}(o)$, for some point $o$ in $N$.  This is done considering a
smooth partition of unity $\{\psi_j\}$ so that the support of $\psi_j$ is contained in $B_{\de}(x_j)$,
for an appropriate sequence $\{x_j\}$ of points in $N$.
Thus, in the rest of the proof we assume that the support of $f$ is contained in $B_{\de}(o)$, for some $o$ in $N$. 

The proof of \rmi\ proceeds along the lines of the proof of \rmii, and it is, in fact, simpler.   
We leave the details to the interested reader.  

Now we prove \rmii.  
Observe that the support of $\mod{\nabla \cJ_{J-1/2}\big(s\cD\big)f}$ is 
contained in the ball $B_{\de+s}(o)$ by finite propagation speed.  By Schwarz's inequality,
$$
\begin{aligned}
\bignorm{\mod{\nabla \cJ_{J-1/2}\big(s\cD\big)f}}{1}  
& \leq \nu\big(B_{\de+s}(o)\big)^{1/2}  \bignorm{\mod{\nabla \cJ_{J-1/2}\big(s\cD\big)f}}{2}. \\
\end{aligned}
$$
Observe that if $n/2<2\rho\leq J-2$, which is compatible with our assumptions, then by 
Proposition~\ref{p: consequences GBE II}~\rmiv\ 
there exists a constant $C$ such that 
\begin{equation} \label{f: est nabla cJ}
\begin{aligned}
	\bignorm{\mod{\nabla \cJ_{J-1/2}\big(s\cD\big)f}}{1}  
	& \leq \frac{C}{s} \, \nu\big(B_{\de+s}(o)\big)^{1/2} \, \, \ga(s) \bignorm{f}{1} \\
	& \leq C \, s^{(\al-3)/2} \, \e^{\be s} \bignorm{f}{1} 
	\quant s \in [\de,\infty),
\end{aligned}
\end{equation}
as required.  Notice that the last inequality is a consequence of \eqref{f: volume growth}.

Finally we prove \rmiii.  By Proposition~\ref{p: consequences GBE II}~\rmii\ (with $\rho = J/2$),
$$
\bigopnorm{\cJ_{J-1/2}\big(s\cD\big)}{1;\infty}
\leq C\, \ga(s)^2 \, \upnorm{\cJ_{J-1/2}}{J}
\leq C\, \upnorm{\cJ_{J-1/2}}{J}
\quant s\geq \de.
$$
Next we estimate the gradient of $\cJ_{J-1/2}\big(s\cD\big)$.  For $\tau > \kappa^2$ we write 
$$
\nabla \cJ_{J-1/2}\big(s\cD\big)f
= \nabla \big(\tau\cI+\cD^2\big)^{-\rho/2}\, \big(\tau\cI+\cD^2\big)^{\rho/2}\cJ_{J-1/2}\big(s\cD\big)f.  
$$
Then Proposition~\ref{p: consequences GBE} implies that 
$$
\bigmod{\nabla \cJ_{J-1/2}\big(s\cD\big)f}
\leq C \, \big((\tau-\kappa^2)\cI+\cD^2\big)^{-\rho/2}\, \bigmod{\nabla \big(\tau\cI+\cD^2\big)^{\rho/2}\cJ_{J-1/2}\big(s\cD\big)f}.  
$$
Now if $\rho > n/2$, then the operator 
$\big((\tau-\kappa^2)\cI+\cD^2\big)^{-\rho/2}$ is bounded from $\ld{N}$ to $\ly{N}$, 
by Proposition~\ref{p: consequences GBE II}~\rmi, whence 
$$
\bignorm{\mod{\nabla \cJ_{J-1/2}\big(s\cD\big)f}}{\infty}
\leq C \bignorm{\mod{\nabla \big(\tau\cI+\cD^2\big)^{\rho/2}\cJ_{J-1/2}\big(s\cD\big)f}}{2}.   
$$
We can apply Proposition~\ref{p: consequences GBE II}~\rmv\ (with $F_s(\la) = (\tau+\la^2)^{\rho/2} \, \cJ_{J-1/2}(s\la)$ and $\rho_1>n/4$), 
and conclude that 
$$
\bignorm{\mod{\nabla \big(\tau\cI+\cD^2\big)^{\rho/2}\cJ_{J-1/2}\big(s\cD\big)f}}{2}
\leq C\, \Xi_{\rho_1}(F_s)  \bignorm{f}{1},
$$
where $2\rho_1+2+\rho < J$.  
Standard estimates of Bessel functions imply that 
$$
\upnorm{F_s}{\rho_1}
= \sup_{\la\geq 0}\, (\tau+\la^2)^{\rho_1+\rho/2} \mod{\cJ_{J-1/2}(s\la)}
\leq \, \sup_{\la\geq 0}\, \frac{(\tau+\la^2)^{\rho_1+\rho/2}}{(1+s\la)^{J}}.
$$
Clearly for any $s\geq \de$ this is dominated by $\ds\sup_{\la\geq 0}\, \frac{(\tau+\la^2)^{\rho_1+\rho/2}}{(1+\de\la)^{J}}$
which is finite.  A similar estimate is satisfied by $\upnorm{F_s}{\rho_1+1}$, and the required bound follows.  

It is straightforward to check that the conditions $\rho>n/2$, $\rho_1>n/4$ 
and $J\geq 2\rho_1+2+\rho$ are compatible provided that $J>n+2$.  
\end{proof}

\subsection{Laplacian cut-off functions}

We need the following result, which will be used in Section \ref{s: The local Riesz transform}. 
In the rest of the paper for each $R>0$ we set 
\begin{equation} \label{f: Upsilon R}
\Upsilon_R := \big\{(x,y)\in N\times N: d(x,y) < R\big\}.
\end{equation}   	 

\begin{lemma}\label{l: Laplacian cut-offs}
Given $R>0$, there exists positive constants $Q$ and $Q'$, depending on $\ka$, $n$ and~$R$, such that:
\begin{enumerate}
	\item[\itemno1] 
	for every $x\in N$ there exists a function $\chi_x$ in $C^\infty_c(N)$ 
	such that $0\leq \chi_x \leq 1$, $\chi_x=1$ on $B_{R/4}(x)$, $\chi_x=0$ on $B_{R/2}(x)^c$,
	$\bignorm{|\nabla \chi_x|}{\infty}\leq Q$ and $\bignorm{\Delta \chi_x}{\infty} \leq Q$;  
	\item[\itemno2] 
	there exists a function $\varphi$ in $C^\infty_c(N\times N)$ such that $0\leq \varphi \leq 1$, $\varphi=1$ 
	in $\Upsilon_{R/4}$ and $\varphi=0$ in $\Upsilon_{R}^c$, $\bignorm{|\nabla \varphi|}{\infty}\leq Q'$ 
	and $\bignorm{\Delta \varphi}{\infty} \leq Q'$.
	\end{enumerate} 
\end{lemma}

\begin{proof}
For the proof of \rmi, see \cite[Theorem~6.33]{CC}.
 
We now prove \rmii.  
Denote by $\fP$ an $R/4$-discretization of $N$, i.e., a set of points $\{p_j: j=1,2,3,\ldots\}$
in~$N$ that is maximal with respect to the property
$$
d(p_j,p_k) > R/8 \qquad\hbox{when $j\neq k$ and}\qquad d(x,\fP) < R/4 \quant x \in N.  
$$
We write $P_j$ instead of $p_j\times p_j$
and $Q_j^{R/2}$ instead of $B_{R/2}(p_j)\times B_{R/2}(p_j)$.
The family $\big\{B_{R}(p_j): j = 1,2,3,\ldots\big\}$
has the finite overlapping property (see, for instance, \cite[Lemma~1.1]{He}).  Hence the same is true of 
$\big\{Q_j^{{R/2}}: j = 1,2,3,\ldots\big\}$.  It is straightforward to check that 
$$
\ds \Upsilon_{R/4} \subseteq \bigcup_{j=1}^\infty \, Q_j^{R/2}
\subseteq \Upsilon_{R}.
$$  
Indeed, if $(x,y)$ is in $\Upsilon_{R/4}$, then $d(x,y) < R/4$.  Since $\fP$ is a $R/4$-discretization of $N$,
there exists an integer $j$ such that $d(x,p_j)<R/4$.  The triangle inequality then implies that $d(y,p_j)<R/2$,
whence $(x,y)$ belongs to $Q_j^{R/2}$, and the left inclusion above is proved.  
The right inclusion follows from the trivial fact that if $(x,y)$ belongs to $Q_j^{R/2}$, then $d(x,y) <R$.  

For each integer $j$ set $\phi_j := \chi_{p_j}\otimes \chi_{p_j}$,
where the $\chi_{p_j}$ are cut-offs on $N$ as in \rmi.  Notice that $\phi_j$
is a smooth function with compact support on $N\times N$, that $\phi_j = 1$ on $Q_j^{R/4}$, $\phi_j = 0$ on $\big[Q_j^{R/2}]^c$, 
$\ds \bignorm{|\nabla \phi_j|}{\infty} \leq 2Q$ and $\bignorm{\Delta \phi_j}{\infty} \leq 2Q$,
where $\nabla$ and $\Delta$ denote here the gradient and the Laplace--Beltrami operator on $N\times N$, respectively.  
Set   
$\ds
\vp_j 
:= {\phi_{P_j}}\Big/{\ds\sum_{k =1}^\infty \phi_{P_k}}$
and  
$
\ds\vp := \sum_{j=1}^\infty \, \vp_j.
$
It is straightforward to check that $\vp$ possesses the required properties.  We omit the details. 
\end{proof}

\section{Estimates for the Poisson operator on slices}
\label{s: Poisson}
In this section we consider the Riemannian manifold $N\times \BR$, endowed with the natural product metric.  
Here $N$ satisfies our standing assumptions (see the beginning of Section~\ref{s: Background material}). 
We shall often, but not always,
denote points in $N\times \BR$ by capital letters $X, Y, Z, \ldots$.  Usually, lower case latin letters
$x, y, z, \ldots$ will denote points in $N$.  Thus, a point $X$ in $N\times \BR$ will be often written
$(x,u)$, where $x$ is in $N$ and $u$ is a real number.   
Denote by $D$ the Riemannian distance on $N\times \BR$, i.e.,
\begin{equation} \label{f: dist prod}
D\big((x,u),(y,v) \big)
:= \sqrt{d(x,y)^2 + \mod{u-v}^2}
\quant x,y\in N \quant u,v\in \BR.
\end{equation}
The Riemannian measure on $N\times \BR$ will be denoted by $\cY$.  Thus, $\wrt \cY (Y) = \wrt \nu(y) \wrt v$
when $Y = (y,v)$.   
We shall denote by $\Nabla$ and $\DDelta$ the gradient on $N\times \BR$ and the (negative) 
Laplace--Beltrami operators on $N\times\BR$, respectively.  
When we choose the natural co-ordinate system $(x,t)$ on $N\times \BR$, where $x$ varies in an open chart of $N$, and $t$
is in $\BR$, we have that $\Nabla F = (\nabla F, \partial_t F)$, and $\DDelta F = \Delta F + \partial_t^2 F$.  

Throughout the paper $\si$ will denote a fixed positive number such that 
\begin{equation} \label{f: sigma}
\si < \frac{\pi}{4\be} \, \min\big(1-1/n, \sqrt c\big) 
\end{equation}
where $c$ and $\be$ are as in \eqref{f: assumptions on ht} and \eqref{f: volume growth}, respectively.
Set 
\begin{equation}\label{f: lambda1}
\la_1 := \pi/(2\si).
\end{equation}  
For any $\eta$ in $[0,\si)$, set $\Sisieta := N\times (\eta,2\si-\eta)$ in~$N\times \BR$. We write $\Si$ for $\Si_0$. 
Most of our analysis is concerned with functions defined on the open slice $\Sisi$.  
In particular, we shall need to consider $\Sisieta$ for some 
$\eta\neq 0$ only in Section~\ref{s: G}. 
We shall write $\norm{\cdot}{p}$ and $\norm{\cdot}{\lp{\Sisieta}}$ for the $L^p$ norms on $N$ and on $\Sisieta$, respectively.  
Given a function $F$ on $\Sisi $, we denote by $F^\flat$ the function on $N$, defined by 
\begin{equation} \label{f: F flat}
F^\flat (x)
= \int_0^{2\si} F(x,t) \wrt t
\quant x \in N,
\end{equation}
whenever the latter integral makes sense.  Observe that, by H\"older's inequality, 
\begin{equation} \label{f: F flat F}
\bignorm{F^\flat}{p}
= \Big[\int_{N} \wrt \nu(x)\, \Bigmod{\int_0^{2\si} F(x,u) \wrt u}^p\, \Big]^{1/p} 
\leq (2\si)^{1/p'} \bignorm{F}{\lp{\Sisi }}.
\end{equation}
For each $\eta$ in $[0,\si)$ and $t$ in $(\eta,2\si-\eta)$ consider the meromorphic function 
$$\ds
\Mteta  (\la)
= \frac{\cosh(t-\si)\la}{\cosh(\si-\eta) \la}.
$$
For the sake of simplicity, we write $\Mt$ instead of $M_t^0$.  Thus, 
$
\ds M_t(\la)
:= \frac{\cosh (t-\si)\la }{\cosh (\si\la)}.  
$
An elementary computation shows that 
$$
\begin{aligned}
	\Mteta  (\la) 
	= \e^{(\eta-t)\la} + \frac{1}{2} \, \big[\e^{(t-\si)\la}-\e^{(2\eta-t-\si)\la}\big]\, \Msieta (\la).  
\end{aligned}
$$
We shall often work with the special case of the formula above corresponding to $\eta=0$.
Set 
$
\cPSeta f(\cdot, t)
:= \Mteta  (\cD) f.
$
In the case where $\eta<t<\si$ it is sometimes convenient to use the expression above for $\Mteta  $ and write
\begin{equation} \label{f: formula cPS}
\cPSeta f(\cdot, t)
	= \cP_{t-\eta}^N f + \frac{1}{2} \,\big[\cP_{\si-t}^N-\cP_{\si+t-2\eta}^N\big] \Msieta \big(\cD\big) f.
\end{equation}
The operator $\cPSeta$ is called the \textit{Poisson operator for~$\Sisieta$ with periodic boundary conditions}. 
The following proposition partially justifies this terminology.

\begin{proposition} \label{p: Dirichlet problem slice}
Suppose that $f$ is in $C_0(N)$ (continuous functions on $N$ vanishing at infinity).  Then the function equal to 
$\cPSeta f$ in $\Sisieta $ and to $f$ on $\partial\Sisieta $,
is smooth on $\Sisieta $, continuous on $\OV{\Si}_{\eta}$, and solves the Dirichlet problem
$$
\DDelta u = 0 \quad\hbox{{\textrm{in}}\quad $\Sisieta $}  \qquad u(\cdot,\eta) = f = u(\cdot,2\si-\eta).
$$
\end{proposition}

\noindent
We postpone the proof of Proposition \ref{p: Dirichlet problem slice} 
at the end of this section.
We analyse $\cPSeta$ by subordinating $\Mteta  (\cD)$ to the wave propagator. 
Denote by $K_t^{\eta}$ the Fourier transform of~$\Mteta  $. 
It is well known (see, for instance, \cite[formula 7.19, p.~34]{Ob}) that 
\begin{equation} \label{f: Kteta}
	K_t^{\eta}(s)
	= 4\pi\de  \,\, \frac{\sin{\pi\de(t-\eta)}
	\, \cosh\pi\de s}{\cosh{2\pi\de s}-\cos{2\pi\de(\eta-t)}},
\end{equation}
where $\de := 1/[2(\si-\eta)]$.  
We shall write $K_t$ instead of $K_t^0$.  Thus, 
\begin{equation} \label{f: Kt}
K_t(s)
= \frac{2\pi}{\si} \, \frac{\ds \sin\frac{\pi t}{2\si}\, 
\cosh\frac{\pi s}{2\si}}{\ds \cosh\frac{\pi s}{\si}-\cos\frac{\pi t}{\si}}.
\end{equation}
By spectral theory and Fourier inversion formula 
$$
\cPSeta f(\cdot,t)
=  \frac{1}{2\pi} \, \ir \, K_t^\eta(s)\, \cos(s\cD) f \wrt s  
$$
and, when $\eta<t<\si$,
$$
\cPSeta f(\cdot,t)
= \cP_{t-\eta}^Nf + \big[\cP_{\si-t}^N-\cP_{\si+t-2\eta}^N\big]\frac{1}{4\pi} \, \ir \, K_\si^{\eta}(s) \,  \cos\big(s\cD) f \wrt s.
$$
We are led to establish certain properties of $K_t^{\eta}$ and of their derivatives.
Most of our applications will involve only $K_t$.  Thus, we mainly concentrate on this special case.
Set $S := \BR\times (0,\si]$.  For each $\de$ in $(0,\si)$ denote by $D_\de$ 
the disc in the plane with radius $\de$ centred at the origin.  Set $S_\de := S\setminus D_\de$. 

\begin{lemma} \label{l: properties ktau}
Suppose that $\de\in(0,\si)$, $\vep\in(0,\la_1)$, $J$ is a positive integer and 
$\ell\in\{0,1\}$.  Then there exists a constant $C$ such that 
$\bigmod{\partial_t^\ell P_J(\cO) K_t(s)} \leq C\, \min\big(\mod{s}, \e^{(\vep-\la_1)\mod{s}}\big)$
for every $(s,t)$ in $S_\de$.
\end{lemma}

\begin{proof}
A straightforward induction argument (using \eqref{f: Kt}) proves that 
$ \bigmod{\partial_t^\ell \partial_s^j K_t(s)} \leq C\, \min\big(1, \e^{-\pi\mod{s}/2\si}\big)$
for every $(s,t)$ in $S_\de$ and every nonnegative integer $j\leq J$.  The required estimate
then follows from the form of the differential operator $P_J(\cO)$.  In particular, 
the required estimate for $\mod{s}$ small follows from the fact that $P_J$ has no constant term.  
\end{proof}

\noindent\label{pp: Dunfors}
Recall that the \textit{extended Dunford class} $\cE(\bS_\psi)$ is defined as follows \cite[p.~28]{Haa}
$$
\cE(\bS_\psi)
= H_0^\infty(\bS_\psi) \oplus \big\langle (1+z)^{-1}\big\rangle \oplus \langle 1\rangle,
$$
where $H_0^\infty(\bS_\psi)$ denotes the class of all holomorphic functions $f$ in the sector 
$\bS_\psi := \{z\in \BC: \mod{\arg z} <\psi \}$ for which there exist positive constants $C$ and $s$ such that
$$
\bigmod{f(z)}
\leq C \, \frac{\mod{z}^s}{1+\mod{z}^{2s}} \quant z \in \bS_\psi.
$$
The space $\cE(\bS_\psi)$ is endowed with the uniform norm.   

\begin{lemma} \label{l: Mt sqrt cL}
Suppose that $0<\de<\si$.  The following hold:
\begin{enumerate}
	\item[\itemno1]
		for each positive even integer $J$ there exists a constant $C$ such that 
		$$
		\bignorm{\cL^J \Mt (\cD)f}{2}
		\leq C \bignorm{f}{1}
		\quant t \in [\de,2\si-\de] \quant f \in \lu{N}; 
		$$
	\item[\itemno2]
		$\Mt (\cD)f$ is smooth, and there exists a constant $C$ such that 
		$$
		\bignorm{\Mt (\cD)f}{C_b^1(N)}
		\leq C \bignorm{f}{1}
		\quant t \in [\de,2\si-\de] \quant f \in \lu{N};
		$$ 
	\item[\itemno3]
		for every $\vep>0$ and $R>0$ there exists a constant $C$ such that 
		$$
		\sup_{t\in (0,\si]} \,  \max\big(\mod{\Mt (\cD)f(x)},\mod{\cL\Mt (\cD)f(x)},\mod{\Nabla \Mt (\cD)f(x)} \big)
		\leq C \, \e^{(\vep-\la_1)d(x,o)} \bignorm{f}{1} 
		$$
		for every $o$ in $N$, every $x$ in $B_{2R}(o)^c$ and 
		every $f$ in $\lu{N}$ with support contained in the ball $B_R(o)$;  
	\item[\itemno4]
		for each $\vp$ in $(\pi/4,\pi/2)$ the function $\Msi$ belongs to $\cE(\bS_{\vp})$ and 
		there exists a constant $C$ such that 
		$$
		\sup_{t\in (0,2\si)} \bigopnorm{\Mt (\cD)}{p} \leq 1+C \bignorm{\Msi }{\cE(\bS_{\vp})}
		$$
		for every $p$ in $[1,\infty]$;
	\item[\itemno5]
		for each $p$ in $(1,\infty]$ there exists a constant $C$ such that 
		$$
		\sup_{y\in N} \bignorm{k_{\Mt (\cD)}(\cdot,y)}{p} 
		= \bigopnorm{\Mt (\cD)}{1;p}
		\leq C \, t^{-n/p'}
		\quant t \in (0,\si].  
		$$
\end{enumerate}
\end{lemma}

\begin{proof}
Part \rmi\ follows from Proposition~\ref{p: consequences GBE II}~\rmiii\ and the trivial fact that 
for any $\rho_1>0$
\begin{equation} \label{f: de duesimde}
\sup_{t\in [\de,2\si-\de]} \, \sup_{\la\geq 0} \, (1+\la^2)^{\rho_1} \, \la^{2J} \Mt (\la) 
\leq  \sup_{\la\geq 0} \, (1+\la^2)^{\rho_1} \, \la^{2J} M_\de(\la) 
< \infty. 
\end{equation}

Now we prove \rmii.  The smoothness of $\Mt (\cD)f$ follows from \rmi\ and a local Sobolev's embedding theorem.  
The estimate $\bignorm{\Mt (\cD) f}{\infty} \leq C \bignorm{f}{1}$ is a direct consequence of 
Proposition~\ref{p: consequences GBE II}~\rmiii\ and of an estimate similar to \eqref{f: de duesimde}. 
Finally, Proposition~\ref{p: consequences GBE}~\rmi\ and Proposition~\ref{p: consequences GBE II}~\rmi,\rmv\
imply that for $\tau>\kappa^2$ and $\si_1>n/4$
$$
\bignorm{\mod{\nabla \Mt (\cD) f}}{\infty}
\leq C \bignorm{\mod{\nabla (\tau\cI+\cD^2)^{\si_1} \Mt (\cD) f}}{2}
\leq C\, \Xi_{\si_1} \big(F_t\big) \bignorm{f}{1},
$$  
where $F_t(\la) := (\tau+\la^2)^{\si_1} \Mt (\la)$.  It is straightforward to check that 
$\ds\sup_{t\in [\de,2\si-\de]}\, \Xi_{\si_1} \big(F_t\big)<\infty$, thereby concluding the proof of \rmii.   

Next, we prove the estimate in \rmiii\ concerning $\mod{\Nabla M_t(\cD)f}$.  The proofs of the estimates
for $\mod{\Mt (\cD)f}$ and $\mod{\cL\Mt (\cD)f}$ are similar, perhaps easier, 
and we leave the details to the interested reader.  Suppose that $d(x,o)\geq 2R$ and choose $J>n+2$.  
By finite propagation speed and Lemmata~\ref{l: P} and \ref{l: properties cJkmum}~\rmiii, 
$$
\begin{aligned}
\bigmod{\nabla \cPS f(x,t)}
& \leq C \, \int_{\mod{s} \geq d(x,o)-R} \, \bigmod{P_J(\cO) K_t(s)} \, 
	\bigmod{\nabla \cJ_{J-1/2}\big(s\cD\big)f}(x) \wrt s \\
& \leq C \bignorm{f}{1}\,\int_{\mod{s} \geq d(x,o)-R} \, \bigmod{P_J(\cO) K_t(s)} \wrt s.
\end{aligned}
$$
Since $(s,t)$ is in $S_R$, Lemma~\ref{l: properties ktau} ensures that there exists a constant $C$ such that 
$$
\sup_{t\in (0,\si)} \, \bigmod{P_J(\cO)K_t(s)} 
\leq C\, \e^{(\vep/2-\la_1) \mod{s}}.
$$  
This and the estimates above imply that 
$$
\bigmod{\nabla \cPS f(x,t)}
	\leq C \, \e^{(\vep-\la_1) d(x,o)} \bignorm{f}{1}
	     \, \int_{\mod{s} \geq d(x,o)-R} \, \e^{-\vep \mod{s}/2}\, \wrt s 
	\leq C \, \e^{(\vep-\la_1) d(x,o)} \bignorm{f}{1}.
$$
The right hand side does not depend on $t$ in $(0,1)$, and the required estimate (with $\nabla$ in place of $\Nabla$)
follows.
It remains to prove a similar estimate for $\partial_t \cPS f(x,t)$.
By finite propagation speed and Lemmata~\ref{l: P} and \ref{l: properties cJkmum}~\rmiii, 
$$
\begin{aligned}
\bigmod{\partial_t \cPS f(x,t)}
& \leq C \, \int_{\mod{s} \geq d(x,o)-R} \, \bigmod{\partial_t P_N(\cO) K_t(s)} \, 
	\bigmod{\cJ_{J-1/2}\big(s\cD\big)f(x)} \wrt s \\
& \leq C \bignorm{f}{1} \int_{\mod{s} \geq d(x,o)-R} \, \bigmod{\partial_t P_N(\cO) K_t(s)}\wrt s.
\end{aligned}
$$
The required estimate follows from this by arguing much as above.  

Next we prove \rmiv.  Since $\Mt  = M_{2\si-t}$ for each $t$ in $(0,2\si)$,
it suffices to prove the required estimate in the case where $0<t\leq \si$. 
The contractivity of the Poisson semigroup on $\lp{N}$ and \eqref{f: formula cPS} imply the estimate 
\begin{equation} \label{f: conditional Mt}
\sup_{t\in(0,2\si)} \bigopnorm{\Mt (\cD)}{p} 
\leq \big[1+\bigopnorm{\Msi (\cD)}{p} \big].  
\end{equation}
Since $\cL$ generates a contraction semigroup on $\lp{N}$, $\cL$ is a sectorial operator of angle $\pi/2$ on $\lp{N}$, 
by the easy part of the Hille--Yosida theorem.  By abstract nonsense, $\cD$ is a sectorial operator of 
angle $\pi/4$ on $\lp{N}$ \cite[Proposition~3.1.2]{Haa}.  It is straightforward to check that the functions
$\ds z \mapsto \frac{1-\e^{-2\si z}}{\e^{\si z} + \e^{-\si z}}$ 
and $z\mapsto \e^{-\si z} - (1+z)^{-1}$ are in $H_0^\infty(\bS_{\vp})$.  
Furthermore
$$
\begin{aligned}
	\Msi (z)
	 = \frac{2}{\e^{\si z} + \e^{-\si z}} - \e^{-\si z} + \e^{-\si z} 
	 = \frac{1-\e^{-2\si z}}{\e^{\si z} + \e^{-\si z}} + \e^{-\si z} - \frac{1}{z+1} + \frac{1}{z+1},
\end{aligned}
$$
whence $\Msi$ belongs to the extended Dunford class $\cE(\bS_{\vp})$. 
Therefore 
$\ds
\bigopnorm{\Msi (\cD)}{p} 
\leq C \bignorm{\Msi }{\cE(\bS_{\vp})} 
$
\cite[Theorem~2.3.3]{Haa}, which, combined with \eqref{f: conditional Mt}, yields the required estimate.  


Finally we prove \rmv.  
The first equality follows from abstract nonsense (see, for instance, \cite[Theorem~VI.8.6, p.~508]{DS}).
Formula \eqref{f: formula cPS}, 
the contractivity of the Poisson semigroup on $\lp{N}$ and the  estimate $\bigopnorm{\cP_t}{1;p}
\leq C\, t^{-n/p'}$ (see \eqref{f: Pt LqLr}) yield
$$
\bigopnorm{\Mt (\cD)}{1;p} 
\leq \big[C\, t^{-n/p'} +\bigopnorm{\Msi (\cD)}{1;p} \big];  
$$
the required bound follows from Proposition~\ref{p: consequences GBE II}~\rmii\ and the fact that $t$ is small.
\end{proof}

\noindent
Given a function $f$ on $N$, we set
\begin{equation} \label{f: notation cN}
\cN f 
:= \bignorm{f}{1} + \bignorm{\mod{\nabla f}}{1} + \bignorm{\cD f}{1},
\end{equation}
whenever the right hand side makes sense.   
For each $p$ in $[1,\infty)$ we denote by $\cB^p$ the space of all measurable 
functions $F: \Si \to \BC$ such that 
$$
\bignorm{F}{\cB^p}
:= \sup_{t\in (0,2\si)} \, \bignorm{F(\cdot,t)}{p} 
< \infty,
$$
endowed with the ``norm'' $\bignorm{\cdot}{\cB^p}$.  

\begin{theorem} \label{t: est nabla cPS}
There exists a constant $C$ such that 
$
\bignorm{\mod{\Nabla \cPS f}}{\cB^1}
\leq C \, \cN f
$
for every function $f$ such that $\cN f$ is finite. 
\end{theorem}

\begin{proof}
Since $\cPS f(\cdot,t) = \cPS f(\cdot,2\si-t)$, it suffices to restrict $t$ to $(0,\si]$.  
By \eqref{f: formula cPS} and the assumption $\cD f\in \lu{N}$,
\[\begin{aligned}
	\partial_t \cP^0 f(\cdot,t) 
& = -\cD \cP_t^N f +\frac 12 \cD \left[\cP_{\si-t}^N + \cP_{\si+t}^N\right]\Msi (\cD) f\\
& = -\cP_t^N \cD f +\frac 12  \left[\cP_{\si-t}^N + \cP_{\si+t}^N\right] \Msi (\cD)\cD f.
\end{aligned}
\]
The estimate $\bignorm{\partial_t \cP^0 f(\cdot,t)}{1} \leq C \bignorm{\cD f}{1}\leq C \,\cN f$, with $C$ independent of $f$
and $t$, follows by arguing as in the proof of Lemma~\ref{l: Mt sqrt cL}~\rmiv (with $\cD f$ in place of $f$).
Hence $\bignorm{\partial_t \cPS f}{\cB^1} \leq C \, \cN f$.

We now estimate $\bignorm{\mod{\nabla\cPS f(\cdot, t)}}{1}$ for $t$ in $(0,\si]$.  
Let $\{B_j\}$ be a covering of~$N$ with geodesic balls of radius $1$ and centre $p_j$ enjoying 
the finite overlapping property.  
Denote by $\{\psi_j\}$ a partition of unity subordinated to this covering with the property that 
$\{\nabla \psi_j\}$ are uniformly bounded with respect to $j$, 
and write $f = \sum_j f_j$, where $f_j := \psi_j \, f$.  Then
$$
\begin{aligned}
\bignorm{\mod{\nabla \cPS f(\cdot,t)}}{1}
\leq \sum_j  \bignorm{\mod{\nabla \cPS f_j(\cdot,t)}}{\lu{2B_j}} 
+  \sum_j  \bignorm{\mod{\nabla \cPS f_j}(\cdot,t)}{\lu{N\setminus 2B_j}},
\end{aligned}
$$
where $2B_j$ denotes the ball with centre $p_j$ and radius $2$. 
Recall that $\cPS f(\cdot, t)  = M_t(\cD)f$ By Lemma~\ref{l: Mt sqrt cL}~\rmiii, 
there exists a constant $C$, independent of~$j$ and of $t$ in $(0,\si]$, such that 
$$
\bigmod{\nabla \Mt (\cD)f_j (x)}
\leq C\, \e^{(\vep-\la_1)d(x,p_j)}  \bignorm{f_j}{1}
\quant x \in N\setminus 2B_j.  
$$
Since $\la_1 > 2\be$ (see \eqref{f: sigma}), the function $x\mapsto \e^{(\vep-\la_1)d(x,p_j)}$ is, for $\vep$ small
enough, in $\lu{N}$, with norm bound independent of $j$, so that
$$
\sup_{t\in (0,\si]} \, \sum_j  \bignorm{\mod{\nabla \Mt (\cD) f_j}}{\lu{N\setminus 2B_j}} 
\leq C\, \sum_j \bignorm{f_j}{1} 
\leq C \bignorm{f}{1}.
$$

Schwarz's inequality and Proposition~\ref{p: consequences GBE II}~\rmv\ imply that for any $\rho_1 >n/4$
$$
\begin{aligned}
	\bignorm{\mod{\nabla M_t(\cD) f_j}}{\lu{2B_j}} 
	\leq      \sqrt{\nu\big(2B_j \big)} \bignorm{\mod{\nabla M_t(\cD) f_j}}{\ld{2B_j}} 
	\leq C \, \sqrt{\nu\big(2B_j \big)} \, \Xi_{\rho_1} \big(\Mt \big) \bignorm{f_j}{1}.
\end{aligned}
$$
Note that $\nu\big(2B_j \big)$ is uniformly bounded with respect to $j$, because $N$ has bounded geometry.  
Now, fix $\de$ in $(0,\si)$.  It is straightforward to check that 
$\ds\sup_{t\in [\de,\si]}\, \Xi_{\rho_1} \big(\Mt \big)$ is finite.  Thus, there exists a constant $C$ such that 
$$
\sup_{t\in [\de,\si]}\, \sum_j  \bignorm{\mod{\nabla \Mt (\cD) f_j}}{\lu{2B_j}} 
\leq C \, \sum_j \bignorm{f_j}{1}
\leq C \bignorm{f}{1}
\quant t \in [\de,\si]. 
$$
It remains to estimate $\ds \sup_{t\in (0,\de)} \, \sum_j  \bignorm{\mod{\nabla \Mt (\cD) f_j}}{\lu{2B_j}}$.
It is convenient to write $\Mt (\cD)$ as in \eqref{f: formula cPS}.  The triangle inequality and 
the decomposition $\cP_t^N = \cQ_t^0+\cQ_t^\infty$ (see \eqref{f: decomp Poisson}) imply that 
$$
\begin{aligned}
\mod{\nabla \Mt (\cD) f_j}
& \leq \mod{\nabla \cQ_t^0 f_j} +  \mod{\nabla \cQ_t^\infty f_j} 
 + \frac{1}{2} \, \bigmod{\nabla [\cP_{\si-t}^N-\cP_{\si+t}^N] \Msi (\cD) f_j}. \\ 
\end{aligned}
$$
Observe that $\mod{\nabla \cQ_t^0 f_j}\leq \e^{\kappa^2}\cQ_t^0 \mod{\nabla f_j}$ 
by Proposition~\ref{p: consequences GBE III}~\rmi, whence 
$$
\bignorm{\mod{\nabla \cQ_t^0 f_j}}{\lu{2B_j}}
\leq \e^{\kappa^2}\bignorm{\cQ_t^0 \mod{\nabla f_j}}{\lu{2B_j}}
\leq \e^{\kappa^2}\bignorm{\cP_t^N \mod{\nabla f_j}}{1}
\leq \e^{\kappa^2}\bignorm{\mod{\nabla f_j}}{1};
$$
we have used the contractivity of the Poisson semigroup on $\lu{N}$ in the last inequality.     
Note that $\mod{\nabla f_j} \leq C \, \mod{\psi_j\nabla f} + C \, \One_{B_j} \, f$, so that 
$$
\sup_{t\in (0,\de)} \,\sum_j \, \bignorm{\mod{\nabla \cQ_t^0 f_j}}{\lu{2B_j}} 
\leq C\,\Big[\sum_j \, \bignorm{\mod{\psi_j\nabla f}}{1}  
  +  \sum_j \, \bignorm{\One_{B_j} \, f}{1}\Big]
\leq C\, \cN f.  
$$
Furthermore, we have trivially 
$$
\begin{aligned}
\sup_{t\in (0,\de)} \sum_j \, \bignorm{\mod{\nabla \cQ_t^\infty f_j}}{\lu{2B_j}} 
	\leq \sum_j  \nu\big(2B_j\big) \, \sup_{t\in (0,\de)}\! \bignorm{\mod{\nabla \cQ_t^\infty f_j}}{\infty} \!\!
	\leq C \sum_j \norm{f_j}{1} 
	\leq C \norm{f}{1},  
\end{aligned}
$$
where the second inequality above follows from Proposition~\ref{p: consequences GBE III}~\rmii\ and
the fact that $N$ has bounded geometry.  
Finally, set $F_t (\la) :=  [\e^{(t-\si)\la}-\e^{-(\si+t)\la}]\, \Msi (\la)$.  It is straightforward to check that
if $\rho_1>n/4$, then $\ds\sup_{t\in (0,\de)} \, \Xi_{\rho_1}\big(F_t\big) < \infty$.   Then Schwarz's inequality
and Proposition~\ref{p: consequences GBE II}~\rmv\ imply that 
$$
\begin{aligned}
\sup_{t\in (0,\de)} \sum_j \, \bignorm{\mod{\nabla [F_t(\cD) f_j]}}{\lu{2B_j}} 
& \leq \sum_j \, \sqrt{\nu\big(2B_j\big)} \, \sup_{t\in (0,\de)} \bignorm{\mod{\nabla [F_t(\cD) f_j]}}{\ld{2B_j}}\\
& \leq C \, \sum_j \, \bignorm{f_j}{1}\\  
& \leq C \bignorm{f}{1}.
\end{aligned}
$$
The required conclusion follows by combining the estimates above.
\end{proof}

We complete this section by proving Proposition \ref{p: Dirichlet problem slice}.

\begin{proof}[Proof (of Proposition \ref{p: Dirichlet problem slice}).]
The function $\cPSeta f$ is harmonic in $\Sisieta$, hence smooth therein, by elliptic regularity.  
We prove the continuity at the boundary.
Note that $\cPSeta f(\cdot,t)= \cPSeta f(\cdot,2\si-t)$ for every $t$ in $(\eta,2\si-\eta)$.
Thus, it suffices to prove the continuity at $t=\eta$.  Fix $x$ in $N$, and write 
$$
\cPSeta f(y,t) - f(x) 
= \cPSeta f(y,t) - f(y) + f(y) - f(x).
$$
Since $f$ is in $C_0(N)$, it is uniformly continuous on $N$; hence for every $\vep>0$ there exists $\de>0$ such that 
$\bigmod{f(y)-f(x)} <\vep$ whenever $d(x,y) <\de$.  

By \eqref{f: formula cPS}, we can write $\ds\cPSeta f(\cdot,t) - f = \cP_{t-\eta}^N f - f +
\frac{1}{2} \,\big[\cP_{\si-t}^N-\cP_{\si+t-2\eta}^N\big] \Msieta \big(\cD\big) f$.  
The heat semigroup $\{\cH_t^N\}$ is strongly continuous on $C_0(N)$ \cite[Lemma~5.2.8]{Da1}.
A straightforward argument using the subordination formula \eqref{f: Poisson} shows that
the same holds for the Poisson  semigroup $\{\cP_t^N\}$.  
Hence $\ds \lim_{t\downarrow \eta} \bignorm{\cP_{t-\eta}^N f - f}{C_0(N)} = 0$.

It remains to prove that
\begin{equation}\label{f: claim pb Diric}
\lim_{t\downarrow\eta}	\bignorm{\big[\cP_{\si-t}^N-\cP_{\si+t-2\eta}^N\big] \Msieta \big(\cD\big) f}{\infty}=0.
	\end{equation}
To this end, fix $\vep>0$ and consider a sequence $\{\varphi_k\}\subset C^\infty_c(N)$ such that 
$\|\varphi_k-f\|_{\infty}\to 0$ as $k\to\infty$. 
The Poisson semigroup is contractive on $L^\infty(N)$, hence so is on $C_0(N)$. Thus, by 
a variant of Lemma~\ref{l: Mt sqrt cL}~\rmiv\ (with $\eta>0$),  
\begin{equation*}
\begin{aligned}
\bignorm{\big[\cP_{\si-t}^N-\cP_{\si+t-2\eta}^N\big] \Msieta \big(\cD\big) (f-\varphi_k)}{\infty}
&\leq 2\bignorm{ \Msieta \big(\cD\big) (f-\varphi_k)}{\infty} \\
&\leq 2\left.\bigopnorm{\Msieta (\cD)}{\infty}\right. \bignorm{ f-\varphi_k}{\infty}\\
&\leq C\,\big[1+\bignorm{\Msieta (\cD)}{\cE(\bS_{\vp})}\! \big]\bignorm{f-\varphi_k}{\infty}.
\end{aligned}
\end{equation*}
In particular, we can fix $k_0$ large enough so that 
\begin{equation}\label{f: Diric appr}
\bignorm{\big[\cP_{\si-t}^N-\cP_{\si+t-2\eta}^N\big] \Msieta \big(\cD\big) 
(f-\varphi_{k_0})}{\infty}<\vep/2 \quant t\in (\eta,\sigma). 
\end{equation}
 Furthermore
\[\begin{aligned}
\ds \left[\cP_{\si-t}^N-\cP_{\si+t-2\eta}^N\right]\Msieta \big(\cD\big) \varphi_{k_0} 
& = -\int_{\si-t}^{\si+t-2\eta} \frac{\textrm{d}}{\textrm{d} s}\cP_s^N \Msieta \big(\cD\big) \varphi_{k_0} \wrt s\\
& = \int_{\si-t}^{\si+t-2\eta} \cD\cP_s^N \Msieta \big(\cD\big) \varphi_{k_0} \wrt s,\\
& = \int_{\si-t}^{\si+t-2\eta} \cP_s^N\cD \Msieta \big(\cD\big) \varphi_{k_0} \wrt s,
\end{aligned}\]
so that, using also Proposition \ref{p: consequences GBE II} \rmiii,
\[\begin{aligned}
\ds \bignorm{\left[\cP_{\si-t}^N-\cP_{\si+t-2\eta}^N\right] \Msieta \big(\cD\big) \varphi_{k_0}}{\infty} &
\leq \int_{\si-t}^{\si+t-2\eta} \bignorm{\cP_s^N\cD \Msieta \big(\cD\big) \varphi_{k_0}}{\infty} \wrt s\\
&\leq \int_{\si-t}^{\si+t-2\eta} \bignorm{\cD \Msieta \big(\cD\big) \varphi_{k_0}}{\infty} \wrt s\\
&\leq C\,(t-\eta)\bignorm{\varphi_{k_0}}{2},
\end{aligned}\]
which is smaller than $\vep/2$ for $t$ close enough to $\eta$. Together with \eqref{f: Diric appr}, 
this proves \eqref{f: claim pb Diric} and concludes the proof of the proposition.
 \end{proof}

\section{Estimates for the Green function on slices} 
\label{s: Green}

The \textit{Dirichlet heat semigroup} for $\Sisi $ is given by 
$\cH_t^{\Sisi } = \cH_t^{N}\otimes \cH_t^{[0,2\si]}$,
where $\{\cH_t^{[0,2\si]}: t>0\}$ denotes the heat semigroup on $[0,2\si]$ with Dirichlet boundary conditions.
Recall that $\la_1 = \pi /(2\si)$ (see \eqref{f: lambda1}): the number $\la_1^2$ 
is the first eigenvalue of the operator $-\textrm{d}^2/\textrm{d} x^2$ 
with Dirichlet boundary conditions on $[0,2\si]$.  The associated eigenfunction is $\sin \la_1 u$.  Set 
\begin{equation}\label{f: definition rho}
\vr (u) = \mathrm{\dist}\big(u, \BR\setminus [0,2\si]\big)
\end{equation}
and observe that $\vr (u) \asymp \sin \la_1 u$ in $(0,2\si)$.   
Let the family $\{h_t^{[0,2\si]}: t>0\}$ denote the heat kernel on $[0,2\si]$ with Dirichlet boundary conditions 
and note the following well known estimate (see, for instance, \cite{Z} and the references therein)
\begin{equation} \label{f: est htO1}
h_t^{[0,2\si]}(u,v)
\leq 
\begin{cases}
\ds C\, \min\Big(\frac{\vr(u)\,\vr(v)}{t},1\Big)\,  t^{-1/2}\, \e^{-\mod{u-v}^2/(4t)}	& \quant t \in (0,1] \\
C\, \vr(u)\,\vr(v)\, \e^{-\la_1^2 t} 				& \quant t \in (1, \infty) \\
\end{cases}
\end{equation}
for every $u$ and $v$ in $[0,2\si]$ .	

The \textit{Green operator} $\cGS$ for the slice $\Sisi $ is defined by 
\begin{equation} \label{f: Green}
\cGS 
:= \ioty \cH_t^{\Sisi } \wrt t.  
\end{equation}
It is not hard to prove that given a reasonable function $B$ on $\Sisi $ (for instance 
$B\in C^0(\Sisi)\cap L^r(\Sisi)$ for some $r$ in $(1,\infty)$), the function $\cGS B$ solves the problem 
$$
-\DDelta u = B \quad\hbox{in $\Sisi $}  \qquad u(\cdot,0) = 0 = u(\cdot,2\si)
$$
in the sense of distributions.  
At least formally, off the diagonal of $\Sisi \times \Sisi$ the kernel of $\cGS$  is given by the formula
\begin{equation} \label{f: kernel cGS}
k_{\cGS}\big((x,u),(y,v)\big)
:= \ioty h_t^{N}(x,y) \, h_t^{[0,2\si]}(u,v) \wrt t;
\end{equation}
here $(x,u)$ and $(y,v)$ are in $\Sisi $, $(x,u) \neq (y,v)$.  
We shall consider the operators $\cGS^j$, $j=1,2,\ldots$, and their distributional kernels $k_{\cGS^j}$. 

We \textit{claim} that 
\begin{equation} \label{f: claim Gj}
k_{\cGS^j} \big((x,u),(y,v)\big)
= \frac{1}{(j-1)!} \, \ioty  h_t^N(x,y) \, h_t^{[0,2\si]}(u,v) \, t^{j-1} \wrt t.  
\end{equation}
We argue by induction.  If $j=1$, then \eqref{f: claim Gj} reduces to \eqref{f: kernel cGS}.
Assume that \eqref{f: claim Gj} holds for $j$, and consider $k_{\cGS^{j+1}}$.  
Clearly
$$
\begin{aligned}
	& k_{\cGS^{j+1}} \big((x,u),(z,w)\big)  \\
	& = \int_{\Sisi }  \, k_{\cGS^{j}} \big((x,u),(y,v)\big)\, k_{\cGS} \big((y,v),(z,w)\big) \wrt \nu(y)\wrt v \\
	& = \frac{1}{(j-1)!} \,\int_{\Sisi } \wrt \nu(y)\wrt v \, \ioty\ioty h_s^{[0,2\si]}(u,v) \, h_s^N(x,y) 
	   \, h_t^{[0,2\si]}(v,w) \, h_t^N(y,z) \, s^{j-1}  \wrt s \wrt t  \\
	& = \frac{1}{(j-1)!} \,\ioty\ioty h_{s+t}^{[0,2\si]}(u,w)  \, h_{s+t}^N(x,z) \, s^{j-1} \wrt s \wrt t;
\end{aligned}
$$
we have used the inductive hypothesis in the second equality above and the semigroup property of the heat kernel 
in the third.  Then we perform two subsequent changes of variables: we set $t=\tau s$ in the integral with respect to $t$
and obtain that 
$$
k_{\cGS^{j+1}} \big((x,u),(z,w)\big) 
= \frac{1}{(j-1)!} \,\ioty\ioty h_{s(1+\tau)}^{[0,2\si]}(u,w)  \, h_{s(1+\tau)}^N\big(x,z) \, s^{j} \wrt s \wrt \tau;
$$
then we set $s(1+\tau) = \si$ in the integral with respect to $s$, and the right hand side of the formula above
transforms to 
$$
\frac{1}{(j-1)!} \,\ioty\ioty h_{\si}^{[0,2\si]}(u,w)  \, h_{\si}^N(x,z) \, 
\frac{\si^{j}}{(1+\tau)^{j+1}} \wrt \si \wrt \tau.
$$
Integrating with respect to $\tau$ gives the required formula \eqref{f: claim Gj}, and concludes the proof of the claim.

In Proposition~\ref{p: estimates for cGS} below we establish pointwise estimates for $k_{\cGS^j}$.  
Preliminarily, we determine the order of magnitude of 
$$
J(d) 
:= \int_1^\infty  t^{j-3/2} \, \e^{-(\la_1 \sqrt t- d\sqrt {c/t})^2}  \wrt t
$$
as $d$ tends to infinity.  

\begin{lemma} \label{l: asymp integral}
If $d$ tends to infinity, then for every integer $j\geq 1$ one has $\ds J(d) \asymp d^{j-1}$ 
(i.e., there exist positive constants $C_1$ and $C_2$
such that $C_1\, d^{j-1} \leq J(d) \leq C_2 \, d^{j-1}$).  
\end{lemma}

\begin{proof}
We write $\ds{-(\la_1 \sqrt t- d\sqrt {c/t})^2} = - \sqrt{c}\la_1 d \, \Big(\sqrt{\frac{\la_1 t}{\sqrt{c}d}} 
-\sqrt{\frac{\sqrt{c}d}{\la_1 t}}\Big)^2$, change variables in the integral ($\sqrt{\la_1 t/\sqrt{c}d} = \tau$), and see that 
$$
J(d)
	= 2\, \Big(\frac{\sqrt{c}d}{\la_1}\Big)^{j-1/2}  \, \int_{\sqrt{\la_1/(\sqrt{c}d)}}^\infty  \tau^{2(j-1)} \,
	\e^{d\psi(\tau)}  \wrt \tau, 
$$
where $\psi(\tau) := {-\sqrt{c}\la_1 (\tau- 1/\tau)^2}$. Note that the phase $\psi(\tau)$ has just one critical point at $1$.  
Fix $0<\tau_1<1<\tau_2$. By applying the Laplace method (see, for instance, \cite[formula (2),
p.~37]{E}), one checks that 
$$
\int_{\tau_1}^{\tau_2} \tau^{2(j-1)} \, \e^{d \psi(\tau)}  \wrt \tau
\asymp d^{-1/2}  
$$
as $d$ tends to infinity. Moreover, for $\de > 0$ small enough $\psi(\tau) \leq -\de \tau^2$ in $[\tau_2,\infty)$, so that
$$
\int_{\tau_2}^{\infty} \tau^{2(j-1)} \, \e^{d \psi(\tau)}  \wrt \tau 
\leq\e^{-\de d \tau_2^2/2} \int_{\tau_2}^{\infty} \tau^{2(j-1)} \, \e^{-\de d \tau^2/2}  \wrt \tau 
\leq C\e^{-\de d \tau_2^2/2}.   
$$
Similarly, for $\ga > 0$ small enough $\psi(\tau) \leq -\ga \tau^{-2}$ in $(0,\tau_1)$, so that
$$
\int_{\sqrt{\la_1/(\sqrt{c}d)}}^{\tau_1} \tau^{2(j-1)} \, \e^{d \psi(\tau)}  \wrt \tau 
\leq\e^{-\ga d \tau_1^{-2}/2} \int_{0}^{\tau_1} \tau^{2(j-1)} \, \e^{-\ga d \tau_1^{-2}/2}  \wrt \tau 
\leq C\e^{-\ga d \tau_1^{-2}/2}.   
$$
By combining the estimates above, we see that $J$ has the required asymptotic behaviour at infinity.  
\end{proof}

\begin{proposition} \label{p: estimates for cGS}
Suppose that $j$ is a positive integer and that $n\geq 2$. 
There exists a positive constant $C$ such that the following hold:
\begin{enumerate}
\item[\itemno1]
if $D\big((x,u),(y,v)\big)\geq 2$, then 
$
k_{\cGS^j}\big((x,u),(y,v)\big)
\leq C \, d(x,y)^{j-1} \, \e^{-2\la_1 d(x,y)\sqrt c}$;
\item[\itemno2]
if $D\big((x,u),(y,v)\big)\leq 2$, then 
$$
k_{\cGS^j}\big((x,u),(y,v)\big)
\leq
\begin{cases}
	\ds C\, D^{2\ga}  				& \textrm{if $\ga<0$} \\
	\ds C \,\log \frac{4}{D}		       	& \hbox{if $\ga=0$} \\ 
	\ds C 						& \textrm{if $\ga>0$}, 
\end{cases}
$$
where $\ga := j-(n+1)/2$.  
\end{enumerate}
\end{proposition}

\begin{proof}
We estimate $k_{\cGS^{j}}$ from above by inserting in the integral in \eqref{f: claim Gj}
the estimates for $h_{t}^{[0,2\si]}$ in \eqref{f: est htO1} and the upper bound \eqref{f: assumptions on ht} 
for $h_{t}^N$ (observing that the constant $c$ in \eqref{f: assumptions on ht} is smaller or equal than $1/4$).  Thus, 
$$
k_{\cGS^j} \big((x,u),(z,w)\big) \leq C\,\left[ I \big((x,u),(z,w)\big) + J (x,z)\right], 
$$
where
$$
I \big((x,u),(z,w)\big) 
= \int_0^1 t^{\ga-1}  \e^{-c D^2/t} \wrt t
$$
and
$$
J (x,z) 
= \int_1^\infty  t^{j-3/2} \e^{-\la_1^2 t- c d(x,z)^2/t}  \wrt t.
$$
We estimate $I$ and $J$ separately.  
First, changing variables ($D^2/t = \tau$), we see that 
$$
\begin{aligned}
I
= D^{2\ga} \, \int_{D^2}^{\infty}  \tau^{-\ga} \,  \e^{-c\tau} {\dtt \tau} 
 \asymp 
\begin{cases}
	C \,  D^{2\ga}    		& \hbox{if $\ga<0$} \\ 
	\ds C \log \frac{1}{D}       	& \hbox{if $\ga=0$} \\ 
	C 				& \hbox{if $\ga>0$} 
\end{cases}
\end{aligned}
$$
as $D$ tends to $0$. Furthermore, 
$$
I
\leq C\,D^{2\ga}  \e^{-cD^2/2}  
\leq C \big[1+d(x,y)\big]^{j-1} \, \e^{-2\la_1 d(x,y)\sqrt c}  
$$
when $D\geq 2$.
Concerning $J$,  clearly it tends to a constant as $d$ tends to $0$.  
Now assume that $d$ is large and write 
$\ds -\la_1^2 t- cd^2/t = -\big(\la_1\sqrt t - d \sqrt{c/t}\big)^2 - 2\la_1 d\sqrt c$.  Then, 
$$
\begin{aligned}
J
= \e^{-2\la_1 d\sqrt c} \int_1^\infty  t^{j-3/2} \, \e^{-(\la_1 \sqrt t- d\sqrt{c/t})^2}  \wrt t 
\leq C \, d^{j-1} \, \e^{-2\la_1 d\sqrt c};
\end{aligned}
$$
the inequality above follows from Lemma~\ref{l: asymp integral}.   

The estimates in \rmi\ and \rmii\ follow directly from the analysis above.  
\end{proof}  

\begin{remark} \label{rem: pointwise}
The estimates for $k_{\cGS}$ in Proposition~\ref{p: estimates for cGS} are not best possible.  
In particular, they do not capture the asymptotic behaviour of $k_{\cGS}$ near the boundary of $\Si$.  
We do not insist on this point because such behaviour is not needed in the sequel.
However, for later purposes (see the proof of Lemma~\ref{l: Geta}), we need the following straightforward estimate:
for each $\de >0$ there exists a positive constant $C$ such that 
\begin{equation} \label{f: easy estimate}
k_{\cGS} (X,Y)
\leq C \, \min\big(\vr(u) \, \vr(v), \e^{-2\la_1 d(x,y)\sqrt c}\big)
\quant X,Y\in \Si: D(X,Y) \geq \de.
\end{equation}
Here $X = (x,u)$ and $Y = (y,v)$ and $\vr$ is defined in \eqref{f: definition rho}.

The estimates in Proposition~\ref{p: estimates for cGS} imply that $k_{\cGS} (X,Y) \leq C \, \e^{-2\la_1 d(x,y)\sqrt c}$ for 
every $X$ and $Y$ in $\Si$ such that $D(X,Y) \geq \de$. 

To prove that $k_{\cGS} (X,Y) \leq C \, \vr(u) \, \vr(v)$ in the same range of $X$ and $Y$, 
we insert in the integral in \eqref{f: claim Gj} (with $j=1$)
the estimates for $h_{t}^{[0,2\si]}$ in \eqref{f: est htO1} and the upper bound \eqref{f: assumptions on ht} 
for $h_{t}^N$.  Observe that the assumption $D(X,Y) \geq \de$ implies that 
$
k_{\cGS} (X,Y)
\leq C \,  \big(I_1+I_2+I_3\big), 
$
where
$$
I_1 
:= \int_0^{\vr(u)\vr(v)}   t^{-(n+1)/2} \, \e^{-c\de^2/t}  \wrt t,
\qquad 
I_2
:= \vr(u) \, \vr(v) \, \int_{\vr(t)\vr(u)}^1   t^{-(n+3)/2} \, \e^{-c\de^2/t}  \wrt t 
$$
and 
$$
I_3
:= \vr(u)\vr(v)\, \int_1^\infty  t^{j-3/2} \e^{-\la_1^2 t- c d(x,z)^2/t}  \wrt t.
$$
The required estimate follows directly from this and a straightforward calculation.  
\end{remark}

Next we establish some mapping properties of $\cGS^j$. 
For simplicity, in the sequel we write $\Upsilon$ instead of $\Upsilon_2$ (see \eqref{f: Upsilon R}).  
Denote by $K_j^0: \Sisi \times\Sisi  \to [0,\infty)$ the function defined by
\begin{equation} \label{f: KjO}
K_j^0
= 
\begin{cases}
	\One_{\Upsilon} \, D^{2j-1-n} 	     	& \hbox{if $j <  (n+1)/2$} \\ 
	\ds\One_{\Upsilon} \, \log \frac{4}{D} 	     	& \hbox{if $j =  (n+1)/2$} \\ 
	\One_{\Upsilon}  	     	  			& \hbox{if $j >  (n+1)/2$},
\end{cases}
\end{equation}
and by $\cK_j^0$ the integral operator  with kernel $K_j^0$ acting on functions defined on $N\times \BR$.
For each $\de >0$ denote by $K_\de^\infty:N\times N \to[0,\infty)$ the function defined by 
\begin{equation} \label{f: Kvepinfty}
K_\de^\infty(x,y)
	= \e^{-\de d(x,y)}
\quant (x,y) \in N\times N,
\end{equation}
and by $\cK_\de^\infty$ the integral operator  with kernel $K_\de^\infty$ acting on functions defined on $N$.
Notice that, by Proposition~\ref{p: estimates for cGS}, for every $\de<2\la_1\sqrt c$ there exists a constant
$C$ such that 
\begin{equation} \label{f: reduction}
\begin{aligned}
	\bigmod{\cGS^j F(x,u)}
	& \leq  C\,  \big[\cK_j^0 \mod{F} (x,u) + \cK_\de^\infty \mod{F^\flat}(x)\big] 
	\quant (x,u)\in \Sisi , \\
\end{aligned}
\end{equation} 
where $\ds F^\flat(x)$ is as in \eqref{f: F flat}.   
This observation reduces the proof of estimates for $\cGS^j$
to the proof of similar estimates for $\cK_j^0$ and $\cK_\de^\infty$.  We study the mapping properties of these operators 
in the next proposition.  

\begin{proposition} \label{p: estimates for cGS II}
Suppose that $j$ is a positive integer and that $n\geq 2$.  
The following hold:
\begin{enumerate}
\item[\itemno1]
	if $\cK_j^0$ is bounded from $\lp{\Sisi }$ to $\lr{\Sisi }$  
	and $\cK_\de^\infty$ is bounded from $\lp{N}$ to $\lr{N}$ for some $\de < 2\la_1\sqrt c$, then $\cGS^j$
	is bounded from $\lp{\Sisi }$ to $\lr{\Sisi }$;
\item[\itemno2]
	$\cGS^j$ is bounded on $\lp{\Sisi }$ for all $p$ in $[1,\infty]$;
\item[\itemno3]
	$\cGS$ is bounded from $\lu{\Sisi }$ to weak-$L^{(n+1)/(n-1)}(\Sisi )$ and
	from $\lp{\Sisi }$ to $\lr{\Sisi }$ when $1<p<(n+1)/2$ and $1/r = 1/p- 2/(n+1)$; 
\item[\itemno4]
	if $r>1$, $F$ is in $C^0(\Sisi)\cap\lr{\Sisi }$ and $J > (n+1)/2$, then $\cGS^JF$ is a bounded 
	continuous function on $\Sisi$ and
	$$
	\bignorm{\cGS^JF}{C_b({\Sisi})}
	\leq C \bignorm{F}{\lr{\Sisi }};
	$$
\item[\itemno5]
	$\cGS$ is bounded on $\cB^p$ for each $p$ in $[1,\infty)$ and 
	$\ds \lim_{t\to\partial (0,2\si)} \bignorm{\cG_\Si F(\cdot,t)}{p} = 0$ for every $F$ in $\cB^p$.
\end{enumerate}
\end{proposition}

\begin{proof}
First we prove \rmi.  Formula \eqref{f: reduction} and the assumptions on $\cK_j^0$ and $\cK_\vep^\infty$ imply that
$$
\begin{aligned}
	\bignorm{\cGS^j F}{\lr{\Sisi }} 
	& \leq  C\,  \big[\bignorm{\cK_j^0 F}{\lr{\Sisi }} + \bignorm{\cK_\vep^\infty F^\flat}{r} \big] \\
	& \leq  C\,  \big[\bignorm{F}{\lp{\Sisi }} + \bignorm{F^\flat}{p} \big]  \\
	& \leq  C\bignorm{F}{\lp{\Sisi }};
\end{aligned}
$$
we have used \eqref{f: F flat F} in the last inequality.  

Now we prove \rmii. 
By interpolation, it suffices to prove that $\cGS^j$ is bounded on $\lu{\Sisi }$
and on $\ly{\Sisi }$.  Since $k_{\cGS^j}$ is symmetric, a duality argument shows that it suffices to prove 
that $\cGS^j$ is bounded on $\lu{\Sisi }$.  Now, the boundedness of $\cGS^j$ on $\lu{\Sisi }$ follows from \rmi\
and the boundedness of $\cK_j^0$ on $\lu{\Sisi }$ and of $\cK_\de^\infty$ on $\lu{N}$ for some $\de$ in $(2\be,2\la_1\sqrt c)$.  
To prove this, it suffices to show that for such values of $\de$
\begin{equation} \label{f: luly}
\sup_{Y\in \Sisi } \, \int_{\Sisi } \cK_j^0(X,Y) \wrt \cY(X) < \infty  
\quad\hbox{and}\quad
\sup_{y\in N} \, \int_{N} \cK_\de^\infty(x,y) \wrt \nu(x) < \infty.    
\end{equation}
These estimates can be obtained easily by integrating in polar co-ordinates centred at $Y$ and at $y$, respectively.
We omit the details.  

Now \rmiii\ follows from \rmi\ and the boundedness of
$\cK_j^0$ from $\lp{\Sisi }$ to $\lr{\Sisi }$ and of $\cK_\de^\infty$ from $\lr{N}$ to
$\laq{N}$ for all $q$ in $[r,\infty]$ and $\de$ in $(2\be,2\la_1\sqrt c)$.  
Specifically, $K_\de^\infty$ is bounded, whence $\cK_\de^\infty$
is bounded from $\lu{N}$ to $\ly{N}$.  We have proved in \rmii\ that $\cK_\de^\infty$
is bounded on $\lu{N}$.   Since $K_\de^\infty$ is symmetric, $\cK_\de^\infty$ is also 
bounded on $\ly{N}$.  By interpolation and duality, it follows that $\cK_\de^\infty$
maps $\lp{N}$ to $\laq{N}$ for all $1\leq p\leq q\leq \infty$.  

The proof that $\cK_j^0$ maps $\lu{\Sisi }$ to weak-$L^{(n+1)/(n-1)}(\Sisi )$ and
$\lp{\Sisi }$ to $\lr{\Sisi }$ when $1<p<(n+1)/2$ and $1/r = 1/p- 2/(n+1)$ can be obtained by 
adapting the proof of \cite[Theorem~1, p.~119]{St1}.  We omit the details.  

Now we prove \rmiv.  Notice that for each positive integer $k\leq J$
$$
\bignorm{\DDelta^k \cGS^J F}{\lr{\Sisi }} 
= \bignorm{\cGS^{J-k} F}{\lr{\Sisi }} 
\leq  C \bignorm{F}{\lr{\Sisi }};
$$
the last inequality follows from \rmii.  

For the sake of completeness we give a proof of the continuity of $\cGS^J F$ on $\Si$, which is, we believe, quite standard. 
Suppose that $X\in \Sisi$. Recall that $\cGS^J F$ is a distributional solution of $\DDelta^J V = F$ on $\Sisi$. 
Choose a harmonic coordinate system $(U,\phi_U)$ with $U\subset\Sisi$ open set containing $X$ and 
$\phi_U:U\to\mathbb R^{n+1}$.  For any function $V$ on $U$, set $\wt V := V\circ \phi_U^{-1}$.  Then 
$L^J(\wt{\cGS^J F}) = (\DDelta^J\cGS^J F)\wt{\phantom a} = \wt F$
in the sense of distributions on $U$, 
where $L$ is the elliptic operator defined in $\phi_U(U)$ by $\ds\sum_{i,j} (g\circ \phi_U^{-1})^{ij}\partial^2_{ij}$. 
Since $\wt F$ is continuous, it is in $L^2_{loc}(U)$.  
By elliptic theory (see, for instance, \cite[Theorem 6.33]{Folland}),
$\wt{\cGS F}\in W^{2J,2}_{loc}(U)$.  The latter inclusion, with $n'=n+1$, is a consequence of
local Sobolev embeddings. 
Now, if $2J>(n+1)/2$, then $W^{2J,2}_{loc}(U)$ is contained in $C(U)$, as required to conclude the proof
of the continuity of $\cGS^J F$.  

It remains to prove that $\cGS^JF$ is bounded on $\Sisi$.   
By \rmi, in order to prove that $\cGS^JF$ is bounded, it suffices to prove that $\cK_J^0$ maps $\lr{\Sisi }$
to $\ly{\Sisi }$, and that $\cK_\de^\infty$ maps $\lr{N}$ to $\ly{N}$.  In the proof of \rmiii, we have
already shown that $\cK_\de^\infty$ maps $\lr{N}$ to $\laq{N}$ for all $q$ in $[r,\infty]$ 
(and $\de\in (2\be,2\la_1\sqrt c)$). 
Thus, it remains to consider $\cK_J^0$.  
The kernel $K_J^0$ of $\cK_J^0$ is supported in a neighbourhood of the diagonal in $\Sisi \times\Sisi $, 
and it is bounded (see \eqref{f: KjO}).  Thus, $\cK_J^0$ maps $\lr{\Sisi }$ to $\laq{\Sisi }$
for all $q$ in $[p,\infty]$.  This concludes the proof of \rmiv.  

Finally we prove \rmv.  We already know that $\bignorm{\cK_\de^\infty F^\flat}{p} \leq C \bignorm{F^\flat}{p}$
for every $\de$ in $(2\be,2\la_1\sqrt c)$.
Trivially, $\bignorm{F^\flat}{p} \leq \bignorm{F}{\cB^p}$, whence 
$\bignorm{\cK_\de^\infty F^\flat}{p} \leq C \bignorm{F}{\cB^p}$.

Thus, arguing as in \rmi, it suffices to show that $\cK^0$ is bounded on 
$\cB^p$ for each $p$ in $[1,\infty)$.  Notice that
$$
\begin{aligned}
\bignorm{\cK^0F}{\cB^p}
	& \leq \sup_{t\in (0,2\si)} \Big[\int_N \!\!\wrt \nu(x) \Bigmod{\int_{\Sisi } 
		(\One_{\Upsilon}D^{1-n})\big((x,t), (y,v)\big)\, F(y,v) \wrt\nu(y)\wrt v}^p\Big]^{1/p} \\
	& \leq C\,  \sup_{t\in (0,2\si)} \Big[\int_N \!\!\wrt \nu(x) \Bigmod{\int_{B_2(x)} d^{1-n}(x,y)\, 
    		F^\flat(y) \wrt\nu(y)}^p\Big]^{1/p} \\
	& \leq C\,  \Big[\int_N \bigmod{F^\flat(x)}^p \wrt \nu(x) \Big]^{1/p} \\
	& \leq C \bignorm{F}{\cB^p}:
\end{aligned}
$$
the penultimate inequality follows from the fact that the kernel 
$\One_{V}\, d^{1-n}$, where $V = \big\{(x,y)\in N\times N: d(x,y) \leq 2\big\}$ is symmetric and satisfies 
$\ds \sup_{y\in N} \, \int_N \big[\One_{V}\, d^{1-n}\big](x,y) \wrt\nu(x) <\infty$, whence the 
corresponding integral operator is bounded on $\lp{N}$ for every $p$ in $[1,\infty]$.  

Suppose now that $F$ is in $\cB^p$.  By \eqref{f: Green},
$$
\begin{aligned}
\bignorm{\cG_\Si F(\cdot,t)}{p}
	& \leq \ioty  \bignorm{\cH_s^{\Si} F(\cdot,t)}{p} \wrt s\\
	& \leq \ioty \wrt s\, \int_0^{2\si} h_s^{[0,2\si]}(t,u) \bignorm{\cH_s^{N} F(\cdot,u)}{p} \wrt u.
\end{aligned} 
$$
Now, $\bignorm{\cH_s^{N} F(\cdot,u)}{p} \leq \bignorm{F(\cdot,u)}{p}$, by the contractivity of $\cH_s^N$ on 
$\lp{N}$, whence 
$$
\bignorm{\cG_\Si F(\cdot,t)}{p}
\leq  \bignorm{F}{\cB^p}\, \int_0^{2\si} \ioty h_s^{[0,2\si]}(t,u) \wrt s \wrt u.     
$$
Now, the pointwise estimates \eqref{f: est htO1} imply that 
$$
\int_1^\infty  h_s^{[0,2\si]}(t,u) \wrt s     
\leq C\, \vr(t)\, \vr(u) \, \int_1^\infty \e^{-\la_1^2 s} \wrt s 
\leq C\, \vr(t) 
\quant u \in (0,2\si). 
$$
and
$$
\begin{aligned}
\int_0^1  h_s^{[0,2\si]}(t,u) \wrt s     
	& \leq C \,  \int_0^{\vr(t)\vr(u)}   \e^{-\mod{t-u}^2/(4s)}  \frac{\wrt s}{\sqrt s}  
	 + C\, \vr(t) \, \vr(u) \, \int_{\vr(t)\vr(u)}^1    \e^{-\mod{t-u}^2/(4s)}  \frac{\wrt s}{s^{3/2}} \\
	& \leq C \,  \int_0^{\vr(t)\vr(u)}   s^{-1/2} \wrt s  
	 + C\, \vr(t) \, \int_{\vr(t)\vr(u)}^1   s^{-3/2}  \wrt s \\
	& \leq C\, \Big(\frac{\vr(t)}{\vr(u)}\Big)^{1/2} 
\quant u \in (0,2\si). 
\end{aligned}
$$
By combining the estimates above, we see that $\bignorm{\cG_\Si F(\cdot,t)}{p} \leq C\, \vr(t)^{1/2} \, \bignorm{F}{\cB^p}$,
which tends to $0$ as $\vr(t)$ tends to $0$, as required.  

This concludes the proof of \rmv, and of the proposition.
\end{proof}

\begin{remark}
Using \rmiii, it is straightforward to see that the assumption $F\in C^0(\Sisi)$ in \rmiv\ can be skipped up to 
choosing larger $J$.
\end{remark}

\section{Maximal inequalities}
\label{s: Maximal}

The purpose of this section is to prove Theorem~\ref{t: maximal estimate}, which contains an analogue for the slice $\Sisi $
of certain maximal inequalities that Dindo\v s \cite[Section~10]{Di} proved in a compact setting.  
We emphasize that our result, Theorem~\ref{t: maximal estimate}~\rmi, is concerned with maximal operators of \textit{harmonic}
functions on $\Sisi$, whereas Dindo\v s proved a similar estimate for generic functions.  
Our additional assumption of harmonicity allows us to use Harnack's inequality in our proof,
thereby simplifying Dindo\v s' argument.  

We need the following notation.
Suppose that $\al$ is a (small) real number.  For $z$ in~$N$, denote by $\Ga_\al(z)$
the subset of $\Sisi $, symmetric with respect to the ``line'' $t=\si$ and whose restriction to the slice
$N\times (0,\si]$ is the cone $\{(x,u)\in N\times (0,\si]: d(x,z) \leq \al \, u\}$.  
We say that a function $F$ on $\Sisi $ is \textit{symmetric} if $F(\cdot, u) = F(\cdot, 2\si-u)$
for every $u$ in $(0,2\si)$.  
Given a nonnegative symmetric function $F$ on $\Sisi $, denote by $F^*$ its nontangential maximal function, defined by 
\begin{equation} \label{f: nontangential max F}
F^*(z)
:= \sup_{(x,u) \in \Ga_\al(z)} \, F(x,u)
\quant z \in N.  
\end{equation} 

\begin{theorem} \label{t: maximal estimate}
Suppose that $p$ is in $(1,\infty)$, $\al$ is small enough and $j$ is a positive integer.  
Then there exists a constant $C$ such that the following hold:
	\begin{enumerate}
		\item[\itemno1]
			$\!\bignorm{\big(\cGS^j H\big)^*}{p}
			\leq C \!\bignorm{H^*}{p}$
			for every positive symmetric harmonic function $H$ on~$\Sisi $;   
		\item[\itemno2]
			if $J> (n+1)/2$, then
			$\bignorm{\big(\cGS^J S\big)^*}{p} \leq C \bignorm{S}{\lp{\Sisi }}$
			for every nonnegative symmetric function~$S$ on $\Sisi $.
	\end{enumerate}
\end{theorem}

\begin{proof} 
We prove \rmi\ in the case where $j\leq \lfloor (n+1)/2\rfloor$.  The modifications
needed to cover the case where $j > \lfloor (n+1)/2\rfloor$ are 
straightforward and are left to the interested reader.   Simply, one needs to use different local estimates
for $k_{\cGS^j}$, depending on the dimension $n$ (see Proposition~\ref{p: estimates for cGS}~\rmii).  

By \eqref{f: reduction} it is enough to estimate $(\cK_j^0 H)^\ast$ and 
$\ds\sup_{(x,u)\in \Ga_\al(z)} \cK_\de^\infty H^\flat(x)$ when $\de$ is in $(2\be,2\la_1\sqrt c)$.
Notice that, for $z$ in $N$
$$
\begin{aligned}
\sup_{X\in \Ga_\al(z)} \cK_\de^\infty H^\flat(x)
&\leq  C\, \sup_{x\in B_{\al\si}(z)}\, \int_{N} \, \e^{-\de d(x,y)} \,H^\flat(y) \wrt \nu(y)\\
&\leq  C\, \int_{N} \, \e^{-\de d(z,y)} \,H^\flat(y) \wrt \nu(y)\\
&= C \, \cK_{\de}^\infty H^\flat(z),
\end{aligned}
$$
where $X = (x,u)$.
Since $\cK_{\de}^\infty$ is bounded on $\lp{N}$ for every $p$ in $[1,\infty]$,
\begin{equation} \label{f: cGSinfty}
\bignorm{\sup_{X\in \Ga_\al(\cdot)} \cK_\de^\infty H^\flat(x)}{p} 
\leq C \bignorm{H^\flat}{p} 
\leq C \bignorm{H}{\lp{\Sisi }}.
\end{equation}
We now estimate the maximal operator $ (\cK_j^0 H)^\ast$.  In this proof for notational convenience we shall write
$\ga$ instead of $j-(n+1)/2$.
For $z$ in $N$, consider the set $Q(z) := B_3(z) \times (0,2\si)$, which is contained in $\Sisi $.  
By \eqref{f: KjO}, for any $X\in \Ga_\al(\cdot)$ with $\alpha$ small enough, 
\begin{equation}\label{f: est Kj0H}
\cK_0^j H(X)
\leq C \, \int_{Q(z)} \, D(X,Y)^{2\ga} \, H(Y) \wrt \cY(Y).  
\end{equation}
It is convenient to split the integral above as the sum of the integrals over $\Ga_{2\al}(z)$ (which is contained
in $Q(z)$ as long as $\al$ is small enough), and $Q(z) \cap \Ga_{2\al}(z)^c$.  

First we estimate $\ds \int_{\Ga_{2\al}(z)} \, D(X,Y)^{2\ga} \, H(Y) \wrt \cY(Y)$.  Recall that $X$ is in $\Ga_\al(z)$.  
Since $Y=(y,v)$ belongs to $\Ga_{2\al}(z)$, the point $Y$ is in the ball with centre $(z,v)$ and radius $3\al v$.   
We choose $\al$ so small that $B_{6\al v}(z,v)$ (this ball is in $N\times \BR$) is contained in~$\Sisi $.  
Since~$H$ is harmonic, by Harnack's inequality (apply, for instance, \cite[Theorem~5.4.3]{SC} with 
$M = N\times \BR$ and $\de= 1/2$.  Note that, under our assumptions, $N$ supports a local Poincar\'e inequality; 
see for instance \cite[Theorems 1.1]{MSC}), there exists a constant~$C$, independent of $Y$ in 
$\Ga_{2\al}(z)$ and of $z$ in $N$, such that $H(Y) \leq C\, H(z,v)$, whence  
$$ 
\begin{aligned}
	\int_{\Ga_{2\al}(z)} \, D(X,Y)^{2\ga} \, H(Y) \wrt \cY(Y)
	& \leq C\, \int_{\Ga_{2\al}(z)} \, D\big(X,Y\big)^{2\ga} \, H(z,v) \wrt \nu(y) \wrt v \\
	& \leq C\, H^*(z)\, \int_{\Ga_{2\al}(z)} \, D\big(X,Y\big)^{2\ga}  \wrt \cY(Y).
\end{aligned}
$$
Observe that $\Ga_{2\al}(z)$ is contained in $B_R(X)$ for $R$ big enough (depending on $\si$ and $\al$).  
Therefore the last integral
is dominated by $\ds \int_{B_R(X)} \, D\big(X,Y\big)^{2\ga}  \wrt \cY(Y)$, which is bounded with respect to 
$X$ in $\Ga_\al(z)$ as a straightforward integration in polar coordinates shows.  This implies that 
\begin{equation} \label{f: est on Gadueal}
\sup_{X\in \Ga_{\al}(z)} \,\, \int_{\Ga_{2\al}(z)} \, D(X,Y)^{2\ga} \, H(Y) \wrt \cY(Y)
\leq C\, H^*(z).  
\end{equation}

Next we estimate $\ds \int_{Q(z) \setminus \Ga_{2\al}(z)} \, D(X,Y)^{2\ga} \, H(Y) \wrt \cY(Y)$,
where $X=(x,u)$ is in $\Ga_\al(z)$.   
Set $Z := (z,u)$.  We \textit{claim} that, for every $Y\in Q(z) \setminus \Ga_{2\al}(z)$, 
\begin{equation} \label{f: DXY DYZ}
	D(X,Y) \geq \Big(1-\frac{\sqrt{4\al^2+1}}{2} \Big) \, \, D(Y,Z).
\end{equation}
To prove the claim, first observe that, by the triangle inequality, 
$$
\begin{aligned}
D(X,Y) 
	& \geq D(Y,Z) - D(X,Z) \\
	& =    D(Y,Z) \, \Big(1-\frac{\sqrt{4\al^2+1}}{2} \Big) +  D(Y,Z) \,\frac{\sqrt{4\al^2+1}}{2} - D(X,Z).   
\end{aligned}
$$
Thus, in order to prove the claim it suffices to show that 
$$
D(Y,Z) \geq \,\frac{2}{\sqrt{4\al^2+1}} \,\,  D(X,Z).   
$$
Denote by $\OV{W}$ any of the points on $\partial \Ga_{2\al}(z)$ that realises the distance from $Z$ to $\Ga_{2\al}(z)^c$.  
Elementary geometric considerations show that $D(\OV{W},Z) = u \sin \te'$, where $\te'$ denotes half the aperture
of $\Ga_{2\al}(z)$, i.e., $\tan \te' = 2\al$.  It is straightforward to check that $\ds \sin \te' = \frac{2\al}{\sqrt{4\al^2+1}}$.  
By combining these formulae, we get that 
$$
D(Y,Z)
\geq D(\OV{W},Z) 
=    \al\, u\,\, \frac{2}{\sqrt{4\al^2+1}} 
\geq  D(X,Z) \,\, \frac{2}{\sqrt{4\al^2+1}},
$$
as required to complete the proof of the claim.  
Notice that $\ds\frac{2}{\sqrt{4\al^2+1}}>1$, provided that~$\al$ is small enough ($\al< \sqrt 3/2$ will do).    

The claim implies that 
$$
\begin{aligned}
	\int_{Q(z) \setminus \Ga_{2\al}(z)} \, D(X,Y)^{2\ga} \, H(Y) \wrt \cY(Y)
	& \leq C \, \int_{Q(z) \setminus \Ga_{2\al}(z)} \, D(Z,Y)^{2\ga} \, H(Y) \wrt \cY(Y) \\
	& \leq C \, \int_{B_3(z)} \, d(y,z)^{2\ga} \, H^\flat(y)\wrt \nu(y),
\end{aligned}
$$
where the constant $C$ depends on $\al$ and $n$.  Therefore 
$$
\sup_{X\in \Ga_\al(z)} \int_{Q(z) \setminus \Ga_{2\al}(z)} \, D(X,Y)^{2\ga} \, H(Y) \wrt \cY(Y) 
\leq C  \int_{B_3(z)} d(y,z)^{2\ga} \, H^\flat(y)\wrt \nu(y).
$$
By combining this and \eqref{f: est on Gadueal}, recalling also \eqref{f: est Kj0H}, we see that  
$$
\big(\cK_0^j H\big)^* (z)
\leq C\, \Big[H^*(z) + \int_{B_3(z)} \, d(y,z)^{2\ga} \, H^\flat(y)\wrt \nu(y) \Big],
$$
so that 
$$
\begin{aligned}
\bignorm{\big(\cK_0^j H\big)^*}{p}
	& \leq C \bignorm{H^*}{p} + C \Bignorm{\int_{B_3(\cdot)} \, d(\cdot,y)^{2\ga} 
	   \, H^\flat(y)\wrt \nu(y)}{p} \\
	& \leq C\, \big[\bignorm{H^*}{p} +  \bignorm{H^\flat}{p}\big] \\
	& \leq C \bignorm{H^*}{p},
\end{aligned}
$$
as required.   

Next we prove \rmii.  By Proposition~\ref{p: estimates for cGS II}~\rmiv\ and Proposition~\ref{p: estimates for cGS}~\rmi-\rmii,
the function $\cGS^JS$ is continuous and for each $\de$ in $(2\be,2\la_1\sqrt c)$ there exists a constant $C$ such that
$k_{\cGS^J} (X,Y) \leq C\, \e^{-\de D(X,Y)}$ for every $X$ and $Y$ in $\Sisi $.
It is straightforward to check that there exists a constant $C$ such that
$$
\sup_{X\in \Ga_\al(z)} \, \e^{-\de D(X,Y)}
\leq C\,  \e^{-\de d(z,y)}
\quant z\in N \quant Y \in \Sisi .   
$$
Here $y$ is the component in $N$ of the point $Y$ in $\Sisi $.
Consequently,  $(\cGS^J S)^* \leq C\, \cK_{\de}^\infty S^\flat$, whence 
$$
\bignorm{(\cGS^J S)^*}{p}
\leq C \bignorm{\cK_{\de}^\infty S^\flat}{p}
\leq C \bignorm{S^\flat}{p}
\leq C \bignorm{S}{\lp{\Sisi }},
$$
as required.  
\end{proof}  

\section{The function $G$}  
\label{s: G}

In this section we adapt some ideas of Dindo\v s to our case
(see \cite[Chapter~6]{Di}, especially Proposition~6.4 therein).  
The main result of this section is Theorem~\ref{t: analogue SW D} below, 
which is a counterpart in our setting of a classical result of Stein and Weiss. 
First we need a of technical lemma.  Recall the space~$\cB^p$, introduced just above Proposition~\ref{p: estimates for cGS II}.   

\begin{lemma} \label{l: Geta}
Suppose that $p$ is in $(1,\infty)$, and that $F$ is a nonnegative continuous function in $\cB^p$ 
satisfying $\ds \lim_{d(x,o) \to \infty}\, \sup_{t\in [\eta,2\si-\eta]} F(x,t) = 0$ for every $\eta\in (0,\si)$.  
Suppose further that for some constant $\al$ the function $G := F-\al \, \cGS F$ is subharmonic in $\Si$.  Then
the following hold: 
\begin{enumerate}
	\item[\itemno1]
		there exists a sequence $\{\vep_k\}$ such that $\vep_k \to 0$ 
		as $k$ tends to infinity, and a nonnegative function $h$ in $L^p(N)$ such that 
		$\ds w-\lim_{k\to \infty}G(\cdot,\vep_k) = h$ (weak limit in $L^p(N)$) and 
		$\bignorm{h}{p} \leq \min \big(\bignorm{F}{\cB^p},\bignorm{G}{\cB^p}$\big);
	\item[\itemno2]
		$G \leq \cPSeta \big[G(\cdot,\eta)\big]$ in $\Sisieta$;  
	\item[\itemno3]
		$G \leq \cPS h$ in $\Sisi$,  where $h$ is as in \rmi. 
\end{enumerate}
\end{lemma}

\begin{proof}
First we prove \rmi.  
By the weak compactness of the unit sphere of $L^{p}(N)$, there exists a sequence $\vep_k$,
which tends to $0^+$ as $k$ tends to infinity, such that $F(\cdot,\vep_k)$ is weakly convergent in $\lp{N}$ 
to a function, $h$ say.  
By abstract nonsense
$
\ds\bignorm{h}{p} 
\leq \sup_k \bignorm{F(\cdot,\vep_k)}{p}.  
$
Furthermore, $h$ is nonnegative, because so is~$F$
by assumption.   

By Proposition~\ref{p: estimates for cGS II}~\rmv, $\bignorm{\cG_\Si F(\cdot,\vep_k)}{p}$
tends to $0$ as $k$ tends to infinity.  \textit{A fortiori}
$\{\cG_\Si F(\cdot,\vep_k)\}$ tends to $0$ weakly in $\lp{N}$.  Thus, 
$w$-$\ds\lim_{k\to \infty} \, G(\cdot,\vep_k)=h$ in $\lp{N}$, whence, by abstract nonsense, 
$
\ds\bignorm{h}{p} 
\leq \sup_k \bignorm{G(\cdot,\vep_k)}{p}.  
$

Next we prove \rmii.
By elliptic regularity, $\cGS F$ is continuous on $\Sisi$, for $F$ is continuous therein by assumption.  
Consequently so is $G$.  
For the sake of completeness we give a proof of the continuity of $\cGS F$, which is, we believe, quite standard. 
Suppose that $X\in \Sisi$.  Recall that $\cGS F$ is a distributional solution of $\DDelta V = F$ on $\Sisi$. 
Choose a harmonic coordinate system $(U,\phi_U)$ with $U\subset\Sisi$ open set containing 
$X$ and $ \phi_U:U\to\mathbb R^{n+1}$.  For any
function $V$ on $U$, set $\wt V := V\circ \phi_U^{-1}$.  Clearly, $V$ is continuous if and only if $\wt V$ is. Then 
\[
L\wt{\cGS F} = \wt F
\]
in the sense of distributions on $U$, 
where $L$ is the elliptic operator defined in $\phi_U(U)$ by $\ds\sum_{i,j} (g\circ \phi_U^{-1})^{ij}\partial^2_{ij}$. 
Since $\wt F$ is continuous, it is in $L^p_{loc}(U)$ for every $p\in[1,\infty)$. 
By elliptic theory (see, for instance, \cite[Theorem 6.33]{Folland}),
$\wt{\cGS F}\in W^{2,2}_{loc}(U)\subset W^{1,\frac{2n'}{n'-2}}(U)$. 
The latter inclusion, with $n'=n+1$, is a consequence of
local Sobolev embeddings. Moreover the Euclidean Laplacian $\DDelta_0$ of $\wt{\cGS F}$ satisfies
\[
|\DDelta_0\wt{\cGS F}|\leq C\, |L\wt{\cGS F}| = C\,\wt F \in L^{\frac{2n'}{n'-2}}_{loc}(U).
\]
By a local Euclidean Calder\'on--Zygmund inequality, 
we obtain that $\wt{\cGS F}\in W^{2,\frac{2n'}{n'-2}}_{loc}(U)$. Indeed, from \cite[Theorem 9.9]{GT} there exists a function 
$w\in W^{2,\frac{2n'}{n'-2}}_{loc}$ solving $Lw = \wt F$ in a neighbourhood of $X$. 
Since $w-\wt{\cGS F}$ solves $L(w-\wt{\cGS F}) = 0$ and is thus smooth, 
we get that also $\wt{\cGS F}\in W^{2,\frac{2n'}{n'-2}}_{loc}$. 
In particular, $\wt{\cGS F}\in W^{1,\frac{2n'}{n'-4}}_{loc}(U)$ by a local 
Sobolev embedding. Since $\DDelta_0\wt{\cGS F}\in L^{\frac{2n'}{n'-4}}_{loc}(U)$, we can iterate the argument, 
thus obtaining that $\wt{\cGS F}\in W^{2,\frac{2n'}{n'-2k}}_{loc}(U)$ 
for every positive integer $k$ such that $2k<n'$. As soon as $2k>n'-4$, 
by a local Sobolev embedding $W^{2,\frac{2n'}{n'-2k}}_{loc}\subset C^0$,
thereby concluding the proof of the continuity of $\cGS F$. 

To prove \rmii, first notice that both sides of the desired inequality are continuous on $N\times [\eta,2\si-\eta]$ 
(the continuity of the right hand side follows from Proposition~\ref{p: Dirichlet problem slice}), the left hand side
and the right hand side are subharmonic and harmonic in $N\times(\eta, 2\si-\eta)$, respectively.   

We claim that $G(\cdot,\eta)$ is in $C_0(N)$.   
By assumption, $F(\cdot,\eta)$ is in $C_0(N)$.  Thus, it remains to prove that $\cGS F(\cdot,\eta)$ is in $C_0(N)$.  
Suppose that $\vep>0$.  Choose $\ga$ in $(0,\eta)$
so that $\vr(\ga) < \vep$ (recall that $\vr(\ga) = \sin \la_1\ga$, see the beginning of Section~\ref{s: Green}; 
clearly it suffices to choose $\ga<\vep/\la_1$).  
The estimate \eqref{f: easy estimate} implies that there exists a constant $C$ such that 
$k_{\cGS}(X,Y) \leq C \, \min\big(\vr(v), \e^{-2\la_1d(x,y)\sqrt c} \big)$ whenever 
$X=(x,\eta)$, $Y:= (y,v)$ and $D(X,Y) \geq \eta-\ga$.   Thus, in particular, $k_{\cGS}(X,Y) < C\, \vep$ 
if $Y:= (y,v)$ belongs either to $N\times (0,\ga]$ or to $N\times [2\si-\ga,2\si)$ (this just because 
$D(X,Y) \geq \eta-\ga >0$).  
Therefore
$$
k_{\cGS}(X,Y)
= k_{\cGS}(X,Y)^\de \, k_{\cGS}(X,Y)^{1-\de}
\leq C\, \vep^\de\, \e^{-2(1-\de)\la_1d(x,y)\sqrt c}
$$
for any $\de$ in $(0,1)$.  Therefore, if $\de$ is small enough, then $\tau:=2(1-\de)\la_1\sqrt c>2\be$ and  
\begin{equation}  \label{f: CO I}
\begin{aligned}
\int_{\Si\setminus \Si_\ga} k_{\cGS}(X,Y)\, F(Y) \wrt \cY(Y)
	& \leq C\, \vep^\de\,  \int_{\Si\setminus\Si_\ga} \e^{-\tau d(x,y)} \, F(Y) \wrt \cY(Y) \\
	& \leq C\, \vep^\de\,  \Big[\int_{\Si} \e^{-\tau p'd(x,y)}\wrt \cY(Y)\Big]^{1/p'} \bignorm{F}{\lp{\Si}} \\
	& \leq C\, \vep^\de \bignorm{F}{\cB^p}.  
\end{aligned}
\end{equation}
Furthermore, by assumption, there exists $R>0$ such that $F(y,u) < \vep$ when $(y,u)$ belongs to 
$B_R(o)^c\times [\ga,2\si-\ga]$.  Hence
\begin{equation}  \label{f: CO II}
	\int_{B_R(o)^c\times [\ga,2\si-\ga]} k_{\cGS}(X,Y)\, F(Y) \wrt \cY(Y)
	\leq C\, \vep;   
\end{equation}
we have used Proposition~\ref{p: estimates for cGS II}~\rmii\ in the last inequality. 
Finally, if $Y$ belongs to $B_R(o)\times [\ga,2\si-\ga] =: Q_{\ga,R}$ and $d(x,o)$ is big enough, then 
there exists a constant $C$ such that $k_{\cGS}(X,Y) \leq C\, \e^{-2\la_1d(x,y)\sqrt c }$ 
(see Proposition~\ref{p: estimates for cGS}~\rmi).  This and the fact that  
$d(x,y) \geq d(x,o) - R$ whenever $d(x,o)$ is large and $y$ belongs to $B_R(o)$ imply that 
$$
\int_{Q_{\ga,R}} k_{\cGS}(X,Y)\, F(Y) \wrt \cY(Y)
\leq C\, \e^{C R}\, \e^{-2\la_1d(x,o)\sqrt c } \int_{Q_{\ga,R}}F(Y) \wrt\cY(Y).
$$
By H\"older's inequality 
$$
\int_{Q_{\ga,R}}F(Y) \wrt\cY(Y) 
	\leq C\, \nu\big(B_R(o)\big)^{1/p'} \,
	   \Big[\int_{Q_{\ga,R}}F(Y)^p \wrt\cY(Y)\Big]^{1/p}  
	\leq C\, \e^{C R}\, \bignorm{F}{\cB^p}.
$$
Thus, we may conclude that 
\begin{equation}  \label{f: CO III}
	\begin{aligned}
		\int_{Q_{\ga,R}} k_{\cGS}(X,Y)\, F(Y) \wrt \cY(Y)
		& \leq C\, \e^{C R}\, \e^{-2\la_1d(x,o)\sqrt c } \bignorm{F}{\cB^p},
	\end{aligned}
\end{equation}
By combining \eqref{f: CO I}, \eqref{f: CO II} and \eqref{f: CO III}, we see that  
$$
\cGS F(x,\eta)
\leq C\, \vep \, \big(\bignorm{F}{\cB^p} + 1\big) +  C\, \e^{C R}\, \e^{-2\la_1d(x,o)\sqrt c } \bignorm{F}{\cB^p}. 
$$
Now we take the limit of both sides as $d(x,o)$ tends to infinity, and obtain that 
$$
\lim_{d(x,o)\to \infty} \, \cGS F(x,\eta)
\leq C\, \vep \, \big(\bignorm{F}{\cB^p} + 1\big),
$$
from which the claim follows directly.  

Note that 
\begin{equation} \label{f: claim O}
\lim_{d(x,o) \to \infty}\,  \sup_{t\in [\eta, 2\si-\eta]}
\cPSeta \big[G(\cdot,\eta)\big](x,t)
= 0.
\end{equation} 
Indeed, since $G(\cdot,\eta)$ is in $C_0(N)$, for every $\vep>0$ there exists $R$ such that 
$G(x,\eta) < \vep$ for every $x$ such that $d(x,o) >R$.  Then 
$$
\bigmod{\big[\cPSeta G(\cdot,\eta)\big](x,t)}
\leq \Bigmod{\int_{{B}_R(o)} \!\! k_{\Mteta  (\cD)}(x,y) \, G(y,\eta) \wrt \nu(y)} 
	+ \vep \int_{{B}_R(o)^c} \!\! \bigmod{k_{\Mteta  (\cD)}(x,y)} \wrt \nu(y).
$$
The operator $\Mteta  (\cD)$ satisfies on the slice $\Sisieta$ estimates similar to those of
$\Mt (\cD)$ on $\Sisi$.  The proofs of such estimates for 
$\Mteta  (\cD)$ are almost \textit{verbatim} the same as the corresponding proofs for $\Mt (\cD)$.   
In particular, for each $\vep>0$, there exists a positive constant $C$ such that for every function $f$ in $\lu{N}$
with support contained in $B_R(o)$ 
$$
\sup_{t\in (\eta,2\si-\eta)} \bigmod{\Mteta  (\cD) f(x)}
\leq C\, \e^{(\vep-\la_1^\eta) d(x,o)}  \bignorm{f}{1}
$$
for every $x$ in $B_{2R}(o)^c$, where $\ds \la_1^\eta = \frac{\pi}{2(\si-\eta)}$ (see the proof of 
Lemma \ref{l: Mt sqrt cL} \rmiii).  Therefore
for every $x$ in $B_{2R}(o)^c$ we have the estimate 
$$
\sup_{t\in (\eta,2\si-\eta)} \, \Bigmod{\int_{{B}_R(o)} \!\! k_{\Mteta  (\cD)}(x,y) \, G(y,\eta) \wrt \nu(y)}
\leq C\, \e^{(\vep-\la_1^\eta) d(x,o)} \bignorm{\One_{{B}_R(o)}\,G(\cdot, \eta)}{1}, 
$$
which tends to $0$ as $d(x,o)$ tends to infinity.  Furthermore, 
$$
\int_{{B}_R(o)^c} \!\! \bigmod{k_{\Mteta  (\cD)}(x,y)} \wrt \nu(y)
\leq \bignorm{k_{\Mteta  (\cD)}(x,\cdot)}{1} 
\leq C\, \big[1+\bignorm{M^\eta_{\si}}{\cE(\bS_{\vp})}\big];
$$
the last inequality follows from Lemma~\ref{l: Mt sqrt cL}~\rmiv\ (with $p=\infty$).  
The right hand side is independent of $t$ in $(\eta,2\si-\eta)$.  By combining the estimates above, we get that 
$$
\lim_{\mod{x} \to \infty} \sup_{t\in [\eta,2\si-\eta]}
\bigmod{\big[\cPSeta G(\cdot,\eta)\big](x,t)}
\leq \vep,
$$
which, of course, implies the required estimate \eqref{f: claim O}.  

Now, consider the function $\Xi (x,t) := G(x,t) - \big[\cPSeta G(\cdot,\eta)\big](x,t)$, which is continuous
on $\OV{\Si}_{\eta}$ (by Proposition~\ref{p: Dirichlet problem slice}).   
Notice that $\Xi(x,\eta) = 0 = \Xi(x,2\si-\eta)$ for every $x$ in $N$.
Since $G$ is subharmonic on $\Sisi $ and $\cPSeta \big[G(\cdot,\eta)\big]$ is harmonic on $\Sisieta$, $\Xi$
is subharmonic on $\Sisieta$.   
For $R>0$ consider the compact set $K_R := \OV{B}_R(o) \times [\eta,2\si-\eta]$.  Fix $\vep >0$. 
Our assumptions and \eqref{f: claim O} yield  
$
\ds \sup_{(x,t) \in \partial K_R} \Xi(x,t) 
< \vep
$
for $R$ large enough.  By the maximum principle for subharmonic functions (see, for instance, 
\cite[Corollary 1, p. 479]{GGW}) applied to $\Xi$ and $K_R$, we have the estimate $\Xi \leq \vep$
on $K_R$.  By letting $\vep$ tend to $0$ (and $R$ to infinity), we may conclude that $\Xi \leq 0$ on $\Sisieta$,
as required.  

Finally, we prove \rmiii.  It suffices to show that 
\begin{equation} \label{f: limit III}
\lim_{k\to \infty} \, \big[\cPSvepk G(\cdot,\vep_k)\big](x,t)
= \cPS h(x,t) 
\end{equation}
for almost every $(x,t)$ in $\Sisi $, where $\{\vep_k\}$ denotes (possibly a subsequence of)
the sequence whose existence is established in \rmi.  
Indeed, this and \rmii\ would imply that $G(x,t) \leq \cPS h(x,t)$ for almost
every $(x,t)$ in $\Sisi$.  Since both $G(x,t)$ and $\cPS h(x,t)$ are continuous functions 
on $\Sisi$, the latter equality would hold everywhere, as required.  

In order to prove~\eqref{f: limit III}, 
we consider preliminarily the function $m_t^\vep (z) := \Mt ^{\vep} (z) - \Mt (z)$ for $\vep >0$ and $t$ in $(0,\si]$.
We \textit{claim} that for each $\vp$ in $(0,\pi/2)$    
\begin{equation} \label{f: claim limit}
\lim_{\vep\downarrow 0} \sup_{z\in \OV{\bS}_\vp} \, \bigmod{m_t^\vep (z)} = 0.
\end{equation}
Here $\bS_\vp$ and $\OV{\bS}_\vp$ denote the sector $\{z\in \BC: \mod{\arg z} <\vp \}$ and its closure,
respectively. 

A straightforward computation shows that given $\vp$ in $(0,\pi/2)$
and $t$ in $(0,\si]$ there exists a constant $C$ such that  $\ds \sup_{z\in \OV{\bS}_\vp} \, \bigmod{m_t^\vep (z)} \leq C$
for every $\vep\leq t/2$.
By the Phragmen--Lindel\"of principle 
$\ds \sup_{z\in \OV{\bS}_\vp} \, \bigmod{m_t^\vep (z)} \leq \sup_{z\in \partial\OV{\bS}_\vp} \, \bigmod{m_t^\vep (z)}$.  
Observe that 
\begin{equation}
\label{f: mteps}
\begin{aligned}
\bigmod{m_t^\vep (r\e^{i\vp})}
	& = \bigmod{\cosh[(\si-t)r\e^{i\vp}]} \,\,  \Bigmod{\frac{1}{\cosh[\si r\e^{i\vp}]} 
	     - \frac{1}{\cosh[(\si-\vep)r\e^{i\vp}]}} \\
	& \leq \bigmod{\cosh[(\si-t)r\e^{i\vp}]} \,\, r\vep\, 
	     \sup_{u\in (\si-\vep,\si)} \frac{\bigmod{\sinh(ur\e^{i\vp})}}{\bigmod{\cosh^2[ur\e^{i\vp}]}} \\ 
	& \leq C \vep r \, \e^{(\vep -t)r\cos\vp};
\end{aligned}
\end{equation}
the first inequality follows from the mean value theorem, applied to the function $u \mapsto 1/\cosh[ur\e^{i\vp}]$.   
Now we take the supremum of both sides with respect to $r$ in $(0,\infty)$, and obtain
$$
\sup_{r>0} \, \bigmod{m_t^\vep (r\e^{i\vp})} 
\leq \frac{C\vep}{(t-\vep) \cos \vp}
\to 0
$$ 
as $\vep$ tends to $0$. 
Since $m_t^\vep (\OV{z}) = \OV{m_t^\vep (z)}$, we can conclude that 
$\ds \lim_{\vep\downarrow 0} \sup_{z\in \partial\OV{\bS}_\vp} \, \bigmod{m_t^\vep (z)} = 0$, thereby concluding the proof
of the claim.

By using \eqref{f: mteps}, it is easy to check that the function 
$m_t^\vep$ belongs to the algebra $H_0^\infty(\bS_\vp)$, which is included 
in the Dunford class $\cE(\bS_\vp)$ (see page \pageref{pp: Dunfors} for the definitions).  Thus, if $\pi/4<\vp<\pi/2$,
then the natural functional calculus \cite[Theorem~2.3.3]{Haa} implies that   
\begin{equation} \label{f: est mtvepk}
\bigopnorm{m_t ^{\vep_k} (\cD)}{p} 
\leq C \bignorm{m_t^\vep}{\cE(\bS_\vp)}
=    C \bignorm{m_t^\vep}{H_0^\infty(\bS_\vp)},  
\end{equation}
because $\cD$ is sectorial of angle $\pi/4$.  Write  
\begin{equation} \label{f: limit IV}
\begin{aligned}
\Mt ^{\vep_k} (\cD)G(\cdot, \vep_k) - \Mt (\cD) h 
= m_t ^{\vep_k} (\cD) G(\cdot,\vep_k) + \Mt (\cD) \big[G(\cdot,\vep_k)-h\big]. 
\end{aligned}
\end{equation}
In order to prove \eqref{f: limit III},
it suffices to show that both 
summands on the right hand side of \eqref{f: limit IV} tend to $0$ pointwise a.e.
Since $\{G(\cdot,\vep_k)\}$ is  uniformly bounded in $\lp{N}$, 
$$
\bignorm{m_t^{\vep_k}(\cD) G(\cdot,\vep_k)}{p}
\leq C \, \bigopnorm{m_t^{\vep_k}(\cD)}{p},
$$
which, by \eqref{f: est mtvepk} and \eqref{f: claim limit},
tends to $0$ as $k$ tends to infinity.  Hence, by abstract nonsense, a suitable subsequence
of $m_t^{\vep_k}(\cD) G(\cdot,\vep_k)$ is pointwise convergent a.e. to $0$.  
Next,
$$
\Mt (\cD) \big[G(\cdot,\vep_k)-h\big](x)
= \int_N k_{\Mt (\cD)}(x,y) \, \big[G(\cdot,\vep_k)-h\big](y) \wrt \nu(y),
$$
which tends to $0$ as $k$ tends to infinity, because
$G(\cdot,\vep_k)-h$ is weakly convergent to $0$ in $L^{p}(N)$, and $\bignorm{k_{\Mt (\cD)}(x,\cdot)}{p'}$
is uniformly bounded with respect to $x$ in $N$ by Lemma~\ref{l: Mt sqrt cL}~\rmv\ (and the symmetry of $\Mt (\cD)$).  

This concludes the proof of \rmiii, and of the lemma.
\end{proof}

We shall apply Lemma~\ref{l: Geta} to the case where $F = \mod{\Nabla u}^q$ and $u$ is a harmonic function on $\Si$.   
Suppose that $\be$ is a positive number.  We say that a function $S$ is \textit{$\be$-subharmonic} 
provided that $\DDelta S \geq -\be\, S$ in the sense of distributions.  
The following result, which generalises old ideas of Stein and Weiss 
(see for instance \cite[pp. 217--220]{St1}), is due to Dindo\v s \cite[Section~6.3]{Di}.  

\begin{proposition} \label{p: Dindos}
Suppose that $\Ric_{N\times\BR} \geq -\kappa^2$.  If $\ds (n-1)/{n} \leq q \leq 1$
and $u$ is a harmonic function on an open subset $\Om$ of $N\times\BR$, then $\mod{\Nabla u}^q$
is $q\kappa^2$-subharmonic in $\Om$.  
\end{proposition} 

\begin{theorem} \label{t: analogue SW D}  
Suppose that $u$ is a harmonic function on $\Si$.  Then the following are equivalent:
	\begin{enumerate}
		\item[\itemno1]
			$\mod{\Nabla u}^*$ is in $\lu{N}$; 
		\item[\itemno2]
			$\mod{\Nabla u}$ is in $\cB^1$, and $\ds \lim_{d(x,o) \to\infty} \sup_{t\in [\eta,2\si-\eta]} 
			\,\mod{\Nabla u(x,t)} = 0$ for every $\eta$ in $(0,2\si)$.
	\end{enumerate}
Furthermore, there exists a constant $C$, independent of $u$, such that 
\begin{equation} \label{f: equiv Nabla u}
\bignorm{\mod{\Nabla u}}{\cB^1}
\leq \bignorm{\mod{\Nabla u}^*}{1} 
\leq C \bignorm{\mod{\Nabla u}}{\cB^1}.
	\end{equation} 
\end{theorem}

\begin{proof}
We prove that \rmi\ implies \rmii, and that the left hand inequality in \eqref{f: equiv Nabla u} holds.
Observe that $\bigmod{\Nabla u}^*(x) \geq \bigmod{\Nabla u(x,t)}$
for every $t$ in $(0,2\si)$, so that 
$$
\bignorm{\mod{\Nabla u}}{\cB^1} \leq \int_N \, \bigmod{\Nabla u}^*(x) \wrt\nu(x).
$$  
It remains to prove that $\ds\sup_{t\in [\eta,2\si-\eta]} \bigmod{\Nabla u(\cdot,t)}$ 
vanishes at infinity for each $\eta$ in $(0,2\si)$.  
We argue by contradiction.  Suppose that 
$ 
\ds\limsup_{d(x,o) \to \infty} \,\sup_{t\in [\eta,2\si-\eta]}  \bigmod{\Nabla u(x,t)} =: \be>0
$ 
for some $\eta$ in $(0,2\si)$.  
Then there exists a sequence $\{x_k\}$ such that $d(x_k,o)$ tends to infinity as $k$ does and 
$\ds \limsup_{k \to \infty} \, \sup_{t\in [\eta,2\si-\eta]} \bigmod{\Nabla u(x_k,t)} = \be$.  
Clearly $\bigmod{\Nabla u}^* (x) \geq \be/2$ for all $x$ in $B_{\al \eta}(x_k)$ and $k$ large enough
(here $\al$ denotes the aperture
of the cone $\Ga_\al$ in \eqref{f: nontangential max F}).  By possibly passing to a subsequence, we may assume that
the balls $B_{\al \eta}(x_k)$ are mutually disjoint.  Therefore
$$
\int_N \, \bigmod{\Nabla u}^*(x) \wrt\nu(x) 
\geq \sum_k\, \int_{B_{\al \eta}(x_k)}  \!\bigmod{\Nabla u}^*(x) \wrt\nu(x)   
= \infty,
$$
which clearly contradicts our assumption.  

Next we prove that \rmii\ implies \rmi\ and the right hand inequality in \eqref{f: equiv Nabla u} holds.
Choose $q$ in $\big((n-1)/n,1\big)$.
Since $\bigmod{\Nabla u}$ is in $\cB^1$, $\bigmod{\Nabla u}^q$ is in $\cB^{1/q}$.
Furthermore $\bigmod{\Nabla u}^q$ is $q\kappa^2$-subharmonic in $\Si$ by Proposition~\ref{p: Dindos}, whence 
$G:=\bigmod{\Nabla u}^q - q\kappa^2 \cGS \bigmod{\Nabla u}^q$ is subharmonic therein.  We may apply 
Lemma~\ref{l: Geta}~\rmiii\ (with $\bigmod{\Nabla u}^q$, $q\kappa^2$ and $1/q$ in place of $F$, $\al$ and $p$,
respectively), and conclude that 
\begin{equation} \label{f: conseq limit ineq G}
\mod{\Nabla u}^q
\leq \cPS h + q\kappa^2\cGS \big(\mod{\Nabla u}^q\big).  
\end{equation}
Fix an integer $J>(n+1)/2$.  The inequality \eqref{f: conseq limit ineq G} can be iterated $J$ times, to wit
$$
\mod{\Nabla u}^q 
\leq \cPS h + C\, \Big(\sum_{j=1}^{J-1} \, \cGS^{j}\cPS h  
	+ \cGS^J \big(\mod{\Nabla u}^q\big)\Big). 
$$
We raise both sides of the last inequality to the power $1/q$, and obtain that 
$$
\mod{\Nabla u} 
\leq C\, \Big\{(\cPS h)^{1/q} + \sum_{j=1}^{J-1} \, \big(\cGS^{j} \cPS h\big)^{1/q} 
	+ \big(\cGS^J \mod{\Nabla u}^q\big)^{1/q} \Big\}. 
$$
This implies the following inequality for the associated nontangential maximal functions
\begin{equation} \label{f: nontangential Nabla u}
\mod{\Nabla u}^* 
\leq C\, \Big\{\big[\big(\cPS h\big)^*\big]^{1/q} + \sum_{j=1}^{J-1} \, \big[\big(\cGS^{j} \cPS h\big)^*\big]^{1/q} 
	+ \big[\big(\cGS^J \mod{\Nabla u}^q\big)^*\big]^{1/q} \Big\}. 
\end{equation}
Suppose that $p$ is in $(1,\infty)$.  
Observe that there exists a constant $C$ such that 
\begin{equation} \label{f: max est Poisson}
\bignorm{\big(\cPS h\big)^*}{p}
\leq C \bignorm{h}{p}
\quant h \in \lp{N}.
\end{equation}
Indeed, a straightforward argument, using the subordination formula \eqref{f: Poisson}, shows that 
$$
	\sup_{t\in(0, \si]} \, \bigmod{\cP_t^N f} 
	\leq \sup_{0<t} \, \bigmod{\cH_t^N f}. 
$$
By the Littlewood--Paley--Stein theory \cite[p.~73]{St3} (see also \cite[Theorem~7]{Co}), 
for each $p$ in $(1,\infty]$, there exists a constant $A_p$ such that 
$$
	\bignorm{\, \sup_{0<t} \,\bigmod{\cH_t^N f}}{p}
	\leq A_p \bignorm{f}{p} 
	\quant f \in \lp{N}.   
$$
The subordination formula \eqref{f: Poisson} implies that 
a similar estimate holds for the Poisson maximal operator.  Since $\cPS h(\cdot, t) = \cPS h(\cdot,2\si-t)$,
$$
\sup_{X \in \Ga_\al(z)} \, \bigmod{\cPS h(X)} 
= \sup_{X \in \Ga_\al(z)'} \, \bigmod{\cPS h(X)},
$$ 
where $\Ga_\al(z)' := \big\{(x,t)\in \Ga_\al(z): 0<t<\si\big\}$.  In the case where $X=(x,t)$ is in $\Ga_\al(z)'$,
formula \eqref{f: formula cPS} and the Markovianity of the Poisson semigroup imply that 
\begin{equation} \label{f: formula maximal}
\bigmod{\cPS h(x, t)}
\leq \cP_t^N h(x) + \frac{1}{2} \,\Big[\cP_{\si-t}^N\mod{\Msi \big(\cD\big) h}(x)
+\cP_{t}^N\mod{\e^{-\si\cD}\Msi \big(\cD\big) h}(x)\Big] .
\end{equation}
Clearly, the right hand side of the above inequality is a positive harmonic function on $N\times (0,\infty)$.  
If the aperture $\al$ is small enough and $(x,t)$ is in $\Ga_\al(z)'$, then the ball in $N\times \BR$
with centre $(z,t)$ and radius $2 d(x,z)$ is contained in $N\times (0,\infty)$.  Therefore, by Harnack's principle
(see, for instance, \cite[Theorem~5.4.3]{SC}), there exists a constant $C$ such that 
$$
\cP_t^N h(x) 
\leq C \, \cP_t^N h(z) 
\quant (x,t) \in \Ga_\al(z)' \quant z \in N,  
$$
and a similar estimate holds for the other summands on the right hand side of \eqref{f: formula maximal}.  
Set $h_0 := h$, $h_1 := \bigmod{\Msi \big(\cD\big) h}$ and $h_2 := \bigmod{\e^{-\si\cD} \Msi \big(\cD\big) h}$.    
Then
$$
\bignorm{\big(\cPS h\big)^*}{p}
\leq C\,  \sum_{j=0}^2 \, \bignorm{\, \sup_{0<t} \,\bigmod{\cP_t^N h_j}}{p} 
\leq C\,  \sum_{j=0}^2 \, \bignorm{h_j}{p} 
\leq C \bignorm{h}{p};
$$
the last inequality follows from the boundedness of the operators $\Msi \big(\cD\big)$ and $\e^{-\si\cD} \Msi \big(\cD\big)$
on $\lp{N}$, which follows from Lemma~\ref{l: Mt sqrt cL}~\rmiv\ and the contractivity of the Poisson semigroup.  
This proves \eqref{f: max est Poisson}.  

Similarly, by Theorem~\ref{t: maximal estimate}~\rmi, for every positive integer $j$
there exists a constant $C$ such that  
\begin{equation} \label{f: max est Poisson II}
\bignorm{\big(\cGS^j \cPS h\big)^*}{p}
\leq C \bignorm{\big(\cPS h\big)^*}{p}.
\end{equation}
By combining \eqref{f: max est Poisson} and 
\eqref{f: max est Poisson II} 
we obtain that 
$\bignorm{\big(\cGS^j \cPS h\big)^*}{p}
\leq C \bignorm{h}{p}$
In particular, the last estimate holds for $p=1/q$.  This and \eqref{f: nontangential Nabla u} imply that 
$$
	\begin{aligned}
		\bignorm{\mod{\Nabla u}^*}{1} 
		& \leq C \, \Big\{\bignormto{\big(\cPS h\big)^*}{1/q}{1/q} 
		+ \sum_{j=1}^{J-1} \bignormto{\big(\cGS^{j} \cPS h\big)^*}{1/q}{1/q} 
		+ \bignormto{\big(\cGS^J \mod{\Nabla u}^q\big)^*}{1/q}{1/q} \Big\} \\
		& \leq C\, \Big\{ \bignormto{h}{1/q}{1/q} 
		+ \bignormto{\big(\cGS^J \mod{\Nabla u}^q\big)^*}{1/q}{1/q} \Big\}.
	\end{aligned}
$$
Since $J>(n+1)/2$, Theorem~\ref{t: maximal estimate}~\rmii\ yields
$$
\bignormto{\big(\cGS^J \mod{\Nabla u}^q\big)^*}{1/q}{1/q} 
\leq C \bignormto{\mod{\Nabla u}^q}{L^{1/q}(\Si)}{1/q} 
\leq C \bignorm{\mod{\Nabla u}}{\cB^1}. 
$$
By combining the estimates above and using the estimate for $\bignorm{h}{1/q}$ in 
Lemma~\ref{l: Geta}~\rmi, we get that 
$\bignorm{\mod{\Nabla u}^*}{1} \leq C \bignorm{\mod{\Nabla u}}{\cB^1}$, as required. 
\end{proof}

\section{Analysis of the local Riesz transform}
\label{s: The local Riesz transform}

\subsection{Goldberg-type spaces}
We introduce the \textit{Goldberg-type} space $\ghu{N}$ (also referred to as \textit{local Hardy space}), 
which generalises the Goldberg space $\ghu{\BR^n}$ and plays a fundamental role in our analysis.  

\begin{definition} \label{d: atom}
Fix a positive number $s$.  
Suppose that $p$ is in $(1,\infty]$. 
A \textit{standard $p$-atom} at scale $s$ is a function $a$ in $\lu{N}$ supported in a ball $B$ of radius at most $s$
satisfying the following conditions:
\begin{enumerate}
\item[\itemno1] \textit{size condition}: $\norm{a}{p}  \leq \nu (B)^{-1/p'}$;
\item[\itemno2] \textit{cancellation condition}:
$\ds \int_B a \wrt \nu  = 0$.
\end{enumerate}
A \textit{global $p$-atom} at scale $s$ is a function $a$
in $\lu{N}$ supported in a ball $B$ of radius \textit{exactly equal to} $s$
satisfying the size condition above (but possibly not the cancellation condition).
Standard and global $p$-atoms will be referred to simply as $p$-\textit{atoms}.
\end{definition}

\begin{definition} \label{d: Goldberg}
Suppose that $s$ is a positive number.  The \textit{local atomic Hardy space} $\frh_s^{1,p}({N})$ is the
space of all functions~$f$ in $\lu{N}$ that admit a decomposition of the form
$
f = \sum_{j=1}^\infty \lambda_j \, a_j,
$
where $\la_j \in \BC$, the $a_j$'s are $p$-atoms at scale $s$ and $\sum_{j=1}^\infty \mod{\lambda_j} < \infty$.
The norm $\norm{f}{\frh_s^{1,p}}$ is the infimum of $\sum_{j=1}^\infty \mod{\lambda_j}$
over all decompositions of $f$ as above.  
\end{definition}

\noindent
It is well known that $\frh_s^{1,p}(N)$ is independent of $p$ and of $s$ and the corresponding norms 
$\norm{\cdot}{\frh_s^{1,p}}$ are pairwise equivalent (see \cite[Proposition~1]{MVo}); henceforth, the space
$\frh_s^{1,2}(N)$ will be denoted simply by $\ghu{N}$.  
The fact that $\frh_s^{1,2}(N)$ is independent of $s$ and $p$ will be used without further comment
in the sequel.  Hereafter, atomic decompositions of functions in $\ghu{N}$ will consist of atoms at scale $1$. 

The definition of the space $\ghu{N}$ is similar to that of the atomic Hardy space $\hu{N}$, introduced by
A. Carbonaro, G. Mauceri and Meda \cite{CMM1,CMM2}, the only difference
being that atoms in $\hu{N}$ are just standard atoms in $\ghu{N}$, and there are no global atoms.
As a consequence, the integral of functions in $\hu{N}$ vanishes, a property not enjoyed by all the
functions in $\ghu{N}$.
Thus, trivially, $\hu{N}$ is properly and continuously contained in $\ghu{N}$.

We need the following result, which is one of the main contributions of \cite{MaMV2}.

\begin{theorem}\label{t: MaMV2}
Suppose that $\si>0$. Under our geometric assumptions, $\ghu{N}$ agrees with the space $\ghuP{N}$ of all functions $f$ in 
$L^1(N)$ such that $\ds \cP_\ast^Nf:=\sup_{s\in(0,\si)}|\cP_s^Nf|$ is in $L^1(N)$. 
Furthermore, there exist positive constants $C_1$ and $C_2$ such that
$$
C_1\bignorm{f}{\ghu N}\leq \bignorm{\cP_\ast^Nf}{1} 
\leq C_2\bignorm{f}{\ghu N}.
$$ 
\end{theorem}

We need also the following simple result.  

\begin{proposition} \label{p: ghu decay}
Suppose that $\vep>0$.  Then there exists a constant $C$ such that 
if $h$ is a measurable function on $N$ satisfying $\mod{h(x)} \leq A\, \e^{-(2\be+\vep) d(x,o)}$
for some $o$ in $N$, then $\bignorm{h}{\ghu{N}}\leq C\, A$.
\end{proposition}

\begin{proof}
A corollary of \cite[proof of Lemma 2]{MVo} is that for any $p$ in $(1,\infty]$ there exists a constant~$C$
such that every function $f$ in $\lp{N}$ supported in a ball $B$ is in $\ghu{N}$ and
\begin{equation} \label{f: normah1}
\bignorm{f}{\ghu{N}}
\leq C \, \nu(B)^{1/p'} \bignorm{f}{p}.
\end{equation} 
Consider an exhaustion of $N$ with $B_1(o)$ and annuli $A_j:= B_{j+1}(o) \setminus B_j(o)$, where $j=1,2, \ldots$.
Correspondingly, write 
$$
h
= \One_{B_1(o)} \, h + \sum_{j=1}^\infty \One_{A_j} \, h.  
$$
Clearly $\One_{B_1(o)} \, h$ is in $\ghu{N}$, for it is a multiple of a global
$\ghu{N}$ atom.  Next, by \eqref{f: normah1} (with $p=2$),
$$
\bignorm{\One_{A_j} \, h}{\ghu{N}}
\leq C \, \nu\big(B_{j+1}(o)\big)^{1/2}\bignorm{h}{\ld{A_j}}  
$$
The pointwise estimate of $h$ implies that 
$\bignorm{h}{\ld{A_j}}  \leq CA\, \e^{-(2\be+\vep)j}  \, \nu(A_j)^{1/2}$, whence 
$$
\bignorm{\One_{A_j} \, h}{\ghu{N}}
\leq CA\,  \nu\big(B_{j+1}(o)\big) \,  \e^{-(2\be+\vep)j}  
\leq CA\,  \e^{-\vep j/2}.
$$
The required estimate follows by summing the estimates above with respect to $j$.  
\end{proof}

\subsection{Analysis of the Riesz transform}
For any $\tau>0$ consider the operator $\cD_\tau := \sqrt{\cL_\tau}$, 
obtained by analytic continuation from the analytic family of operators $\{\cL_\tau^{-\al/2}: \Re \al >0\}$.  
We write $\cD_\tau^{-1} = \cKO + \cKinfty$, where $\cKO$ and $\cKinfty$
are the operators associated to the kernels
\begin{equation} \label{f: kernel cKO and cKinfty}
	k_{\cKO} = \vp \, k_{\cD_\tau^{-1}}
	\qquad\hbox{and}\qquad
	k_{\cKinfty} = (1-\vp) \, k_{\cD_\tau^{-1}},  
\end{equation}
where $\vp:N\times N\to [0,1]$\label{pp: vp} is the smoooth function introduced in Lemma \ref{l: Laplacian cut-offs}~\rmii\
(with $R=1$).  We further decompose $\cKO$ as $\cKOO + \cKOinfty$, where 
$$
k_{\cKOO} = \frac{\vp}{\sqrt\pi} \, \int_0^1 t^{-1/2}\, \e^{-\tau t} \, h_t^N \wrt t 
\qquad\hbox{and} \qquad
k_{\cKOinfty} =  \frac{\vp}{\sqrt\pi} \, \int_1^\infty t^{-1/2}\, \e^{-\tau t} \, h_t^N \wrt t.  
$$
Recall the definition of $\Upsilon_R$ (see\eqref{f: Upsilon R}).   

\begin{lemma} \label{l: Riesz Bessel}
There exists a positive constant $C$ such that the following hold:
\begin{enumerate}
\item[\itemno1]
$
\ds k_{\cD_\tau^{-1}}
\leq C\, \big[d^{1-n} \, \One_{\Upsilon_1} +  \e^{-2 d\sqrt{\tau c}}  \, \One_{\Upsilon_1^c} \big],
$
where $c$ is as in \eqref{f: assumptions on ht};  
\item[\itemno2]
	$\cD_\tau^{-1}$ is bounded from $\lu{N}$ to $\ghu{N}$ 
provided that $\tau > \be^2/c$.    
\end{enumerate}
\end{lemma}

\begin{proof}
First we prove \rmi.  
Fr{}om the estimates \eqref{f: assumptions on ht} for $h_t^N$, we deduce that 
$$
k_{\cD_\tau^{-1}}
\leq C\, \int_0^1 t^{(1-n)/2} \, \e^{-\tau t-cd^2/t} {\dtt t}
+ C\, \int_1^\infty \e^{-\tau t-cd^2/t} {\dtt t}.  
$$
Changing variables ($d^2/t=u$) we see that the first integral above is $\asymp d^{1-n}$
as $d$ tends to $0$ and decays superexponentially as $d$ tends to infinity.

Clearly, the second integral above is bounded as $d$ tends to $0$.  For $d$ large
we write $\tau t+cd^2/t = 2d\sqrt{\tau c} + \big(\sqrt{\tau t} - d\sqrt{c/t})^2$, and the
second integral above may be written as $\e^{-2d\sqrt{\tau c}}\, J(d) \asymp d^{-1/2} \, \e^{-2d\sqrt{\tau c}}$, 
where $J(d)$ is as in Lemma~\ref{l: asymp integral} (with $j=1/2$ and $\sqrt{\tau}$ instead of $\lambda_1$).  
The required estimate follows by combining the estimates above.  

To prove \rmii,  
write $g = \sum_j \, g_j$, where $g_j := g \, \One_{B_j}$ and $\{B_j\}$ is a covering of $N$ by balls of radius $1$  
with the finite overlapping property.  Choose $p$ in $\big(1, n/(n-1)\big)$. 
We \textit{claim} that there exist a constant $C$, independent of $j$, and a function $h_j$, 
with support contained in $2B_j$, such that 
\begin{equation} \label{f: lu ghu}
	\bigmod{\cD_\tau^{-1} g_j(x)}
	\leq h_j + C\, \e^{-2\sqrt{\tau c}\, d(x,c_j)}\, \bignorm{g_j}{1}
\end{equation}
and $\bignorm{h_j}{p} \leq C \bignorm{g_j}{1}$.  Here $c_j$ denotes the centre of $B_j$.

Note that $h_j/\big(\nu(2B_j)^{1/p'}\bignorm{h_j}{p}\big)$ is a global $\frh_2^{1,p}(N)$ atom.  
Hence $\bignorm{h_j}{\ghu{N}} \leq \nu(2B_j)^{1/p'}\bignorm{h_j}{p} \leq C \bignorm{g_j}{1}$. 
Notice also that, by Proposition~\ref{p: ghu decay}, each of the functions 
$x\mapsto \e^{-2\sqrt{\tau c}\, d(x,c_j)}$ is in $\ghu{N}$ with norm uniformly bounded with respect to $j$.
Thus, given \eqref{f: lu ghu}, we may conclude that  
$$
	\bignorm{\cD_\tau^{-1} g}{\ghu{N}}
	\leq  \sum_j\, \bignorm{\cD_\tau^{-1} g_j}{\ghu{N}}
	\leq  C\, \sum_j\, \bignorm{g_j}{1}
	\leq C \bignorm{g}{1},  
$$
as required.  

Thus, it remains to prove \eqref{f: lu ghu}.  
Note that, by \rmi,
$$
\bigmod{\cD_\tau^{-1} g_j(x)} 
\leq C\, h_j (x) +C\, \One_{(2B_j)^c} (x) \int_{B_j} \e^{-2\sqrt{\tau c}\, d(x,y)}\, g_j(y) \wrt \nu(y),
$$ 
where $\ds h_j (x) :=  \One_{2B_j} (x)  \int_{B_j} d(x,y)^{1-n} \, g_j(y) \wrt \nu(y)$. 
Since the integral operator with kernel $(x,y) \mapsto \One_{2B_j} (x)  \,  d(x,y)^{1-n} \, \One_{B_j}(y)$ 
is bounded from $\lu{N}$ to $\lp{N}$, the required estimate for $\bignorm{h_j}{p}$ follows.  
By the triangle inequality $d(x,y) \geq d(x,c_j) - d(y,c_j) \geq d(x,c_j)-1$, so that  
$$
\Bigmod{\int_{B_j} \e^{-2\sqrt{\tau c}\, d(x,y)}\, g_j(y) \wrt \nu(y)}
\leq C \, \e^{-2\sqrt{\tau c}\, d(x,c_j)}\,  \bignorm{g_j}{1}.
$$
This concludes the proof of the claim and of \rmii.
\end{proof}

\begin{lemma} \label{l: difference Riesz Bessel}
Suppose that $\tau>\be^2/c$, where $c$ is as in \eqref{f: assumptions on ht}.  
Then there exists a constant~$C$ such that 
$$
\bignorm{\cD_\tau\cKOinfty f}{\ghu{N}}\leq C \bignorm{f}{1} 
\quad\hbox{and}\quad 
\bignorm{\cD_\tau\cKinfty f}{\ghu{N}}\leq C \bignorm{f}{1}
$$
for every function $f$ in $\lu{N}$ with compact support in a ball of radius $\leq 1$.
\end{lemma}

\begin{proof}
We assume that the support of $f$ is contained in $B_R(o)$, with $R\leq 1$.
Since $\cD_\tau \cKOinfty f = \cD_\tau^{-1}\cL_\tau \cKOinfty f$ and 
$\cD_\tau\cKinfty f = \cD_\tau^{-1} \cL_\tau\cKinfty f$, Lemma~\ref{l: Riesz Bessel}~\rmii\ implies that 
it suffices to show that there exists a constant $C$, independent of $f$, such that  
\begin{equation} \label{f: alternative statement}
\bignorm{\cL_\tau \cKOinfty f}{1} \leq C \bignorm{f}{1} 
\quad\hbox{and}\quad 
\bignorm{\cL_\tau\cKinfty f}{1} \leq C \bignorm{f}{1}. 
\end{equation}
The first inequality above will follow from the fact that $\cL_\tau \cKOinfty f$
is a bounded function with support contained in $B_2(o)$ and the estimate $\bignorm{\cL_\tau \cKOinfty f}{\infty}
\leq C \bignorm{f}{1}$.
Since $\cL_\tau \cKOinfty f = \cL \cKOinfty f + \tau \cKOinfty f$,
it suffices to show that both $\cKOinfty f$ and $\cL\cKOinfty f$ are
bounded functions with compact support contained in $B_2(o)$
and the estimates $\bignorm{\cKOinfty f}{\infty}\leq C \bignorm{f}{1}$ and 
$\bignorm{\cL \cKOinfty f}{\infty}\leq C \bignorm{f}{1}$ hold. 
Observe that 
$$
\cKOinfty f(x)
= \frac{1}{\sqrt\pi} 
\, \int_{B_R(o)} \wrt\nu(y) \,\vp(x,y) \, f(y) \, \int_1^\infty t^{-1/2} \, \e^{-\tau t} \, h_t^N(x,y) \wrt t.  
$$ 
By our choice of $\vp$ and the fact that the support of $f$ is contained in $B_1(o)$,
the support of $\cKOinfty f$ is contained in $B_2(o)$.  
The upper estimate \eqref{f: assumptions on ht} for $h_t^N(x,y)$ implies that the inner integral above is dominated by 
a constant independent of $x$ and $y$.  Therefore 
$\bignorm{\cKOinfty f}{1} \leq C\bignorm{f}{1}$, as required. 
We now prove that the same is true of $\cL \cKOinfty f$.    Notice that for each $y$ in $N$
$$
\cL\big[\vp h_t^N\big](\cdot,y) 
= \cL\vp(\cdot,y) \, h_t^N(\cdot,y) - 2\prodo{\nabla \vp(\cdot,y)}{\nabla h_t^N(\cdot,y)} + 
\vp(\cdot,y) \,\cL h_t^N(\cdot,y);
$$
note also that, by our choice of $\vp$, the first and the second summand on the right hand side 
vanish when $d(\cdot,y) \leq 1/4$.  
Correspondingly, $\cL\cKOinfty f$ may be written as $\cB_1 f - \cB_2 f + \cB_3 f$, where
$$
\cB_1 f
= \frac{1}{\sqrt\pi} 
\, \int_{B_R(o)} \!\!\wrt\nu(y) \,\cL\vp(\cdot,y) \, f(y) \, \int_1^\infty t^{-1/2} \, \e^{-\tau t} \, h_t^N(\cdot,y) \wrt t,  
$$
$$
\cB_2 f
= \frac{2}{\sqrt\pi} 
\, \int_{B_R(o)} \!\!\wrt\nu(y) \,\prodo{\nabla\vp(\cdot,y)f(y)}{\int_1^\infty t^{-1/2} \, \e^{-\tau t} 
\, \nabla h_t^N(\cdot,y) \wrt t}  
$$
and 
$$
\cB_3 f
= \frac{1}{\sqrt\pi} 
\, \int_{B_R(o)} \!\!\wrt\nu(y) \,\vp(\cdot,y) \, f(y) \, \int_1^\infty t^{-1/2} \, \e^{-\tau t} 
\, \cL h_t^N(\cdot,y) \wrt t.  
$$ 
We estimate $\cB_3 f$.  
The estimates of $\cB_1 f$ and $\cB_2 f$ are easier, for the kernels of these operators are
supported in $\Upsilon_1\setminus \Upsilon_{1/4}$ (recall also that $\mod{\nabla\vp(\cdot,y)}$ and $\mod{\cL \vp(\cdot,y)}$ 
are uniformly bounded; see Lemma~\ref{l: Laplacian cut-offs}~\rmii), and are left to the interested reader. 
Notice that $\cL h_t^N(\cdot,y) = -\partial_t h_t^N(\cdot,y)$.  Then, by integrating by parts in the inner integral,
we find that 
$$
\cB_3 f
= \frac{1}{2\sqrt\pi} 
\, \int_{B_R(o)} \!\!\wrt\nu(y) \,\vp(\cdot,y) \, f(y) \, \Big[2 \e^{-\tau} \, h_1^N 
- \int_1^\infty t^{-3/2} \, \e^{-\tau t} \, (1+2\tau t) \, h_t^N(\cdot,y) \wrt t\Big].  
$$ 
The upper estimate in \eqref{f: assumptions on ht} and simple considerations show that 
$\bigmod{\cB_3f}$ is dominated by $\ds\int_{B_R(o)} k(\cdot,y) \, \bigmod{f(y)} \wrt\nu(y)$, 
where $k$ is bounded, nonnegative and supported in $\Upsilon_1$ (see \eqref{f: Upsilon R} for the notation).
Consequently, $\bigmod{\cB_3f}$ is a bounded function with support contained in $B_2(o)$ and 
$\bignorm{\cB_3f}{\infty} \leq C \bignorm{f}{1}$.   This concludes the 
proof of the first inequality in \eqref{f: alternative statement}. 

Next we prove the second inequality in \eqref{f: alternative statement}. 
Recall that 
$$
\begin{aligned} 
\cKinfty f
	& = \frac{1}{\sqrt\pi} 
              \, \int_{B_R(o)} \!\!\wrt\nu(y) \,\big[1-\vp(\cdot,y)\big] \, f(y) \, 
	      \ioty t^{-1/2} \, \e^{-\tau t} \, h_t^N(\cdot,y) \wrt t \\
	& = \frac{1}{\sqrt\pi} \, \int_{B_R(o)} \,\big[1-\vp(\cdot,y)\big] \,k_{\cD_\tau^{-1}}(x,y) \, f(y)\wrt\nu(y).  
\end{aligned}
$$
This and the estimates for $k_{\cD_\tau^{-1}}$ in \rmi\ imply that there 
exists a constant $C$ such that $\bigmod{\cKinfty f(x)} \leq C\, \e^{-2 d(x,o)\sqrt{\tau c}} \bignorm{f}{1}$ for every $x$ in $N$.
Thus, $\bignorm{\cKinfty f}{1} \leq C \bignorm{f}{1}$, because, by assumption, $\tau > \be^2/c$.
Since $\cL_\tau \cKinfty f = \cL \cKinfty f + \tau \cKinfty f$,
it remains to show that $\cL \cKinfty f$ satisfies a similar estimate.   
Now,
$$
\cL\big[(1-\vp) h_t^N\big](\cdot,y) 
= -\cL\vp(\cdot,y) \, h_t^N(\cdot,y) + 2\prodo{\nabla \vp(\cdot,y)}{\nabla h_t^N(\cdot,y)} + 
\big[1-\vp(\cdot,y)\big] \,\cL h_t^N(\cdot,y);
$$
note that each of the summands on the right hand side vanishes in $\Upsilon_{1/4}$.
Correspondingly, $\cL\cKinfty f$ may be written as $\cA_1 f + \cA_2 f + \cA_3 f$, where
$$
\cA_1 f
= -\frac{1}{\sqrt\pi} 
\, \int_{B_R(o)} \!\!\wrt\nu(y) \,\cL\vp(\cdot,y) \, f(y) \, \ioty t^{-1/2} \, \e^{-\tau t} \, h_t^N(\cdot,y) \wrt t,  
$$
$$
\cA_2 f
= -\frac{2}{\sqrt\pi} 
\, \int_{B_R(o)} \!\!\wrt\nu(y) \,\prodo{\nabla\vp(\cdot,y)f(y)}{\ioty t^{-1/2} \, \e^{-\tau t} 
\, \nabla h_t^N(\cdot,y) \wrt t}  
$$
and 
$$
\cA_3 f
= \frac{1}{\sqrt\pi} 
\, \int_{B_R(o)} \!\!\wrt\nu(y) \,\big[1-\vp(\cdot,y)\big] \, f(y) \, \ioty t^{-1/2} \, \e^{-\tau t} 
\, \cL h_t^N(\cdot,y) \wrt t.  
$$ 
We estimate $\cA_3 f$.  The estimates of $\cA_1 f$ and $\cA_2 f$ are easier, for the kernel of these operators are
supported in $\Upsilon_1\setminus\Upsilon_{1/4}$ (recall also that $\mod{\nabla\vp(\cdot,y)}$ and $\mod{\cL \vp(\cdot,y)}$ 
are uniformly bounded; see Lemma~\ref{l: Laplacian cut-offs}~\rmii), and are left to the interested reader. 

Notice that $\cL h_t(\cdot,y) = -\partial_t h_t(\cdot,y)$.  Then, by integrating by parts in the inner integral,
we find that 
$$
\cA_3 f
= -\frac{1}{2\sqrt\pi} 
\, \int_{B_R(o)} \!\!\wrt\nu(y) \,\big[1-\vp(\cdot,y)\big] \, f(y) \, \ioty t^{-3/2} \, \e^{-\tau t} \, (1+2\tau t)
\, h_t(\cdot,y) \wrt t.  
$$ 
The inner integral is dominated by 
$$
C \, \int_0^1 t^{-(n+3)/2} \, \e^{-cd^2/t} \wrt t 
+ C \, \int_1^\infty t^{-1} \, \e^{-(\tau t+cd^2/t)} \wrt t.
$$
We need to estimate these integrals in the case where $d$ is large (because of the cutoff $1-\vp$).  
The first is bounded above by 
$$
C \,\e^{-cd^2/2}  \int_0^1 t^{-(n+3)/2} \, \e^{-cd^2/(2t)} \wrt t 
=     C \,\frac{\e^{-cd^2/2}}{d^{n+1}}  \int_{d^2}^\infty u^{(n-1)/2} \, \e^{-cu/2} \wrt u 
\leq  C \,\frac{\e^{-cd^2/2}}{d^{n+1}}, 
$$
and the second by $C\, \e^{-2d\sqrt{\tau c}}$ (see the proof of \rmi), which is integrable at infinity 
because $\tau>\beta^2/c$.  
These estimates imply that $\bignorm{\cA_3 f}{1} \leq C \bignorm{f}{1}$. 
A similar conclusion applies also to $\cA_1 f$
and $\cA_2 f$.  Thus, $\bignorm{\cL \cKinfty f}{1} \leq C \bignorm{f}{1}$, as required to conclude the proof of
the second inequality in \eqref{f: alternative statement}.
\end{proof}

Denote by $k_{\cRtau}$ the distributional kernel of~$\cRtau$, and write $k_{\cRtau}$
as the sum of $\vp \, k_{\cRtau}$ and $(1-\vp) \, k_{\cRtau}$, 
where $\vp$ is the smooth function on $N\times N$ given by Lemma \ref{l: Laplacian cut-offs} (with $R=1$).
Denote by $\cRtauO$ and by $\cRtauinfty$ the operators associated to the kernels 
$\vp \, k_{\cRtau}$ and $(1-\vp) \, k_{\cRtau}$, respectively.   Obviously,
\begin{equation} \label{f: splitting cRtau I}
\cRtau
= \cRtauO + \cRtauinfty.  
\end{equation}
Observe that 
$$
\begin{aligned}
k_{\cRtauO} (x,y) 
& = \frac{\vp (x,y)}{\sqrt \pi} \, \ioty t^{-1/2} \, \e^{-\tau t}\, \nabla_x h_t^N(x,y) \wrt t \,.
\end{aligned}
$$
It is convenient to further decompose the operator $\cRtauO$ as the sum of the operators 
$\cRtauOO$ and $\cRtauOinfty$,  which are associated to the kernels $k_{\cRtauOO}$
and $k_{\cRtauOinfty}$, defined by 
$$
k_{\cRtauOO}(x,y)
= \frac{\vp (x,y)}{\sqrt \pi} \, \int_0^1 t^{-1/2} \, \e^{-\tau t}\, \nabla_x h_t^N(x,y) \wrt t 
$$
and
$$
k_{\cRtauOinfty}(x,y)
= \frac{\vp (x,y)}{\sqrt \pi} \, \int_1^\infty t^{-1/2} \, \e^{-\tau t}\, \nabla_x h_t^N(x,y) \wrt t.  
$$
Notice that $k_{\cRtauOO}(x,y) = \nabla_x k_{\cJOOtau}(x,y) - k_{\cV}(x,y)$, where $\ds
k_{\cV} 
:= \frac{\nabla_x \vp}{\sqrt\pi} \, \int_0^1 t^{-1/2}\, \e^{-\tau t} h_t^N \wrt t. 
$

\begin{lemma}\label{l: pointwise Msi}
For each $\vep>0$ there exists a constant $C$ such that for every $p$ in $N$ and every 
function $f\in L^1(N)$ with support contained in a ball with centre $p$ and radius $\leq 1$ the following hold:
\begin{enumerate}
	\item[\itemno1] $\ds \Big|\Msi (\cD)\cJOOtau f(x)\Big|\leq C\, \e^{(\vep-\la_1) d(x,p) }\bignorm{f}1$;
	\item[\itemno2] $\ds \Big|\cL \Msi (\cD)\cJOOtau f(x)\Big|\leq C\, \e^{(\vep-\la_1) d(x,p) }\bignorm{f}1$. 
\end{enumerate}
Consequently, $\Msi (\cD)\cJOOtau f$ and $\cL \Msi (\cD)\cJOOtau f$ are in 
$\ghu N$ and their norms in $\ghu N$ are controlled by $C\bignorm{f}1$.
\end{lemma}

\begin{proof}
By Proposition~\ref{p: consequences GBE II}~\rmii,  $\Msi (\cD)$ and $\cL \Msi (\cD)$ 
are bounded from $L^1(N)$ to $L^\infty(N)$. Moreover, clearly $\cJOOtau$ is bounded in $L^1(N)$. Therefore, 
\begin{equation}\label{f: Linfty estimate MsicDcJOOtau}
\begin{aligned}
\bignorm{\Msi (\cD)\cJOOtau f}{\infty} 
\leq C\, \bignorm{\cJOOtau f}{1}
\leq  C\,\bignorm{f}{1}
\end{aligned}
\end{equation}
and a similar estimate holds for $\cL \Msi (\cD)\cJOOtau f$.
Furthermore, since $\cJOOtau f$ is supported in $B_{2}(p)$ for some $p\in N$, Lemma \ref{l: Mt sqrt cL}~\rmiii\ gives that
\begin{equation}\label{f: Linfty estimate MsicDcJOOtau off balls}
\begin{aligned}
\big|\Msi (\cD)\cJOOtau f(x)\big| 
& \leq C\,\e^{(\vep-\la_1) d(x,p) }\bignorm{\cJOOtau f}1\\
&\leq  C\,\e^{(\vep-\la_1) d(x,p) }\bignorm{f}{1}\quant x\in B_{4}(p)^c.
\end{aligned}
\end{equation}
Now \rmi\ follows by combining \eqref{f: Linfty estimate MsicDcJOOtau} and 
\eqref{f: Linfty estimate MsicDcJOOtau off balls}.
	
The assertion in \rmii\ follows in a similar way.  
The last statement of the lemma is a direct consequence of \rmi, \rmii\ and Proposition~\ref{p: ghu decay}. 
\end{proof}

\begin{theorem} \label{t: claim}
There exists a constant $C$ such that  
$$
\bignorm{f}{\ghu{N}}
\leq C \, \big[\bignorm{\mod{\cRtauOO f}}{1} + \bignorm{f}{1}\big].
$$
for every function $f$ with support contained in a ball of radius $\leq 1$ for which the right hand side is finite.
\end{theorem}

\begin{proof}
Let $o$ be the centre of the ball of radius $R\leq 1$ which contains the support of~$f$.   
Then the support of $\cJOOtau f$ is contained in a ball of radius $R+1\leq 2$.
Define $H := \cPS (\cJOOtau f)$.   
By Theorem~\ref{t: est nabla cPS}, there exists a constant $C$ such that 
$$
\bignorm{\mod{\Nabla H}}{\cB^1} \leq C \, \cN \big(\cJOOtau f\big).
$$
Furthermore, by Lemma~\ref{l: Mt sqrt cL}~\rmiii, for every $\vep>0$
there exists a constant $C$ such that 
$$
	\sup_{t\in (0,2\si)}\, \mod{\Nabla H (x,t)} 
	\leq C \, \e^{(\vep-\la_1)d(x,o)} \bignorm{\cJOOtau f}{1}
	\leq C \, \e^{(\vep-\la_1)d(x,o)} \bignorm{f}{1}
$$
for every $x$ in $B_{4}(o)^c$.  Hence $\mod{\Nabla H}$ satisfies 
the assumptions of Theorem~\ref{t: analogue SW D}~\rmii, whence there exists a constant such that 
$\bignorm{\mod{\Nabla H}^*}{1} \leq C \bignorm{\mod{\Nabla H}}{\cB^1}$.  By combining the estimates above
we see that 
$$
\bignorm{\mod{\Nabla H}^*}{1}
\leq C \, \cN (\cJOOtau f),  
$$
provided that the right hand side is finite (as we shall prove below). 
The required norm estimate will follow from the following two facts:
	\begin{enumerate}
		\item[(a)] 
			there exists a constant $C$ such that 
			$\cN (\cJOOtau f) \leq C \, \big[\bignorm{\mod{\cRtauOO f}}{1} + \bignorm{f}{1}\big]$; 
		\item[(b)] 
			there exists a constant $C$ such that 
			$\bignorm{f}{\ghu{N}} \leq C \,\big[\bignorm{\mod{\Nabla H}^*}{1}+\bignorm{f}{1}\big]$.  
	\end{enumerate}
First we prove (a). 
Recall that $\cN (\cJOOtau f) := \bignorm{\cJOOtau f}{1} + \bignorm{\mod{\nabla \cJOOtau f}}{1} + \bignorm{\cD \cJOOtau f}{1}$.  
Clearly $\bignorm{\cJOOtau f}{1} \leq C \bignorm{f}{1}$.  Notice that 
$$
\nabla (\cJOOtau f) = \cRtauOO f + \cV (\cJOOtau f),  
$$
where $\cV$ is the operator with kernel 
$\ds
k_{\cV} 
:= \frac{\nabla_x \vp}{\sqrt\pi} \, \int_0^1 t^{-1/2}\, \e^{-\tau t} h_t^N \wrt t. 
$
It is straightforward to check that there exists a constant $C$ such that 
$$
\bignorm{\mod{\cV (\cJOOtau f)}}{1}
\leq C \bignorm{\cJOOtau f}{1}
\leq C \bignorm{f}{1}.
$$
We leave the verification of this fact to the interested reader.  Therefore 
$$
	\bignorm{\mod{\nabla (\cJOOtau f)}}{1}  
	\leq C\, \big[\bignorm{\mod{\cRtauOO f}}{1} + \bignorm{f}{1}\big]. 
$$
The proof of (a) will be complete, once the following \textit{claim} will be proved.  
There exists a constant $C$ such that 
\begin{equation}  \label{f: claim a} 
\bignorm{\cD \cJOOtau f}{1} 
\leq C \bignorm{f}{1}.  
\end{equation}
Write
\begin{equation} \label{f: formula I}
\cD\cJOOtau f = \big[\cD- \cD_\tau\big]\cJOOtau f + \cD_\tau\cJOOtau f.
\end{equation}
The operator $\cD-\cD_\tau$ corresponds to the spectral multiplier $\sqrt{\la}-\sqrt{\tau+\la}$ of $\cL$.
The latter function is in $\cE(\bS_\vp)$ for every $\vp$ in $(0,\pi)$.  Indeed, 
$$
\sqrt{\la}-\sqrt{\tau+\la}
= -\frac{\tau}{\sqrt{\la}+\sqrt{\tau+\la}}
= \tau\, \Big[\frac{1}{\sqrt{\tau+\la}}-\frac{1}{\sqrt{\la}+\sqrt{\tau+\la}}\Big] - \frac{\tau}{\sqrt{\tau+\la}}.
$$
It is straightforward to check that the function within square brackets is in $H_0^\infty(\bS_\vp)$
and $(\tau+\la)^{-1/2}$ is in $\cE(\bS_\vp)$ by \cite[Lemma~2.2.3]{Haa}. 
The operator $\cL$ is sectorial of angle $\pi/2$ on $\ghu{N}$ \cite[Theorem~3.1]{MaMV1}.  
Therefore the natural functional calculus
\cite[Theorem~2.3.3]{Haa} implies that $\cD-\cD_\tau$ is bounded on $\ghu{N}$.  
It is straightforward to check that $\cJOOtau f$ is a function in
$\lp{N}$ for each $p\in[1,n/(n-1))$, and that its support is contained in $B_{2}(o)$. 
Furthermore, 
$$
\bignorm{\cJOOtau f}{\ghu N}\leq C\,\nu \big(B_{2}(o)\big)^{1/p'}\bignorm{f}{1}.
$$
Thus,
$$
\bignorm{[\cD- \cD_\tau]\cJOOtau f}{\ghu{N}} \leq C \bignorm{\cJOOtau f}{\ghu{N}}\leq C\bignorm{f}{1}.
$$
In particular, $\cD\cJOOtau f- \cD_\tau\cJOOtau f$ is in $\lu{N}$.
By \eqref{f: formula I}, $\cD\cJOOtau f$ is in $\lu{N}$ if and only if $\cD_\tau\cJOOtau f$ is.  
Recall that $\cKOO f + \cKOinfty f + \cKinfty f = \cD_\tau^{-1} f$.  Thus,
\begin{equation} \label{f: formula II}
\cD_\tau\cJOOtau f 
=  \cD_\tau\cD_\tau^{-1} f - \cD_\tau\cKOinfty f - \cD_\tau\cKinfty f.
\end{equation}
By Lemma~\ref{l: difference Riesz Bessel}, the second and the third summands on the right hand side are in $\ghu{N}$, 
hence in $\lu{N}$.   Furthermore, $f$ belongs to $\lu{N}$ by assumption,
whence so does $\cD_\tau\cJOOtau f$, equivalently so does $\cD\cJOOtau f$.  Thus, the $L^1$ norm
of each of the summands is dominated by $C \bignorm{f}{1}$.  This implies the claim \eqref{f: claim a},
and concludes the proof of (a).  
	
Next we prove (b).  Notice that 
\begin{equation}  \label{f: partialt H}
\begin{aligned}
\partial_t H 
=- \cP_t^N (\cD \cJOOtau f) +\frac 12  \left[\cP_{\si-t}^N + \cP_{\si+t}^N\right] \cD \Msi (\cD) (\cJOOtau f).
\end{aligned}
\end{equation}
We \textit{claim} that there exists a constant $C$, independent of $f$, such that
\begin{equation} \label{f: second claim}
\bignorm{\cD \Msi (\cD) (\cJOOtau f)}{\ghu N}\leq C \bignorm f1.
\end{equation}
Given the claim, Theorem \ref{t: MaMV2} implies that 
$$
\begin{aligned}
\Bignorm{\sup_{s\in (0,2\si)}\big|\cP_s^N \cD \Msi (\cD) (\cJOOtau f)\big|}{1} 
\leq C\, \bignorm{\cD \Msi (\cD) (\cJOOtau f)}{\ghu N}
\leq C\,  \bignorm{f}{1}.
\end{aligned}
$$
This and \eqref{f: partialt H} imply that   
$$
\begin{aligned}
\Bignorm{\sup_{t\in (0,\si)}\big|\cP_t^N (\cD  \cJOOtau f)\big|}{1} 
& \leq C\, \big[ \bignorm{ |\partial_t H |^\ast}{1}+ \bignorm{f}{1}\big].
\end{aligned}
$$
A further application of Theorem \ref{t: MaMV2} 
and the trivial inequality $\bignorm{\mod{\partial_t H}^*}{1} \leq \bignorm{\mod{\Nabla H}^*}{1}$ yield 
$$
\bignorm{\cD  \cJOOtau f}{\ghu N}
\leq C \bignorm {\cP_\ast^N\cD  \cJOOtau f}1   
\leq C \, \big[\bignorm{\mod{\Nabla H}^*}{1} +  \bignorm{f}{1}\big].
$$
as required.

Thus it remains to prove \eqref{f: second claim}. In the proof of fact (a) above 
we have shown that $\cD-\cD_\tau$ is bounded on $\ghu N$. 
Then Lemma \ref{l: pointwise Msi} implies that 
$$
\bignorm{(\cD-\cD_\tau) \Msi (\cD) (\cJOOtau f)}{\ghu N}
\leq C \bignorm{f}{1}.
$$  
Thus, in order to prove the claim it suffices to prove that 
$\bignorm{\cD_\tau \Msi (\cD) (\cJOOtau f)}{\ghu N} \leq C \bignorm{f}{1}$. Write
$$
\cD_\tau \Msi (\cD) (\cJOOtau f)= \cD_\tau^{-1}[(\cL+\tau) \Msi (\cD) (\cJOOtau f)].
$$
By Lemma \ref{l: pointwise Msi}~\rmi-\rmii, for each $\vep>0$ there exists a constant $C$ such that
$$ 
\Big|(\cL+\tau) \Msi (\cD)\cJOOtau f(x)\Big|\leq C\, \e^{(\vep-\la_1) d(x,o) }\bignorm{f}1
\quant x \in N.
$$
We use the estimate for the kernel of $\cD_\tau^{-1}$ contained in Lemma \ref{l: Riesz Bessel}, and obtain that
$$
\begin{aligned}
\big|\cD_\tau[ \Msi (\cD) (\cJOOtau f)](x)\big|
& \leq C\,\bignorm f1 \Big[\int_{B_1(x)}d(x,y)^{1-n}\e^{(\vep-\la_1)d(y,o)}\wrt\nu(y) \\
& \quad + \int_{B_1(x)^c}\e^{-2d(x,y)\sqrt{\tau c}+(\vep-\la_1)d(y,o)}\wrt\nu(y)\Big]
\end{aligned}
$$
By the triangle inequality, the sum of the last two integrals is dominated by
$$
\e^{(\vep-\la_1)d(x,o)}\Big[\int_{B_1(x)}d(x,y)^{1-n}\e^{(\la_1-\vep)d(y,x)}\wrt\nu(y)
+ \int_{B_1(x)^c}\e^{(\la_1-\vep-2\sqrt{\tau c})d(x,y)}\wrt\nu(y)\Big]
$$
By integrating in polar coordinates centred at $x$, it is straightforward to see that the integral above are convergent, 
provided that $\tau>\la_1^2/(4c)$ and $\vep$ is small enough. Thus, we may conclude that
$$
\big|\cD_\tau[ \Msi (\cD) (\cJOOtau f)](x)\big|
\leq C\,\e^{(\vep-\la_1)d(x,o)}\bignorm f1 \quant x\in N.
$$
Then Proposition \ref{p: ghu decay} yields 
$\bignorm{\cD_\tau \Msi (\cD) (\cJOOtau f)}{\ghu{N}} \leq C \bignorm{f}{1}$. This concludes the proof of 
\eqref{f: second claim} and of the theorem.
\end{proof}

Recall that the \textit{local Riesz--Hardy space} $\ghuRtau{N}$ is defined in \eqref{f: ghuBRn}.
The main result of the paper is the following.  

\begin{theorem} \label{t: char local Riesz}
Suppose that $N$ is an $n$-dimensional complete, connected noncompact Riemannian manifold with 
Ricci curvature bounded from below and positive injectivity radius. 
Assume that $\tau$ is a large positive number.
Then $\ghuRtau{N}=\ghu{N}$ and their norms are equivalent.
\end{theorem}

\noindent 
In particular, the conclusion of Theorem~\ref{t: char local Riesz} holds provided that 
$\tau>\lambda_1^2/4c$ (this implies that $\tau>\be^2/c$, 
where $\beta$ and $c$ are as in \eqref{f: volume growth} and \eqref{f: assumptions on ht}), 
and is so large that Proposition \ref{p: local Riesz transform} holds.

\begin{remark}
We observe that the claim of Theorem \ref{t: char local Riesz} is invariant under rescaling of the Riemannian 
metric by a constant conformal factor, since the spaces $\ghuRtau{N}$ and $\ghu{N}$ are 
invariant, and their norms rescale by the same factor. Accordingly, instead of choosing $\si$ 
small enough depending on $\be$ (see \eqref{f: sigma}), one could have fixed $\si$ 
and rescaled the Riemannian metric of $N$ in order to make $\be$ small enough.  
\end{remark}

\noindent
The proof of Theorem~\ref{t: char local Riesz} occupies the rest of this section.
First we analyse the kernel of $\cRtauOO$. 

\begin{lemma} \label{l: char local Riesz II}
Under the same assumption as in Theorem~\ref{t: char local Riesz}, 
there exists a constant $C$ such that 
$$
\bigmod{k_{\cRtauOO}(x,y)} 
\leq C \, \vp(x,y) \, d(x,y)^{-n}
$$
off the diagonal.  
\end{lemma}

\begin{proof}
The proof is a straightforward consequence of the definition of $\cRtauOO$ and the pointwise
estimate \cite[Theorem~6, Case II]{Da2} for the gradient of the heat kernel on $N$.    
We leave the details to the interested reader.  
\end{proof}

\noindent
Denote by $\{\psi_j\}$ a \textit{locally uniformly finite partition of unity} on $N$ such that 
the following holds:
the support of $\psi_j$ is contained in the ball $B_j$ with radius $1$,  
$0\leq \psi_j \leq 1$, $\psi_j = 1$ on $(1/4) B_j$, and there exists a constant $C$,
independent of~$j$, such that 
\begin{equation} \label{f: uniform Lipschitz}
\bigmod{\psi_j(x) - \psi_j(y)}
\leq C \, d(x,y)
\quant x,y \in N.  
\end{equation}  
For the construction of such a partition of unity see, for instance, \cite[Lemma~1.1 and pp.~59--60]{He}. 
We recall the following norm estimate for the local Riesz 
transform on $N$, due to E.~Russ \cite[proof of Theorem 14]{Ru}; see also \cite[Theorem 8]{MVo}.

\begin{proposition} \label{p: local Riesz transform}
For every $\tau>0$ large enough there exists a constant $C$ such that 
$
\bignorm{\mod{\cRtau f}}{1} \leq C \bignorm{f}{\ghu{N}}
$
for every $f$ in $\ghu{N}$.
\end{proposition}

\begin{lemma} \label{l: char local Riesz I}
Under the same assumption as in Theorem~\ref{t: char local Riesz}, the following hold:
	\begin{enumerate}
		\item[\itemno1]
			the operator $\cRtauinfty$ is bounded on $\lu{N}$;
		\item[\itemno2]
			the operator $\cRtauOinfty$ is bounded on $\lu{N}$.  
		\item[\itemno3]
			if $f$ is in $\ghuRtau{N}$, then $\bigmod{\cRtauOO f}$ is in $\lu{N}$;
		\item[\itemno4]
			for each $p$ such that $1\leq p< n/(n-1)$ there exists a constant $C$, independent of $j$, such that
			if $f$ is in $\ghuRtau{N}$, then 
			$$
			\bignorm{\mod{\cRtauOO(\psi_jf) - \psi_j\, \cRtauOO f}}{p}
			\leq C \bignorm{f}{\lu{2B_j}};
			$$
		\item[\itemno5]
			there exists a constant $C$, independent of $j$, such that 
			$$
			\bignorm{\mod{\cRtauOO(\psi_jf)}}{\lu{2B_j}}
			\leq C \, \big[\bignorm{f}{\lu{2B_j}} + \bignorm{\mod{\cRtauOO f}}{\lu{2B_j}} \big].  
			$$
	\end{enumerate}
\end{lemma}

\begin{proof}
First we prove \rmi.  By \cite[Theorem~14]{Ru},  $\ds\sup_{y\in N} \int_N  \, \bigmod{k_{\cRtauinfty}(x,y)} \wrt \nu(x) 
< \infty$.  Consequently the operator $\cRtauinfty$ is bounded on $\lu{N}$, as required.  

To prove \rmii\ observe that, at least formally,
$$
\cRtauOinfty f (x)  
= \frac{1}{\sqrt \pi} \, \int_1^\infty \!\! \frac{\wrt t}{t^{1/2}} \, \e^{-\tau t}\, 
   \int_N  \vp (x,y) \, \nabla_x h_t^N(x,y) \, f(y) \wrt \nu(y).  
$$ 
Therefore 
$$
\begin{aligned}
\bignorm{\mod{\cRtauOinfty f}}{1}
& \leq \frac{1}{\sqrt \pi} \, \int_1^\infty \!\! \frac{\wrt t}{t^{1/2}} \, \e^{-\tau t}\, 
   \int_N  \!\!\wrt \nu(y)\, \bigmod{f(y)} \int_N  \vp (x,y) \, \bigmod{\nabla_x h_t^N(x,y)} \wrt \nu(x) \\
& \leq \frac{1}{\sqrt \pi} \, \int_1^\infty \!\! \frac{\wrt t}{t^{1/2}} \, \e^{-\tau t}\, 
   \int_N  \!\!\wrt \nu(y)\, \bigmod{f(y)} \int_{B_{1}(y)}\bigmod{\nabla_x h_t^N(x,y)} \wrt \nu(x) \\
& \leq C\, \int_1^\infty \!\! \frac{\wrt t}{t^{1/2}} \, \e^{-\tau t}\, 
   \int_N  \bigmod{f(y)}  \bignorm{\mod{\nabla_x h_t^N(\cdot,y)}}{2} \wrt \nu(y),
\end{aligned}
$$ 
where the last inequality follows from Schwarz's inequality and the uniform ball size condition 
of $N$.  Observe that 
$$
\begin{aligned}
\bignormto{\mod{\nabla_x h_t^N(\cdot,y)}}{2}{2}
& =     \big(\nabla_x h_t^N(\cdot,y), \nabla_x h_t^N(\cdot,y) \big) \\ 
& =     \big(\cL_x h_t^N(\cdot,y), h_t^N(\cdot,y) \big) \\ 
& \leq  \bignorm{\cL_x h_t^N(\cdot,y)}{2} \,  \bignorm{h_t^N(\cdot,y)}{2}.
\end{aligned}
$$
Now, the ultracontractivity of the heat semigroup and \cite[Proposition~2.2]{MMV0} 
imply that the supremum with respect to $y$ in $N$ of the right hand side is dominated 
by a constant multiple of $t^{-3/2}$.  Therefore we may conclude that 
$$
\bignorm{\mod{\cRtauOinfty f}}{1}
\leq C\,  \int_1^\infty \!\! \frac{\wrt t}{t^{5/4}} \, \e^{-\tau t}\, \int_N  \bigmod{f(y)} \wrt \nu(y) 
\leq C \bignorm{f}{1},
$$
i.e., the operator $\cRtauOinfty$ is bounded on $\lu{N}$, as required.  

Next we prove \rmiii.  The assumption that $f$ is in $\ghuRtau{N}$ 
together with the decomposition \eqref{f: splitting cRtau I} and \rmi\ above yields that $\bigmod{\cRtauO f}$ in $\lu{N}$.   
Since $\cRtauOO f = \cRtauO f - \cRtauOinfty f$ and both $\bigmod{\cRtauO f}$ and $\bigmod{\cRtauOinfty f}$ 
are in $\lu{N}$ (by \rmi\ and \rmii), the same is true of $\bigmod{\cRtauOO f}$, as required.

To prove \rmiv\  observe that, at least formally, 
$$
\cRtauOO(\psi_jf)(x) - \psi_j(x)\, \cRtauOO(f) (x)
= \int_N \, k_{\cRtauOO}(x,y) \, \big[ \psi_j(y)-\psi_j(x)\big] \, f(y) \wrt \nu(y).
$$
If $x$ is not in $2B_j$, then $\psi_j(x)$ vanishes, and so does $k_{\cRtauOO}(x,y)$ as long as $y$
belongs to $B_j$.  Hence $\cRtauOO(\psi_jf) - \psi_j\, \cRtauOO f$ vanishes at $x$. In particular
$$
\bignorm{\cRtauOO(\psi_jf) - \psi_j\, \cRtauOO f}{p}=\bignorm{\cRtauOO(\psi_jf) - \psi_j\, \cRtauOO f}{\lp{2B_j}}
$$

If, instead, $x$ is in $2B_j$, then we use the estimates for $k_{\cRtauOO}$ in Lemma~\ref{l: char local Riesz II},
the uniform Lipschitz property of $\psi_j$, and 
conclude that there exists a constant $C$, independent of $j$, such that  
$$
\bigmod{\cRtauOO(\psi_jf)(x) - \psi_j(x)\, \cRtauOO f (x)}
\leq  C\, \int_{2B_j}  \, \frac{\vp(x,y)}{d(x,y)^{n-1}} \, \bigmod{f(y)} \wrt \nu(y).
$$
It is not hard to check that if $1/p > 1-1/n$, then the integral operator with kernel 
$\ds \One_{2B_j}(x) \, \frac{\vp(x,y)}{d(x,y)^{n-1}}\, \One_{2B_j}(y)$ is bounded from 
$\lu{2B_j}$ to $\lp{2B_j}$ uniformly in $j$.  Therefore there exists a constant $C$ such that  
$$
\bignorm{\cRtauOO(\psi_jf) - \psi_j\, \cRtauOO f}{\lp{2B_j}}
\leq C \bignorm{f}{\lu{2B_j}},  
$$
as required to conclude the proof of \rmiv.  

Now we prove \rmv.  
Clearly 
$
\cRtauOO(\psi_jf) = \cRtauOO(\psi_jf) - \psi_j\cRtauOO f + \psi_j \cRtauOO f. 
$ 
By \rmiv, the function $\cRtauOO(\psi_jf) - \psi_j\cRtauOO(f)$ is in $\lp{2B_j}$
with norm $\leq C \bignorm{f}{\lu{2B_j}}$.  H\"older's inequality, together with local Ahlfors regularity, imply that 
$$
\bignorm{\cRtauOO(\psi_jf) - \psi_j\cRtauOO f}{\lu{2B_j}}
\leq C \bignorm{f}{\lu{2B_j}}.  
$$
Therefore 
$$
	\begin{aligned}
	\bignorm{\cRtauOO(\psi_jf)}{\lu{2B_j}}
		& \leq C \, \big[\bignorm{\cRtauOO(\psi_jf) - \psi_j\cRtauOO f}{\lu{2B_j}} 
		+ \bignorm{\cRtauOO(f)}{\lu{2B_j}} \big] \\
		& \leq C \, \big[\bignorm{f}{\lu{2B_j}} + \bignorm{\cRtauOO(f)}{\lu{2B_j}} \big],  
	\end{aligned}
$$
as required to conclude the proof of \rmv, and of the lemma.
\end{proof}

\begin{proof}[Proof of Theorem~\ref{t: char local Riesz}] 
The containment $\ghu{N}\subseteq \ghuRtau{N}$ is a direct consequence of Proposition \ref{p: local Riesz transform}.  

It remains to prove that $\ghuRtau{N}\subseteq \ghu{N}$.  
Suppose that $f$ is in $\ghuRtau{N}$.  By Lemma~\ref{l: char local Riesz I}~\rmiii, $\bigmod{\cRtauOO f}$
is in $\lu{N}$.  Then, by Lemma~\ref{l: char local Riesz I}~\rmv,
$$
\bignorm{\mod{\cRtauOO(\psi_jf)}}{\lu{2B_j}}
\leq C \, \big[\bignorm{f}{\lu{2B_j}} + \bignorm{\mod{\cRtauOO f}}{\lu{2B_j}} \big].  
$$
By Theorem~\ref{t: claim}, there exists a constant $C$, independent of $j$, such that 
$$
\bignorm{\psi_jf}{\ghu{N}}
\leq C \bignorm{\mod{\cRtauOO(\psi_jf)}}{1} + C \bignorm{\psi_jf}{1}.
$$
Then, using also Lemma~\ref{l: char local Riesz I}~\rmi,~\rmii~and~\rmiv,
$$
\begin{aligned}
\bignorm{f}{\ghu{N}} 
& \leq \sum_j \bignorm{\psi_jf}{\ghu{N}} \\
& \leq C \, \sum_j \bignorm{\mod{\cRtauOO(\psi_jf)}}{1} + C \, \sum_j \bignorm{\psi_jf}{1}\\ 
& \leq C \, \sum_j \bignorm{\psi_j\mod{\cRtauOO f}}{1} +
	C \, \sum_j \, \Big[\bignorm{\mod{\cRtauOO(\psi_jf)- \psi_j\cRtauOO f}}{1} + \bignorm{\psi_jf}{1}\Big]\\  
& \leq C  \bignorm{\mod{\cRtauOO f}}{1} + C \bignorm{f}{1}\\  
& \leq C  \, \big[\bignorm{\mod{\cRtau f}}{1} + \bignorm{\mod{\cRtau^{0,\infty} f}}{1}
	+\bignorm{\mod{\cRtau^\infty f}}{1}+  \bignorm{f}{1} \big]\\  
& \leq C  \, \big[\norm{\mod{\cRtau f}}{1} + \norm{f}{1}\big],
\end{aligned}
$$
as required.  
\end{proof}

\thanks{\textbf{Acknowledgements}.  The authors would like to thank Alessio Martini, Alberto Setti and Maria Vallarino for 
valuable conversations on the subject of this paper.}

\end{document}